\documentclass[a4,11pt]{amsart}
\usepackage[top=3cm, bottom=3cm, left=2.5cm, right=2.5cm]{geometry}

\usepackage{amssymb,amsmath,amsthm,txfonts, amsfonts}
\usepackage{amscd}
\usepackage[all]{xy} 
\usepackage[colorlinks=true,citecolor=blue,linkcolor=blue]{hyperref}
\usepackage{graphicx}
\usepackage{epstopdf}
\usepackage{bookmark}

\usepackage{newtxtext}
\usepackage{newtxmath}

\usepackage{mathrsfs}
\usepackage{color}
\usepackage{cite}
\usepackage{nicefrac}
\usepackage{hyperref}
\usepackage{textcomp}
\usepackage{protosem}

\usepackage[hang,small,bf]{caption}
\usepackage[subrefformat=parens]{subcaption}
\newtheorem{thm}{Theorem}[section]
\newtheorem{lem}[thm]{Lemma}

\newtheorem{prop}[thm]{Proposition}
\theoremstyle{definition}
\newtheorem{defn}[thm]{Definition}
\theoremstyle{remark}
\newtheorem{rem}[thm]{\textbf{Remark}}
\newtheorem{rems}[thm]{\textbf{Remarks}}
\newtheorem{q}[thm]{\textbf{Question}}

 \makeatletter
    
    \@addtoreset{equation}{section}
  \makeatother

\makeatletter
      \def\@makefnmark{%
         \leavevmode
            \raise.9ex\hbox{\check@mathfonts
                \fontsize\sf@size\z@\normalfont%
                            \@thefnmark}%
       }
      \makeatother

\makeatletter
\@namedef{subjclassname@2020}{\textup{2020} Mathematics Subject Classification}
\makeatother

\newcommand{\nrm}[1]{\Vert#1\Vert}
\newcommand{\omg}{\omega}
\newcommand{\tht}{\theta}
\newcommand{\bt}{\beta}
\newcommand{\dlt}{\delta}
\newcommand{\bbR}{\mathbb R}

\begin{document}

\title{Homogeneous steady states for the generalized surface quasi-geostrophic equations}
\author{Ken Abe, Javier G\'{o}mez-Serrano, In-Jee Jeong}  
\address[Ken Abe]{Department of Mathematics, Graduate School of Science, Osaka Metropolitan University, 3-3-138 Sugimoto, Sumiyoshi-ku Osaka, 558-8585, Japan}
\email{kabe@omu.ac.jp}
\address[Javier G\'{o}mez-Serrano]{Department of Mathematics
Brown University
314 Kassar House, 151 Thayer St. Providence, RI 02912, USA \textit{and} Institute for Advanced Study, 1 Einstein Drive, Princeton, NJ 08540, USA}
\email{javier\_gomez\_serrano@brown.edu}
\address[In-Jee Jeong]{Department of Mathematical Sciences \textit{and} RIM, Seoul National University, Seoul 08826, Korea}
\email{injee\_j@snu.ac.kr}

\date{\today}

\begin{abstract}
We consider homogeneous (stationary self-similar) solutions to the generalized surface quasi-geostrophic (gSQG) equations parametrized by the constant $0<s<1$, representing the 2D Euler equations ($s=1$), the SQG equations $(s=1/2)$, and stationary equations ($s=0$); namely, solutions whose stream function $\psi$ and advected scalar $\omega$ are of the form 
\begin{align*}
\psi=\frac{w(\theta)}{r^{\beta}},\quad \omega=\frac{g(\theta)}{r^{\beta+2s}},
\end{align*}
in polar coordinates $(r,\theta)$ with parameter $\beta\in \mathbb{R}$. 
We classify homogeneous steady states across the full parameter space, and we identify the limiting singular regimes assuming an odd symmetric profile $(w,g)$ with Fourier modes larger than $m_0\geq 1$. Specifically, we show existence of such solutions for $-m_0-2s<\beta<-2s$ and $0<\beta<m_0+2$ ($1/2-s<\beta< m_0+2$ for $0<s<1/2$) and non-existence of such solutions for $-2s\leq \beta\leq 0$. The main result provides examples of self-similar solutions which belong to critical and supercritical regimes for the local well-posedness of the gSQG equations for $0<s<1$ and the first examples of self-similar solutions for the SQG equations and the more singular equations $0<s\leq 1/2$ in the stationary setting.  We also complement our findings with a numerical illustration of the solutions.
\end{abstract}

\maketitle

\tableofcontents

\section{Introduction}
	The surface quasi-geostrophic (SQG) equations are the fundamental model for the temperature/potential vorticity on a fluid surface under the influence of the Coriolis force. Besides its significance in geophysics, the SQG equations serve as a two-dimensional model for the vorticity evolution of the 3D Euler equations; see Constantin et al. \cite{CMT1,CMT94}. The works \cite{CFMR,CCCGW} posed the generalized surface quasi-geostrophic (gSQG) equations for $0<s<1$,

	\begin{equation}\label{eq:gSQG}
		\begin{aligned}
			\omega_t+u\cdot \nabla \omega=0,\quad 
			u=-\nabla^{\perp}(-\Delta)^{-s}\omega,
		\end{aligned}
	\end{equation}\\
	interpolating the 2D Euler equations ($s=1$) and the SQG equations ($s=1/2$), and extrapolating to the stationary equation ($s=0$), where $\nabla^{\perp}={}^{t}(-\partial_2,\partial_1)$ and the stream function $\psi$ is given by
	
	\begin{align*}
		\psi=(-\Delta)^{-s}\omega=C(s)\frac{1}{|x|^{2-2s}}*\omega,\quad  C(s)=\frac{\Gamma(1-s)}{4^{s}\pi \Gamma(s)},
	\end{align*}\\
	where $\Gamma(s)$ is the Gamma function. Further generalizations were studied in \cite{CCW3,OH1,CJO1,CJNO,MTXX,Kwon,ChWu}. In contrast to the 2D Euler equations, some gSQG solutions for $1/2\leq s<1$ develop finite time singularities in the upper half-plane $\mathbb{R}^2_+$ with vorticity $\omega$ non-vanishing on the boundary \cite{KRYZ,GP21,Zlatos23,MTXX}. 
    Without a boundary, it is still open whether singularities can be formed in finite time, for any $0<s<1$; see \cite{Scott,SD19} for some numerical evidence for singularities. We shall discuss these issues more below. 
	
\subsection{Self-similar solutions}
	In the recent work \cite{GGS24}, Garc\'{i}a and G\'{o}mez-Serrano constructed forward self-similar solutions to the gSQG equations exhibiting spiral streamlines, for initial data of the form 
	
\begin{equation*}
		\omega(x,0)=\frac{g_{*}(\theta)}{r^{\beta+2s}}\quad\mbox{with} \quad \frac{1}{2}<s<1,
	\end{equation*}\\
	for $0<\beta<2-2s$ and any $2\pi/m$-periodic function $g_{*}(\theta)$ for $m\geq 1$ close to the radially symmetric homogeneous solutions 
	
	\begin{align*}
		\bar\psi(x)=\frac{\bar{w}}{r^{\beta}},\quad \bar\omega(x)=\frac{1}{r^{\beta+2s}},
	\end{align*}\\
	where $(r,\theta)$ denote the polar coordinates and $\bar{w}$ is a constant. The condition $0<\beta<2-2s$ stems from the integrability of the radially symmetric $\bar\omega$ by the operator $(-\Delta)^{-s}$. Here, we say that $(\psi,\omega)$ is a \textit{forward self-similar} solution to \eqref{eq:gSQG} if there exists $\beta\in \mathbb{R}\backslash\{-2\}$ such that  
	
	\begin{equation}\label{eq:FSS}
	\psi(x,t)=\lambda^{\beta}\psi(\lambda x, \lambda^{\beta+2}t),\quad \omega(x,t)=\lambda^{\beta+2s} \omega(\lambda x, \lambda^{\beta+2}t),\quad x\in \mathbb{R}^{2},\ t>0,\ \lambda>0.
	\end{equation}\\
	We say that $(\psi,\omega)$ is a \textit{scale-invariant} solution to \eqref{eq:gSQG} if this scaling law holds for $\beta=-2$ \cite{EJ20b,DE,CCZ21}. We say that $(\psi,\omega)$ is a \textit{homogeneous} (stationary self-similar) solution if a forward self-similar/scale-invariant solution $(\psi,\omega)$ is time-independent; namely, if it is of the form  
	
\begin{align}\label{eq:HS}
\psi(x)=\frac{w(\theta)}{r^{\beta}},\quad \omega(x)=\frac{g(\theta)}{r^{\beta+2s}}.
\end{align}\\	
	Motivated by the close connection between two kinds (non-stationary and stationary) of self-similar solutions, we discuss the following question: 
	
	\begin{q}\label{q:Q}
		What kinds of non-trivial homogeneous solutions \eqref{eq:HS} exist?  
	\end{q}
	
	To the best of our knowledge, there were no examples of non-trivial self-similar solutions for the SQG case ($s=1/2$) and more generally for $0<s<1/2$ prior to this work, \cite{GGS24} being the only work covering the range $1/2<s<1$.
	
	\subsection{The statement of the main result}
	
	In this paper, we explore Question \ref{q:Q}, assuming odd symmetry for profile functions
	
	\begin{align}\label{eq:Odd}
		w(\theta)=-w(-\theta),\quad g(\theta)=-g(-\theta),\quad \theta\in \mathbb{T}=[0,2\pi].
	\end{align}\\
	This is equivalent to seeking solutions that can be represented by a sine series:
	
	\begin{align}\label{eq:Sineseri}
w(\theta)=\sum_{m=m_0}^{\infty}\left(\frac{1}{\pi}\int_{0}^{2\pi} w(q)\sin(mq)d q\right)\sin(m\theta),\quad g(\theta)=\sum_{m=m_0}^{\infty}\left(\frac{1}{\pi}\int_{0}^{2\pi} g(q)\sin(mq)d q\right)\sin(m\theta),\quad m_0\geq 1.
	\end{align}\\
This odd symmetry assumption allows us to avoid solutions close to the radial functions, and is relevant to existing infinite and finite-time singularity results \cite{KS,Den15a,Den15b,EH,KRYZ,KYZ17,GP21,Zlatos23,MTXX,JKY25}. At a technical level, this symmetry reduces the problem to the upper half-plane $\mathbb{R}^2_+$ and introduces an extra cancellation in the kernel for $\psi$, allowing us to get existence in a larger range of $\beta$.
	
	More specifically, we consider the stationary equations of the form, 
	
	\begin{align}\label{eq:SgSQG}
		\nabla^{\perp}\psi\cdot \nabla \omega=0,\quad (-\Delta)^{s}\psi=\omega.
	\end{align}\\
	The fractional Laplacian restricts the homogeneity to $-2s<\beta<0$ in differentiating radially symmetric homogeneous stream functions. (Integrating the radially symmetric homogeneous vorticity $\psi=(-\Delta)^{-s}\omega$ imposes the condition $0<\beta<2-2s$). We show, in contrast to the radially symmetric homogeneous solutions, the non-existence of odd symmetric homogeneous solutions in the range $-2s\leq \beta\leq 0$ and the existence of odd symmetric homogeneous solutions in the complementary range $\beta<-2s$ and $0<\beta$. The main result of this paper is as follows: \\

\begin{thm}\label{t:thm}
Let $0< s< 1$, $m_0\in \mathbb{N}$, and $-m_0-2s<\beta< m_0+2$. The following holds for odd symmetric homogeneous solutions \eqref{eq:HS} and \eqref{eq:Sineseri} for the stationary gSQG equations \eqref{eq:SgSQG} with profiles $(w,g)$.\\

\begin{itemize}
\item[(A)] \textbf{Between SQG and Euler}; $1/2< s< 1$\\
\textbf{(Non-existence)}\\
(i) For $-2s<\beta \leq 0$, there exist no solutions $(w,g)\in C^{1}(\mathbb{T})$ except for $w=C\sin\theta$ and $g=0$ with a constant $C$.\\
(ii) For $\beta=-2s$, there exist no solutions $(w,g)\in C^{2}(\mathbb{T})$.\\
\noindent
\textbf{(Existence)} \\
For $0<\beta<m_0+2$ and $-m_0-2s<\beta<-2s$, there exists a solution $(w,g)\in C^{2s+1+\frac{2s}{\beta}-\varepsilon}(\mathbb{T})\times C^{1+\frac{2s}{\beta}-\varepsilon}(\mathbb{T})$ for arbitrary small $\varepsilon>0$. If $2s+1+2s/\beta\notin \mathbb{N}$ and $1+2s/\beta\notin \mathbb{N}$, $\varepsilon=0$. 
\item[(B)] \textbf{SQG}; $s=1/2$\\
\textbf{(Non-existence)}\\
(i) For $-1<\beta\leq 0$, there exist no solutions $(w,g)\in C^{1}(\mathbb{T})$.\\
(ii) For $\beta=-1$, there exist no solutions $(w,g)\in C^{2}(\mathbb{T})$ except for $w=C\sin\theta$ and $g=0$ with a constant $C$.\\
\textbf{(Existence)} \\
For $0<\beta< m_0+2$ and $-m_0-1<\beta<-1$, there exists a solution $(w,g)\in C^{2+\frac{1}{\beta}-\varepsilon}(\mathbb{T})\times C^{1+\frac{1}{\beta}-\varepsilon}(\mathbb{T})$ for arbitrary small $\varepsilon>0$. If $1/\beta\notin \mathbb{N}$, $\varepsilon=0$.
\item[(C)] \textbf{More singular case}; $0< s< 1/2$\\
\noindent
\textbf{(Non-existence)}\\
(i) For $-2s< \beta \leq 0$, there exist no solutions $(w,g)\in C^{1}(\mathbb{T})$.\\
(ii) For $\beta=-2s$, there exist no solutions $(w,g)\in C^{2}(\mathbb{T})$.\\
\textbf{(Existence)} \\
(i) For $1/2-s<\beta< m_0+2$ and $-m_0-2s<\beta<-1$, there exists a solution $(w,g)\in C^{2s+1+\frac{2s}{\beta}-\varepsilon}(\mathbb{T})\times C^{1+\frac{2s}{\beta}-\varepsilon}(\mathbb{T})$ for arbitrary small $\varepsilon>0$. If $2s+1+2s/\beta\notin \mathbb{N}$ and $1+2s/\beta\notin \mathbb{N}$, $\varepsilon=0$.\\ 
(ii) For $-1\leq \beta <-2s$, there exists a solution $(w,g)\in C^{-\beta-\varepsilon}(\mathbb{T})\times C^{-\beta-2s-\varepsilon}(\mathbb{T})$ for any small $\varepsilon>0$.\\
\end{itemize}
\end{thm}

See Figure \ref{fig:1} for the non-existence range and the existence range of $\beta$.

\begin{figure}[h]
  \begin{minipage}[b]{0.45\linewidth}
\hspace{-118pt}
\includegraphics[scale=0.195]{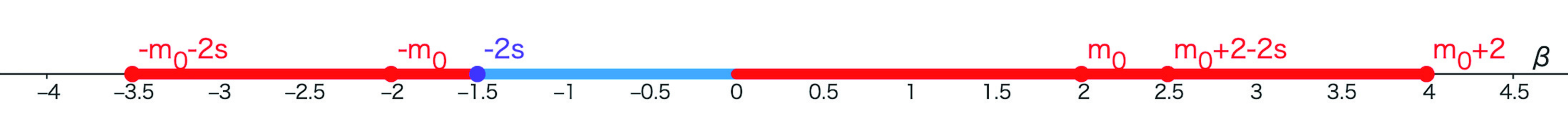}
\subcaption{$1/2< s< 1$ (Case $s=3/4$)}
  \end{minipage}\\
    \begin{minipage}[b]{0.45\linewidth}
\hspace{-118pt}
\includegraphics[scale=0.195]{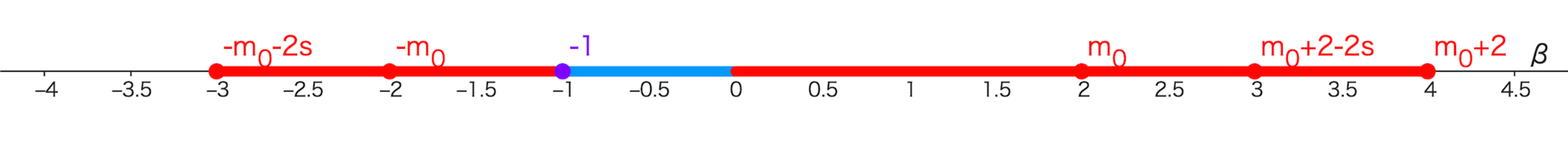}
\subcaption{$s=1/2$}
  \end{minipage}\\
  \vspace{-3pt}
    \begin{minipage}[b]{0.45\linewidth}
\hspace{-118pt}
\includegraphics[scale=0.195]{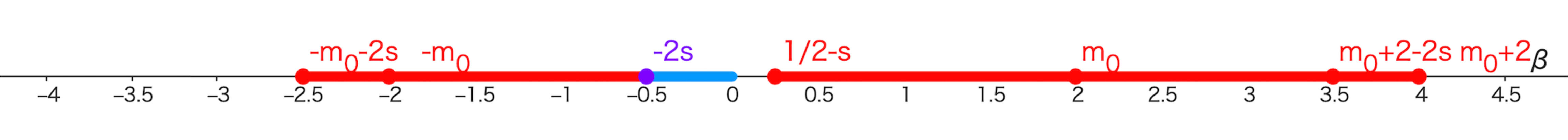}
\subcaption{$0<s< 1/2$ (Case $s=1/4$)}
  \end{minipage}
  \caption{The existence range (red) and non-existence range (blue) of $\beta$ in Theorem \ref{t:thm}. Case $m_0=2$.}\label{fig:1}
\end{figure}

\vspace{5pt} 

	A key observation in Theorem \ref{t:thm} is the relationship between the homogeneous degree $\beta$ and the frequency parameter $m_0$; namely, that large $|\beta|$ homogeneous solutions lie on high-frequency spaces for the angular variable. Indeed, the irrotational solutions to \eqref{eq:SgSQG} 
	
	\begin{align*}
		w=\sin(\beta\theta),\quad g=0,\quad \beta\in \mathbb{Z},
	\end{align*}\\
	exemplify that the degree plays a role in frequencies: Figure \ref{fig:2} shows oscillations of streamlines of irrotational solutions for large $|\beta|$. For the Euler equations, large $|\beta|$ solutions exist without removing low-frequencies \cite{Abe11}. Namely, the existence result holds without the condition $-m_0-2<\beta< m_0+2$. (The non-existence result in the statement (A) is also valid for the case $s=1$.)

	\vspace{5pt} 
	
	\begin{figure}[h]
		\centering
		\begin{minipage}{0.22\columnwidth}
			\centering
			\includegraphics[width=\columnwidth]{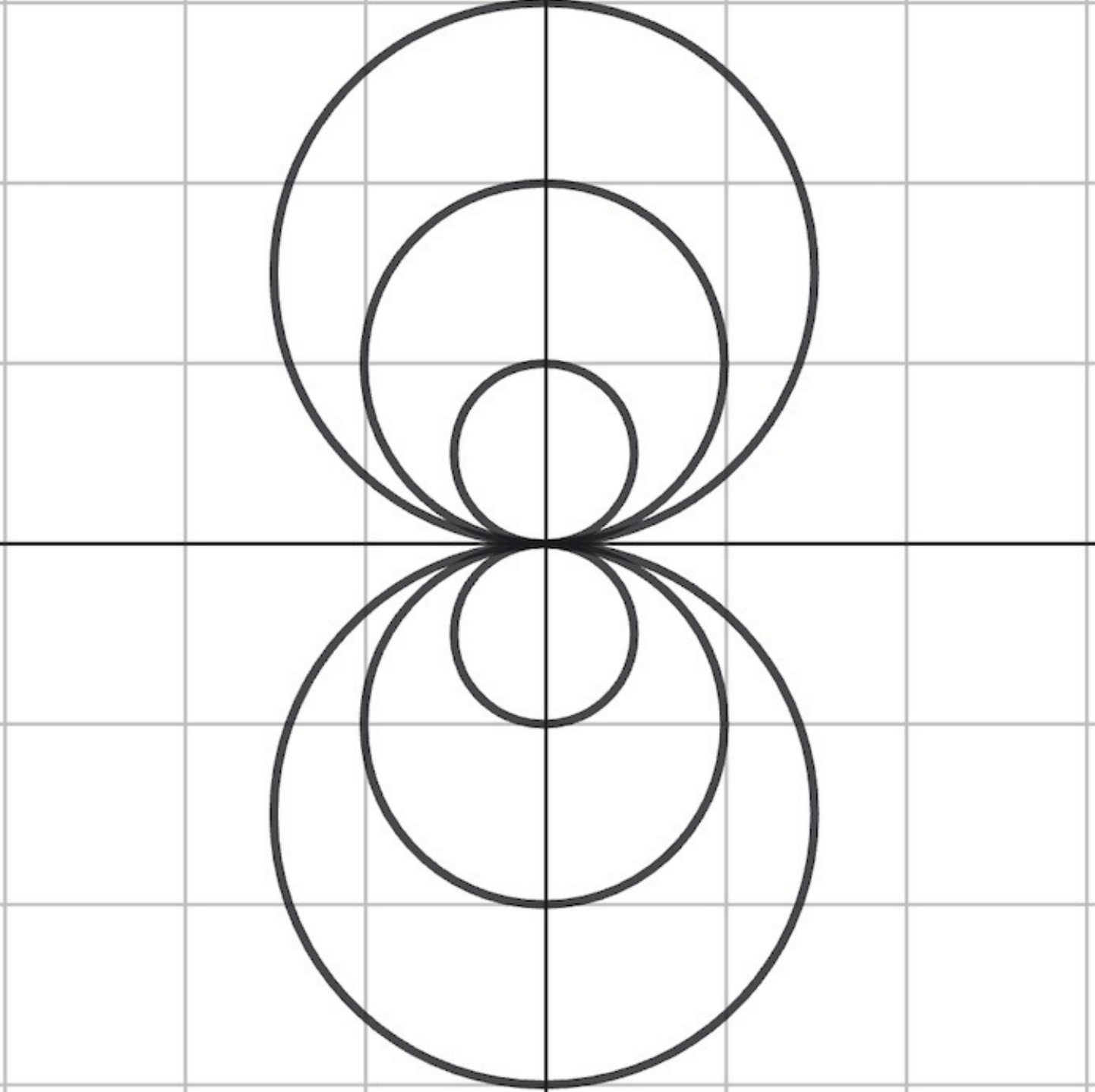}
			\subcaption*{$\beta=1$}
			\label{fig:サンプルA}
		\end{minipage}
		\hspace{5mm}
		\begin{minipage}{0.22\columnwidth}
			\centering
			\includegraphics[width=\columnwidth]{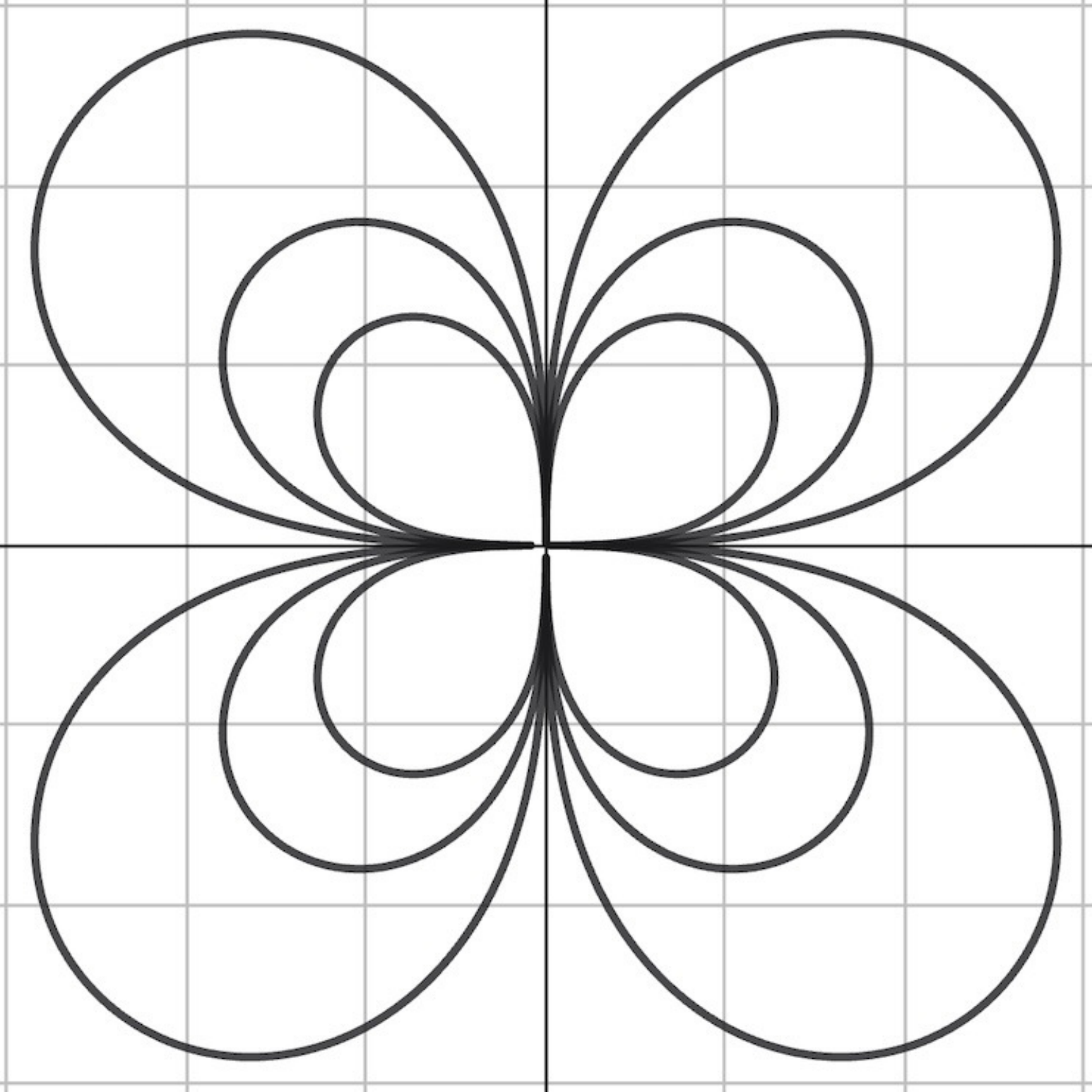}
			\subcaption*{$\beta=2$}
			\label{fig:サンプルB}
		\end{minipage}
		\hspace{5mm}
		\begin{minipage}{0.22\columnwidth}
			\centering
			\includegraphics[width=\columnwidth]{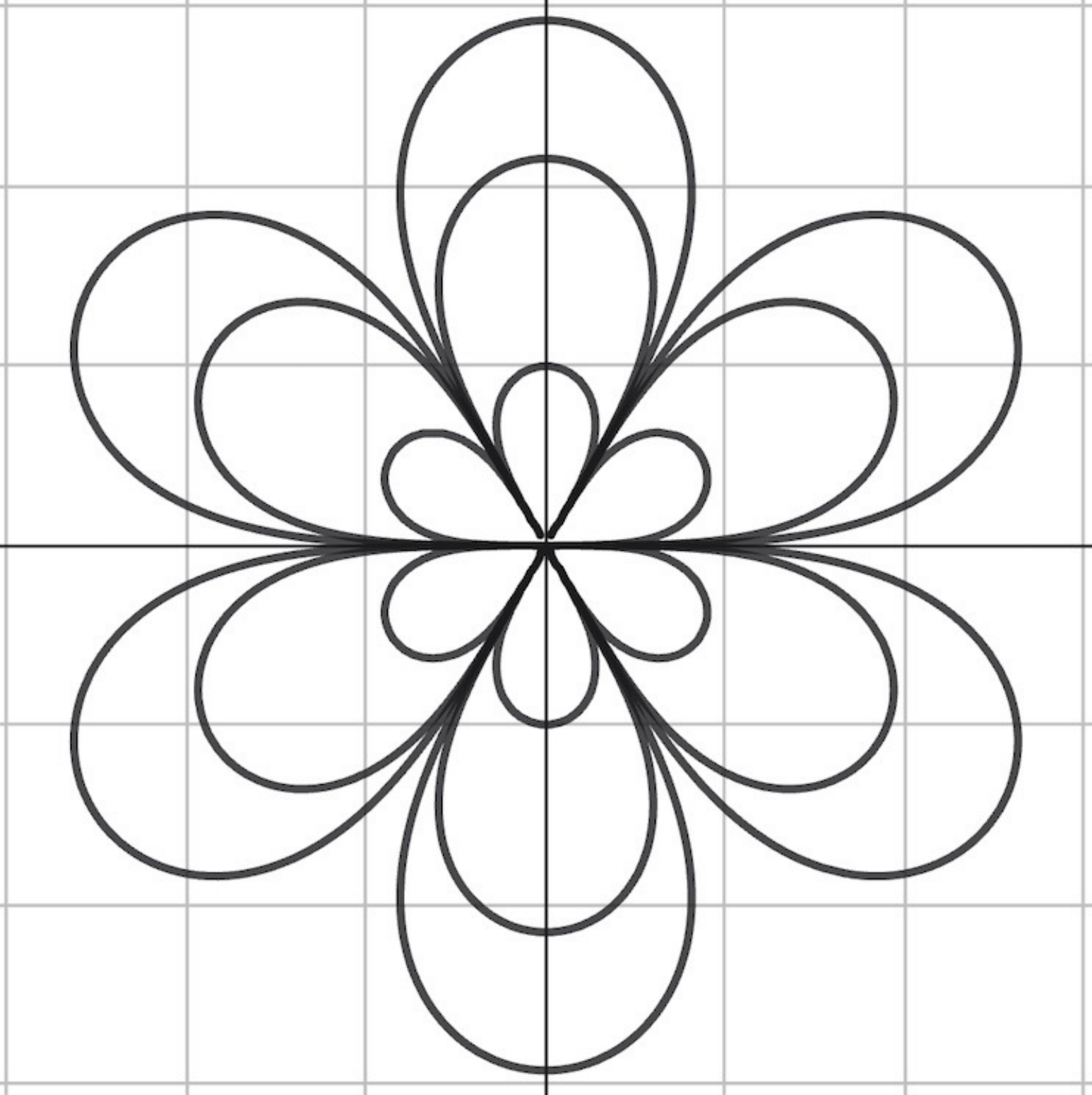}
			\subcaption*{$\beta=3$}
			\label{fig:サンプルC}
		\end{minipage}
		
		\vspace{3mm}
		
		\begin{minipage}{0.22\columnwidth}
			\centering
			\includegraphics[width=\columnwidth]{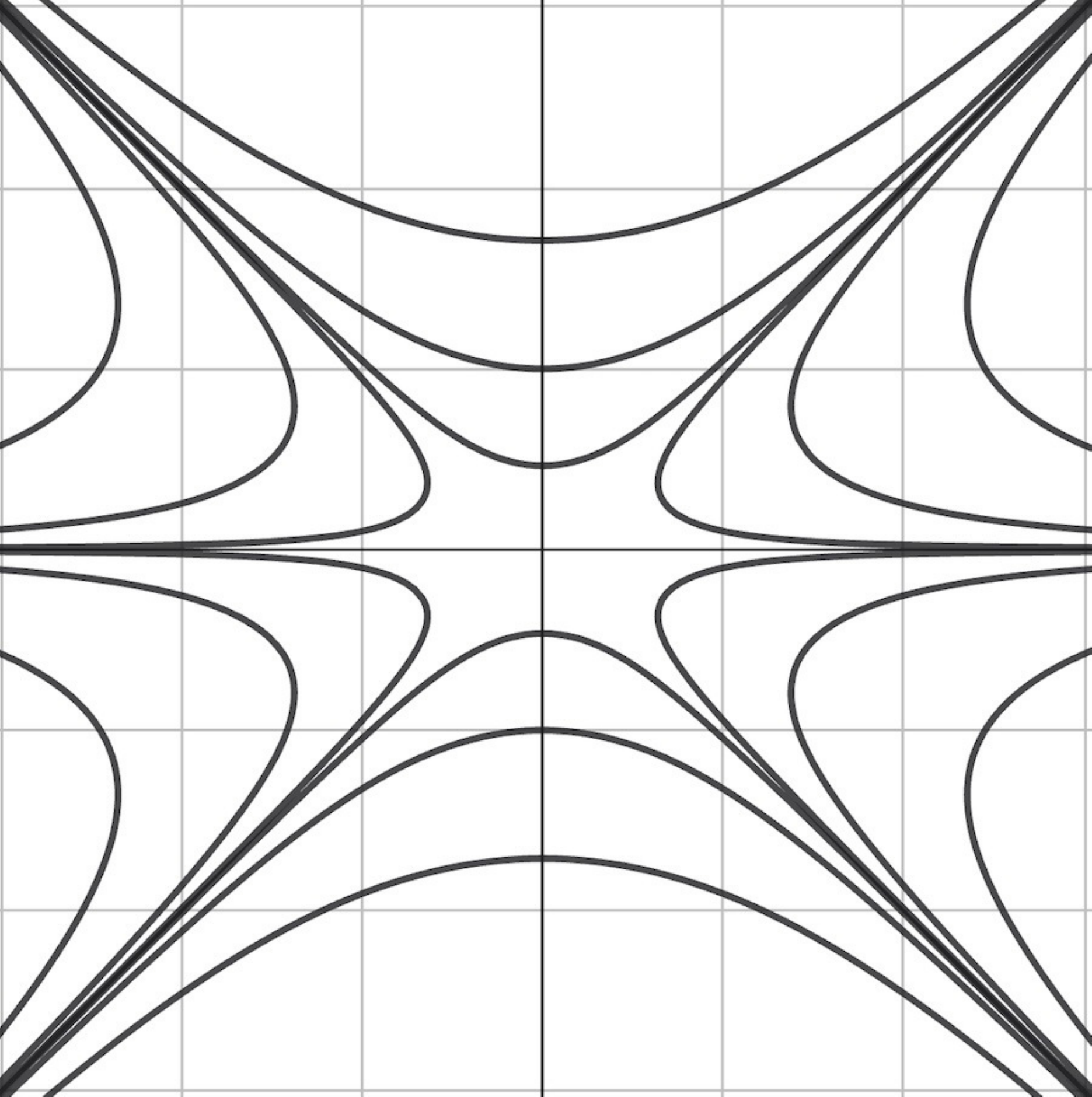}
			\subcaption*{$\beta=-3$}
			\label{fig:サンプルD}
		\end{minipage}
		\hspace{5mm}
		\begin{minipage}{0.22\columnwidth}
			\centering
			\includegraphics[width=\columnwidth]{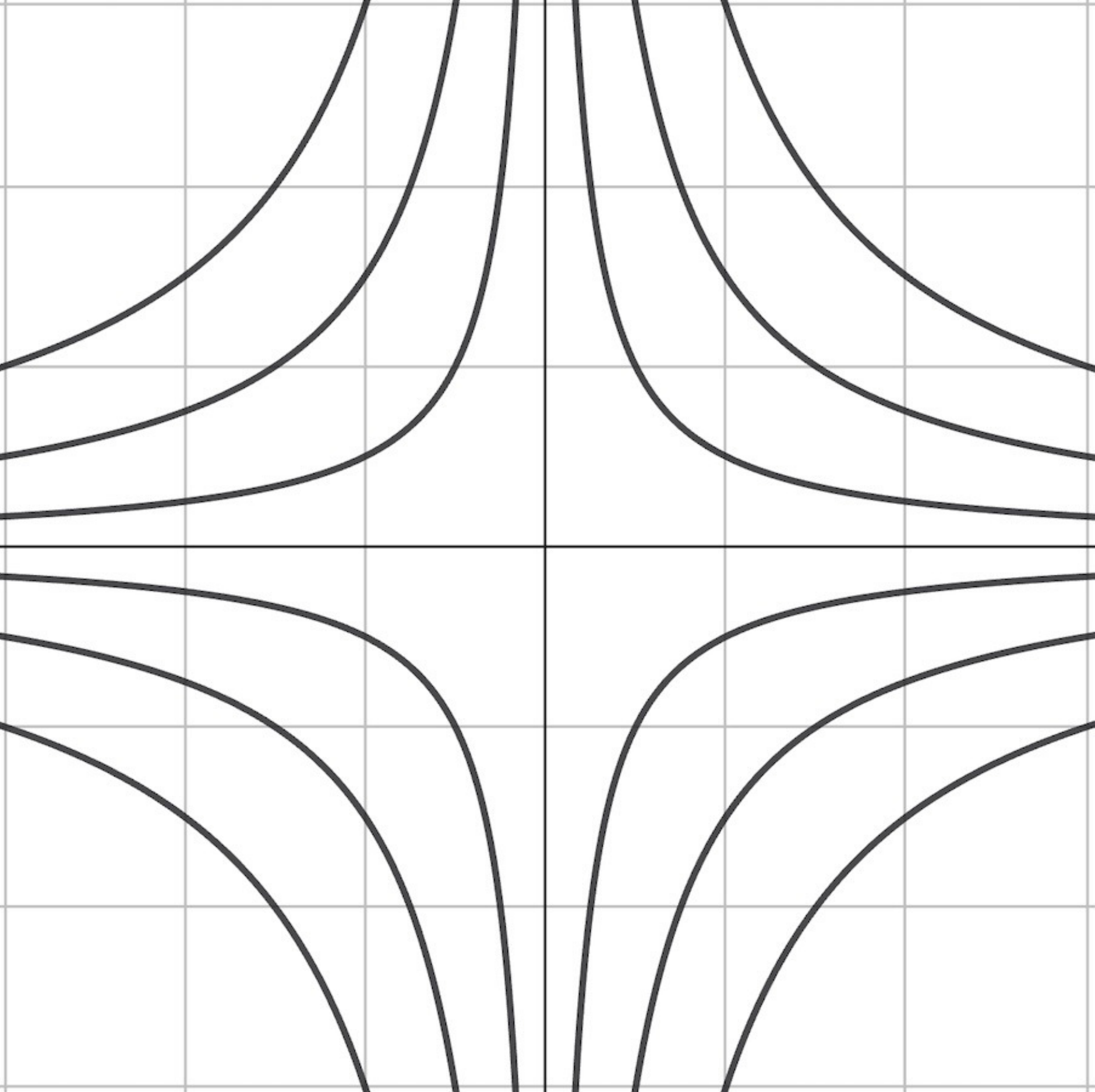}
			\subcaption*{$\beta=-2$}
			\label{fig:サンプルE}
		\end{minipage}
		\hspace{5mm}
		\begin{minipage}{0.22\columnwidth}
			\centering
			\includegraphics[width=\columnwidth]{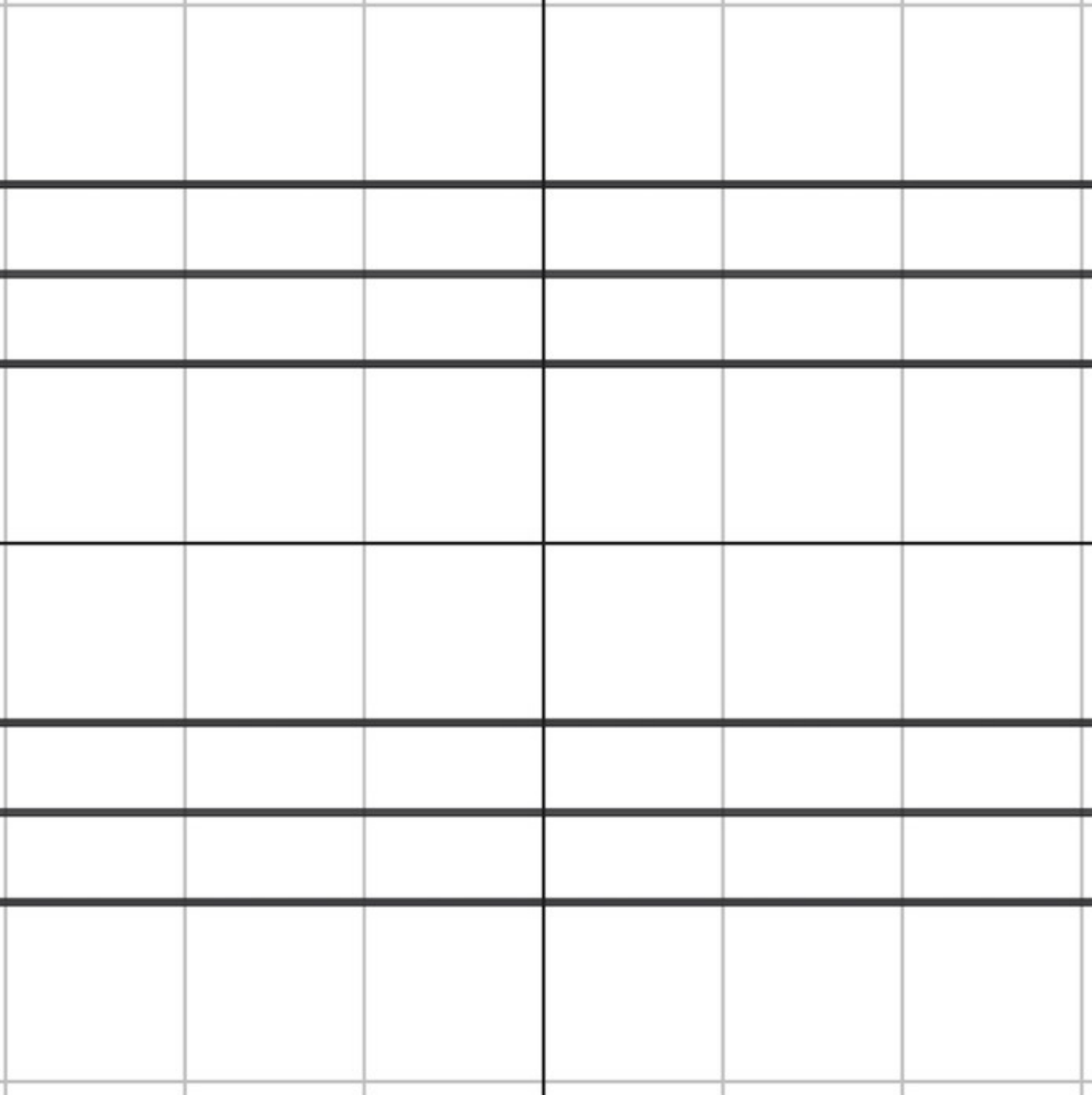}
			\subcaption*{$\beta=-1$}
			\label{fig:サンプルF}
		\end{minipage}
		\caption{Streamlines of irrotational solutions $\psi=r^{-\beta}\sin \beta \theta$ and $\omega=0$}\label{fig:2}
	\end{figure}

 
    In the existence statements of Theorem \ref{t:thm}, $(\psi,\omega)$ satisfies the equation \eqref{eq:SgSQG} in the sense that $\omega=c\psi|\psi|^{\frac{2s}{\beta}}$ for a constant $c>0$. If $(\psi,\omega)$ is $C^{1}$ on $\{r=1\}$ with $\beta>0$, then it satisfies \eqref{eq:SgSQG} pointwise. This is guaranteed in the case (ii) of (A) for $\beta>0$, since we have $w\in C^{2}(\mathbb{T})$ and $g\in C^{1}(\mathbb{T})$. 

    We comment on the condition for $(\psi,\omega)$ to be a weak solution of \eqref{eq:SgSQG}. For a compactly supported smooth test function $\varphi$, we may rewrite \begin{align*}
		\int_{\mathbb{R}^2} \varphi \nabla^\perp \psi \cdot \nabla \omg dx = \int_{\mathbb{R}^2} (-\Delta)^{s/2}\psi \, [ \nabla \varphi \cdot \nabla^\perp , (-\Delta)^{s/2} ] \psi dx. 
	  \end{align*} 
    Since the commutator gains one  derivative, $(-\Delta)^{(s-1)/2}\psi \in L^{2}_{loc}(\mathbb{R}^2)$ guarantees that the above integral is well-defined. For $w \in H^{1-s}(\mathbb{T}),$ this condition is satisfied whenever $\beta < 2 - s.$ When $\beta \ge 2 - s$, $(\psi,\omega)$ is a weak solution only in $\mathbb{R}^{2}\backslash\{ 0 \}$. Theorem \ref{t:thm} does not guarantee $w \in H^{1-s}(\mathbb{T})$ in the case (ii) of (C) for $s-1\leq \beta<-2s$, when we show only $w\in C^{-\beta-\varepsilon}(\mathbb{T})$.  
	
\subsection{{Five self-similar regimes}}

	Let us discuss some motivation, related literature, and potential applications of Theorem \ref{t:thm}, {based on the scaling law for self-similar solutions to the gSQG equations \eqref{eq:FSS} with $\beta$}. We shall see that there are at least {five distinguished regimes} for $\beta$; see Figure \ref{fig:3} \\

	\begin{enumerate}
		\item[(i)] $\beta = -2$ ({scale-invariant}), which is critical for local well-posedness of the gSQG equations. 
		\item[(ii)] {$-2<\beta < -2s$, which is supercritical for local well-posedness of the gSQG equations with locally bounded $\omega$.}
		\item[(iii)] $\beta = -2s$ ({$0$-homogeneous vortex}), which is critical in terms of the strongest conserved quantity $\nrm{\omg}_{L^\infty}$.  
		\item[(iv)] {$-2s<\beta<0$, which is a regime with locally unbounded $\omega$ and bounded $\psi$.}
		\item[(v)] $\beta = 0$ ({$0$-homogeneous stream function}), after which $\psi$ becomes unbounded near $x = 0$. 
	\end{enumerate}

	\begin{figure}[h]
  \begin{minipage}[b]{0.45\linewidth}
\hspace{-118pt}
\includegraphics[scale=0.195]{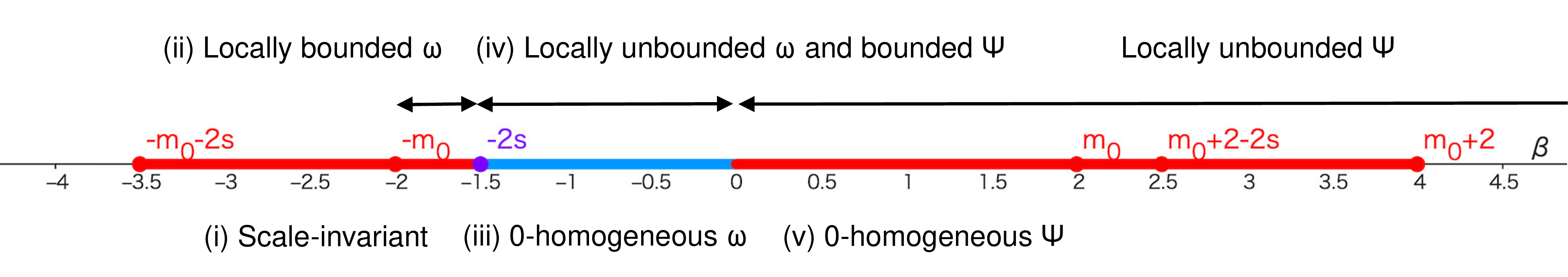}
  \end{minipage}
  \caption{Five self-similar regimes (i)-(v) in critical and supercritical regimes $\beta\geq -2$ for local well-posedness of gSQG (Case $s=3/4$ and $m_0=2$)}\label{fig:3}
\end{figure}

	\subsubsection{{$\beta=-2$ (scale-invariant)}}

	Note that Theorem \ref{t:thm} shows the existence of homogeneous solutions for the scale-invariant case $\beta = -2$. Scale-invariant solutions provide important information about generic solution behavior on blow-up, long-time dynamics, instability, etc. This is related to the fact that for all $0<s\le1$, the scale-invariant solutions are \textit{critical} for local well-posedness, as the corresponding velocity is exactly Lipschitz continuous and not better.
	
	The evolution of scale-invariant solutions is described by a time-dependent system on the torus \cite{EJ20b}. For the case of the Euler equations ($s=1$), the time-dependent profile $(w,g)$ satisfies the non-local transport equations 
	
	\begin{align*}
		\partial_t g+2w \partial_{\theta}g=0,\quad 
		-(\partial_{\theta}^{2}+4)w=g.
	\end{align*}\\
	The work \cite{EJ20b} establishes the existence and uniqueness of global in time solutions to this system for bounded $m$-fold symmetric initial vorticity. The subsequent work \cite{Elgindi2022} investigates steady states and long-time behavior of solutions, and in particular shows that all bounded steady states $g\in L^{\infty}(\mathbb{T})$ of this system are piecewise constants (under $m$-fold symmetry). See also \cite{EMS}. 
	
	Furthermore, the scale-invariant class is the natural setup for studying vortex patches with corners \cite{EJSVP1,EJSVP2,ElJo} and for logarithmic vortex spirals, including Prandtl and Alexander spirals 
    \cite{JS24,EG,CKO24a,CKO24b,CKO25,Cho}.
	 
	Such studies can be extended to scale-invariant solutions for the gSQG equations in the whole range $0<s<1$; in this case, the time-dependent profile $(w,g)$ satisfies the following system reminiscent of 1D blow-up models of the 3D Euler equations \cite{EJ20b,CCZ21}: 
	
	\begin{align*}
		\partial_t g+2w \partial_{\theta}g=2(1-s)g\partial_{\theta}w,\quad 
		\mathcal{L}(s,-2)w=g.
	\end{align*}\\
	In the current work, we show that for all $0<s<1$, $\mathcal{L}(s,-2)$  is a non-local $2s$-th order differential operator on the torus, with some series expansions for the multiplier. For the SQG equations ($s=1/2$), this system resembles the generalized Constantin--Lax--Majda (gCLM) equations for $a=2$ introduced in \cite{OSW08,OSW14}: 
	
	\begin{align*}
		\partial_t g+aw \partial_{\theta}g=g\partial_{\theta}w,\quad 
		\partial_{\theta}w=Hg,
	\end{align*}\\
	where $H$ denotes the Hilbert transform (The cases $a=0$ and $a=1$ are the CLM equations \cite{CLM85} and De Gregorio equations \cite{DeG}, respectively). It has recently been demonstrated that the gCLM equations for $0\leq a\leq 1$ exhibit various self-similar blow-ups \cite{EJ20c,CHH21,LSS,EGM21,HQWW, HTW,WLGB}. At the same time, the De Gregorio equations admit a large class of global-in-time solutions \cite{JSS19,LLR,Chen25,Chen21}. In particular, the work \cite{JSS19} shows the asymptotic stability of the steady state $g=\sin\theta$. Similar global regularity and asymptotic stability results could hold also for the scale-invariant solutions in the range $0<s<1$; note that Theorem \ref{t:thm} shows the existence of H\"older continuous steady states $g\in C^{1-s}(\mathbb{T})$ for the scale-invariant case $\beta=-2$ and all $0<s<1$. A closely related one-dimensional model derived from the SQG equation was studied in \cite{CC10}.

\subsubsection{{$-2<\beta<-2s$}} The case $-2 < \beta$ is \textit{supercritical} for local well-posedness (for any $s$), since the corresponding velocity field is strictly less regular than Lipschitz continuous, which is essential for solving the transport equation uniquely. Already at the critical regularity when the initial velocity is (almost) Lipschitz continuous, the corresponding gSQG solution may lose regularity \cite{BL1,BL2,KJ24,CMZ1,Elgindi20,Jeong21,EJ17,CJK25}. A few recent works even provide non-existence results in the supercritical Sobolev spaces \cite{CMZ2,CMZO,CMZO2}. Furthermore, with a forcing term, non-uniqueness of the gSQG solutions was recently obtained even in the ``slightly'' supercritical regime \cite{Vishik18,Vishik18b,ABCDGK,DoMe,CFMS}. Remarkably, even without forcing, \cite{CFMS} obtained a non-uniqueness result at infinite time for $1/2\leq s<1$ by proving non-linear instability of certain radial vortices. Theorem \ref{t:thm} constructs homogeneous solutions in the entire range $-2<\beta<-2s$, for all $0<s<1$. Studying the behavior of perturbations of these solutions might provide alternative approaches to non-existence and non-uniqueness problems in a larger range of supercritical spaces.

\subsubsection{{$\beta=-2s$ ($0$-homogeneous vortex)}}
	
	The case $\beta = -2s$ corresponds to having $\omega$ as a pure function of $\theta$, i.e. $\omega(x) = g(\theta)$, and deserves a separate discussion. Note that $g \in L^{\infty}$ implies $\omega \in L^{\infty}$, and $L^{\infty}$ is the strongest conserved quantity for the gSQG equations. 
	
	
	For the Euler equations $s = 1$, the Bahouri--Chemin patch \cite{BC94}, which has piecewise constant vorticity {$g(\theta)=\textrm{sgn}({\cos\theta})\textrm{sgn}({\sin\theta})$}, is a particularly important steady state, as it may arise as the infinite time limit of Euler solutions under odd symmetries at the axes \cite{EH,Den09,Den15a,Den15b,KS}. More precisely, the \textit{structure of the set of steady states} near the Bahouri--Chemin patch could determine the behavior of nearby solutions, both in the smooth and patch cases. This was one of the motivations for \cite{EH,Den09,Den15a,Den15b} to consider families of Euler steady states near Bahouri--Chemin. 
		
	For all $0<s<1$, in the specific scaling $\beta = -2s$, the same formula of $g(\tht)$ defines a weak solution to \eqref{eq:SgSQG}, which we shall also refer to as the Bahouri--Chemin patch.
    For all $0<s<1$, Theorem \ref{t:thm} gives the existence of a continuous family of steady states where $\beta \to -2s$. It would be very interesting to prove that this family actually converges to the Bahouri--Chemin patch; when $\beta = -2s$, Theorem \ref{t:thm} also provides the non-existence of steady solutions with a regularity assumption on $g$.
	
	More generally, $0$-homogeneous patch solutions to the gSQG equations provide important candidates for algebraic spiral formation and non-uniqueness of the initial value problem, see discussions in \cite{Elling,SWZ25,GGS24,JeonJeong,CFMS}. Several recent works provide well/ill-posedness for gSQG patches with relatively low regularity \cite{KYZ17,GP21,GNP22,KL25,AA24}.

\subsubsection{{$-2s < \beta < 0$}} The range  $-2s < \beta < 0$ precisely corresponds to the case when $|\omega(x)|\to\infty, |\psi(x)|\to 0$ as $|x|\to 0$. We prove that, for $-2s \le \beta \le 0$, there cannot exist homogeneous solutions with $g$ regular in $\theta$, under the odd symmetry. We remark that even in this parameter range, there could exist homogeneous solutions with very low regularity in $\theta$. In the Euler and the 2D Boussinesq system, such low regularity homogeneous solutions were constructed in \cite{EH,AGJ} relying on the local nature of the resulting equations on $\mathbb{T}$. It seems like a very challenging problem to extend these results to the gSQG case.

\subsubsection{{$\beta = 0$ ($0$-homogeneous stream function)}} Lastly, Theorem \ref{t:thm} gives the existence of homogeneous solutions in the very singular regime $\bt > 0$, where $|\psi(x)|\to\infty$ as $|x|\to 0$. For $1/2 \le s < 1$, Theorem \ref{t:thm} gives the existence for arbitrarily small $\bt>0$, and one may study the limit of these homogeneous solutions when $\bt\to0^+$. This is a singular limit, since Theorem \ref{t:thm} also shows the non-existence of a homogeneous solution which is smooth in $\tht$ for $\bt = 0$. 
	
This ``desingularization'' problem was solved in the Euler case $s = 1$ in \cite[\S 1.2.5, Theorem 1.10]{AGJ}: there is a sequence of homogeneous Euler solutions $\bt\to0^+$ converging to the vortex sheet $\omg_{sh} := r^{-2}\dlt_{\pi/2}(\tht)$, where $\dlt_{\pi/2}$ denotes the Dirac delta distribution at $\tht = \pi/2$. This defines a weak steady solution to Euler in $\bbR^2\backslash \{ 0\}$, see \cite[Chapter 9]{MaB}. Furthermore, this singular solution was considered in the astrophysics literature in relation to current sheet formation \cite{LB94,Aly94}.

\subsection{Ideas of the proof}

We study a homogeneous version of the stationary equations \eqref{eq:SgSQG} for the profile $(w,g)$ on the torus $\mathbb{T}$: 

\begin{equation}\label{eq:HSgSQG}
\begin{aligned}
\beta w \partial_{\theta}g =(\beta+2s)g\partial_{\theta}w,\quad  
\mathcal{L}(s,\beta)w=g,\quad  0<s<1,
\end{aligned}
\end{equation}\\
We first develop (i) a linear theory for the one-dimensional non-local operator $\mathcal{L}(s,\beta)$ involving the parameter $\beta$ (without restricting it to odd functions). We then apply it to show (ii) the non-existence of odd symmetric solutions to \eqref{eq:HSgSQG} and (iii) the existence of odd symmetric solutions to \eqref{eq:HSgSQG}.

\subsubsection{The polar-mode formula}

We first compute the fractional Laplace operator for homogeneous functions by the Hankel transform and derive  a closed‑form spectral formula for $(-\Delta)^s$
 on homogeneous modes: the polar-mode formula

 \begin{align}
(-\Delta)^{s}\!\bigl(r^{-\beta} e^{i m \theta}\bigr)
= \mu_{|m|}(s,\beta) r^{\,-\beta-2s}\, e^{i m \theta},\quad \beta<2,  \label{eq: EV0} 
\end{align}\\
with the constant 

\begin{align}
\mu_{|m|}(s,\beta)
= 2^{2s}\frac{\Gamma(\frac{|m|+\beta}{2}+s )\Gamma(\frac{|m|-\beta}{2}+1 )}{\Gamma(\frac{|m|+\beta}{2})\Gamma(\frac{|m|-\beta}{2}+1-s )}.  \label{eq: C0}
\end{align}\\
The formula \eqref{eq: EV0} includes the classical examples $\mu_{|m|}(1,\beta)=m^{2}-\beta^{2}$ and $\mu_{|m|}(1/2,-2)=\frac{|m|(|m|^{2}-4)}{|m|^{2}-1}$ \cite{EJ20b} and shows that generically the constant $\mu_{|m|}(s,\beta)$ has finite poles and zeros as a function of $|m|$; see Figure \ref{fig:5}. We show that $\mu_{|m|}(s,\beta)$ is increasing from the largest pole to infinity and has the asymptotics $|m|^{-2s}\mu_{|m|}(s,\beta)\to 1$ as $|m|\to\infty$. Simultaneously, formula \eqref{eq: C0} yields the optimal intervals $-|m|-2s < \beta < |m| + 2$ over which $\mu_{|m|}(s,\beta)$ remains finite, and $-|m| < \beta < |m| + 2 - 2s$, over which $\mu_{|m|}(s,\beta)$ is positive, for each Fourier mode $m \in \mathbb{Z}$; see Figure \ref{fig:6}. We remark that in \cite[Proposition A.1]{GGS24} an alternative, much more complicated formula of \eqref{eq: EV0} was provided in terms of an integral.

The interval $-|m|-2s < \beta < |m| + 2$ becomes larger for $|m|$. Namely, for $m_0\in \mathbb{N}_0=\mathbb{N}\cup\{0\}$ and $-m_0-2s < \beta < m_0 + 2$, the constant $\mu_{|m|}(s,\beta)$ remains finte for all $|m|\geq m_0$. We thus consider the space of high-frequency functions $H^{2s}_{m_0}(\mathbb{T})$ with vanishing Fourier modes $|m|<m_0$ and define the nonlocal $2s$-th order differential operator on the torus by using the finite symbols $\mu_{|m|}(s,\beta)$:

\begin{align}
\mathcal{L}(s,\beta)w=\sum_{|m|\geq m_0}\mu_{|m|}(s,\beta)\hat{w}_me^{im\theta},\quad \hat{w}_m=\frac{1}{2\pi}\int_{\mathbb{T}}w(\theta)e^{-im\theta}d\theta.  \label{eq:Formula}
\end{align}\\
This formula~\eqref{eq:Formula} extends the classical definition of the operator $r^{\beta+2s}(-\Delta)^{s}(r^{-\beta}\,\cdot)$ to large values of $|\beta|$, and constitutes a key ingredient in the analysis of the non-linear problem \eqref{eq:HSgSQG}. We establish the following basic properties of the operator $\mathcal{L}(s,\beta)$: 

\begin{itemize}
\item The inverse operator $\mathcal{L}(s,\beta)^{-1}$ exists for $-m_0<\beta<m_0+2-2s$
\item The associated bilinear form $B(\cdot,\cdot)$ is a bounded operator on $H^{s}_{m_0}(\mathbb{T})\times H^{s}_{m_0}(\mathbb{T})$.
\item The $L^{p}$ and H\"older regularity estimates hold.
\end{itemize}

\subsubsection{Non-existence}

We show the non-existence of odd symmetric solutions to \eqref{eq:HSgSQG} for $-2s< \beta< 0$ based on the idea in the works on axisymmetric homogeneous Euler equations \cite{Shv}, \cite{Abe11}. We first integrate the first equation of \eqref{eq:HSgSQG} and show that $wg=0$ on $\mathbb{T}$ and reduce to the irrotational case $g=0$. However, unlike the Euler equations, the operator $\mathcal{L}(s,\beta)$ is non-local and $g=0$ does not immediately follow from $wg=0$. We thus consider the equations for Fourier modes $(m^{2}-\beta^{2})\hat{w}_m=K_{|m|}(s,\beta)^{-1}\hat{g}_m$ by using some positive symbol $K_{|m|}(s,\beta)$ and show $g=0$ by Parseval's identity. On the physical side, the factrization $\mu_{|m|}(s,\beta)=K_{|m|}(s,\beta)(m^{2}-\beta^{2})$ means the decomposition of the operator $\mathcal{L}(s,\beta)=\mathcal{K}(s,\beta)(-L(\beta))$ for $L(\beta)=\partial_{\theta}^{2}+\beta^{2}$ and some operator $\mathcal{K}(s,\beta)$. We consider the equation with a local operator $-L(\beta)w=\mathcal{K}(s,\beta)^{-1}g$ and integrate $0=(-L(\beta)w)g=(\mathcal{K}(s,\beta)^{-1}g)g$ on $\mathbb{T}$. We apply similar ideas for the endpoint cases $\beta=-2$ and $0$.

\subsubsection{Existence}
The main task of this study is to show the existence of odd symmetric solutions to \eqref{eq:HSgSQG} for $\beta<-2s$ and $0<\beta$. We choose $g=cw|w|^{\frac{2s}{\beta}}$ with a constant $c>0$ so that $(w,g)$ are odd symmetric and deduce the existence of solutions to \eqref{eq:HSgSQG} from the semilinear elliptic problem:  

\begin{align*}
\mathcal{L}(s,\beta)w=cw|w|^{\frac{2s}{\beta}}.
\end{align*}\\
We seek critical points of the functional

 \begin{align*}
I[w]
=\frac{1}{2}B(w,w)-\frac{c\beta }{2(\beta+s)}\int_{\mathbb{T}}|w|^{2+\frac{2s}{\beta}}d \theta, 
\end{align*}\\
on the odd symmetric high-frequency fractional Sobolev space $X=H^{s}_{m_0,\textrm{odd}}(\mathbb{T})$ with Fourier modes larger than $m_0\in \mathbb{N}$. We show the embedding $H^{s}(\mathbb{T})\subset L^{2+\frac{2s}{\beta}}(\mathbb{T})$ and $I\in C^{1}(X; \mathbb{R})$ (for $1/2-s<\beta$ in the case $0<s<1/2$). We apply $L^{p}$ and H\"older regularity estimates established in this work for critical points $w\in H^{s}_{m_0}(\mathbb{T})$ with $g=cw|w|^{\frac{2s}{\beta}}\in H^{-s}_{m_0}(\mathbb{T})$ and show the desired regularity of the critical points (stated in Theorem \ref{t:thm}) by an iterative argument. We restrict the parameter $\beta$ to the following ranges: \\

\noindent
(a) $-m_0-2s<\beta<m_0+2$ for the usage of the linear operator $\mathcal{L}(s,\beta)$ \\
(b) $\beta<-2s$ and $0<\beta$ for the positive nonlinearity\\
(c) $1/2-s<\beta$ in the case $0<s<1/2$ for the embedding $H^{s}(\mathbb{T})\subset L^{2+\frac{2s}{\beta}}(\mathbb{T})$\\

We apply the following four variational principles for $I$ (see Figure \ref{fig:4})\\

\noindent
(i) Minimizing method: $-m_0<\beta<-2s$\\
(ii) Mountain pass theorem: $0<\beta<m_0+2-2s$ ($1/2-s<\beta<m_0+2-2s$ in the case $0<s<1/2$)\\
(iii) Linking theorem: $m_0+2-2s\leq \beta<m_0+2$\\
(iv) Saddle-point theorem: $-m_0-2s<\beta\leq -m_0$\\

\begin{figure}[h]
  \begin{minipage}[b]{0.45\linewidth}
\hspace{-118pt}
\includegraphics[scale=0.195]{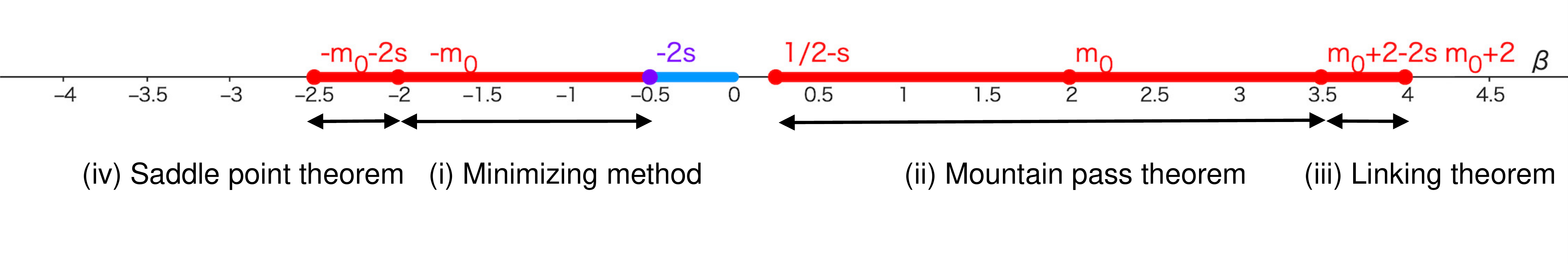}
  \end{minipage}
    \vspace{-10pt}
  \caption{The ranges of $\beta$ and variational principles. Case $s=1/4$ and $m_0=2$}\label{fig:4}
\end{figure}

Those divisions (i)-(iv) are due to signs of eigenvalues of the operator $\mathcal{L}(s,\beta)$ and sub/superlinear nonlinearity. We show the Palais--Smale condition and functional estimates on proper subsets and apply those variational principles. 

\subsection{Organization of the paper} The rest of the paper is organized as follows.
\S \ref{s:2} derives the polar-mode formula and investigates the properties of the constant $\mu_{|m|}(s,\beta)$. \S \ref{s:3} defines the operator $\mathcal{L}(s,\beta)$ and the associated bilinear form on the high-frequency space. \S \ref{s:4} establishes the regularity estimates for the operator $\mathcal{L}(s,\beta)$ on $L^{p}$ and H\"older spaces. \S \ref{s:5} shows the non-existence of odd symmetric solutions to \eqref{eq:HSgSQG} in Theorem \ref{t:thm}. \S \ref{s:6} shows the existence of odd symmetric solutions to \eqref{eq:HSgSQG} by variational principles and completes the proof of Theorem \ref{t:thm}. \S \ref{s:7} performs numerical computations of the aforementioned solutions. Appendix \ref{a:A} shows an increasing property of the ratio of the gamma functions used in \S \ref{s:2}. Appendix \ref{a:B} shows the Sobolev inequality on the torus.

\subsection{Acknowledgments}
KA has been supported by the JSPS through the Grant in Aid for Scientific Research (C) 24K06800, MEXT Promotion of Distinctive Joint Research Center Program JPMXP0723833165, and Osaka Metropolitan University Strategic Research Promotion Project (Development of International Research Hubs).  JGS has been partially supported by the MICINN
(Spain) research grant number PID2021--125021NA--I00, by NSF under Grants
DMS-2245017, DMS-2247537 and DMS-2434314, by the AGAUR project 2021-SGR-0087 (Catalunya) and by a Simons Fellowship. This material is based upon work supported by a grant from the Institute for Advanced Study School of Mathematics. IJ has been supported by the NRF grant from the Korea government (MSIT), RS-2024-00406821,  2022R1C1C1011051 and by the Asian Young Scientist Fellowship. IJ acknowledges the hospitality of the Mathematics Department at Brown University during his April 2024 visit, when this work was initiated.

Upon completion of this work, we learned about the paper \cite{miguel} by Pascual-Caballo, where the author proves similar results in different settings. We  coordinated so that both papers appeared simultaneously on the arXiv.

\section{The fractional Laplacian for homogeneous functions} \label{s:2}

We show that the constant $\mu_{|m|}(s,\beta)$ of the polar-mode formula \eqref{eq: C0} has poles and zeros as a function of $|m|$ and $\mu_{|m|}(s,\beta)$ is increasing from the largest pole to infinity and has the asymptotics $|m|^{-2s}\mu_{|m|}(s,\beta)\to 1$ as $|m|\to\infty$. We then derive optimal intervals $-|m|-2s<\beta<|m|+2$ over which $\mu_{|m|}(s,\beta)$ is finite and $-|m|<\beta<|m|+2-2s$ over which $\mu_{|m|}(s,\beta)$ is positive for fixed Fourier mode $m\in \mathbb{Z}$. At the end of this section, we prove the polar-mode formula \eqref{eq: EV0} by the Hankel transform.

\subsection{The polar-mode formula}

We first observe properties of the constant $\mu_{|m|}(s,\beta)$.

\begin{lem}\label{l: FML}
Let $0<s<1$, $m\in \mathbb{Z}$, and $\beta\in \mathbb{R}$. Then, for $\beta$ not leading to Gamma-function singularities (see Proposition \ref{p: poleszerosm} below),

\begin{align}
(-\Delta)^{s}\!\bigl(r^{-\beta} e^{i m \theta}\bigr)
&= \mu_{|m|}(s,\beta) r^{\,-\beta-2s}\, e^{i m \theta},\quad \beta<2,  \label{eq: EV} \\
(-\Delta)^{-s}\!\bigl(r^{-\beta-2s} e^{i m \theta}\bigr)
&= \mu_{|m|}(s,\beta)^{-1} r^{\,-\beta}\, e^{i m \theta},\quad \beta<2-2s,  \label{eq: IEV}
\end{align}\\
hold with the constant 

\begin{align}
\mu_{|m|}(s,\beta)
= 2^{2s}\frac{\Gamma(\frac{|m|+\beta}{2}+s )\Gamma(\frac{|m|-\beta}{2}+1 )}{\Gamma(\frac{|m|+\beta}{2})\Gamma(\frac{|m|-\beta}{2}+1-s )}.  \label{eq: C}
\end{align}
\end{lem}

\vspace{5pt}

\begin{rems}
The constant \eqref{eq: C} includes the following examples:\\

\begin{itemize}
\item Case $s=1$; $\mu_{|m|}(1,\beta)=|m|^{2}-\beta^{2}$.
\item Case $s=0$; $\mu_{|m|}(0,\beta)=1$.
\item Case $s=1/2$; Elgindi and Jeong \cite[Lemma 4.5]{EJ20b} computed the symbol of the constitutive law for scale-invariant solutions ($\beta=-2$) to the SQG equations. The formula \eqref{eq: C} provides the same symbol 
\begin{align*}
\mu_{|m|}(1/2,-2)
= 2\frac{\Gamma(\frac{|m|-1}{2} )\Gamma(\frac{|m|+4}{2} )}{\Gamma(\frac{|m|-2}{2})\Gamma(\frac{|m|+3}{2})}
=\frac{|m|(|m|^{2}-4)}{|m|^{2}-1}.
\end{align*}\\
The constant $\mu_{|m|}(1/2,-2)$ has the pole at $|m|=1$ and zeros at $|m|=0$ and $2$. For $|m|>1$, $\mu_{|m|}(1/2,-2)$ is increasing and $\lim_{|m|\to\infty}|m|^{-1}\mu_{|m|}(1/2,-2)=1$; see Figure \ref{fig:5}.
\end{itemize}
\end{rems}

\begin{figure}[h]
  \begin{minipage}[b]{0.45\linewidth}
\hspace{-60pt}
\includegraphics[scale=0.65]{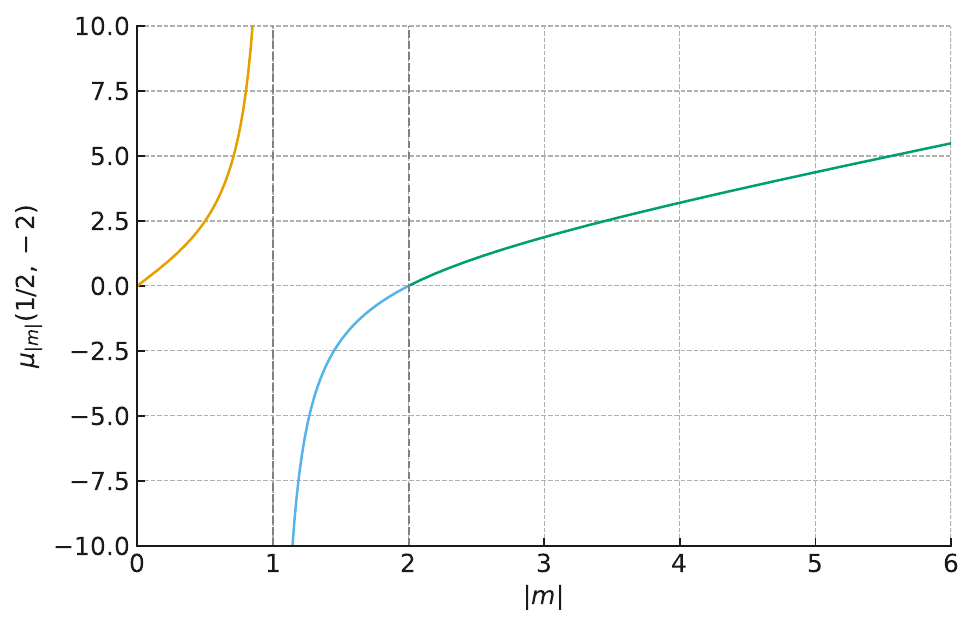}
  \end{minipage}
  \caption{The graph of $\mu_{|m|}(1/2,-2)=\frac{|m|(|m|^{2}-4)}{|m|^{2}-1}$}
  \label{fig:5}
\end{figure}


\subsection{Poles and zeros for $|m|$}

We show that generically $\mu_{|m|}(s,\beta)$ has alternate poles and zeros for $|m|>0$.

\begin{prop}[Poles and zeros for $|m|$]\label{p: poleszerosm}
Let $0<s<1$ and $\beta\in \mathbb{R}$. 

\noindent
(i) (Poles) The constant $\mu_{|m|}(s,\beta)$ diverges if and only if $|m|$ is the following values:

\begin{align}
|m|=
\begin{cases}
& -2k-\beta-2s,\quad  k\in \mathbb{N}_{0},\quad \textrm{for}\ \beta\leq -2s,\\
& -2k+\beta-2, \quad k\in \mathbb{N}_{0}, \quad \textrm{for}\ \beta\geq 2. 
\end{cases}
\label{eq: P}
\end{align}\\
In particular, $\mu_{|m|}(s,\beta)$ is finite for $-2s<\beta<2$ and $|m|\geq 0$.

\noindent
(ii) (Zeros) The constant $\mu_{|m|}(s,\beta)$ vanishes if and only if $|m|$ is the following values:

\begin{align}
|m|=
\begin{cases}
& -2k-\beta, \quad  k\in \mathbb{N}_{0},\quad \textrm{for}\ \beta\leq 0, \\
& -2k+\beta-2+2s,\quad  k\in \mathbb{N}_{0},\quad \textrm{for}\ \beta\geq 2-2s.
\end{cases}
\label{eq: Z}
\end{align}\\
In particular, $\mu_{|m|}(s,\beta)$ is positive for $0<\beta<2-2s$ and $|m|\geq 0$.
\end{prop}

\vspace{5pt}

\begin{proof}
The constant $\mu_{|m|}(s,\beta)$ diverges if and only if the numerator of \eqref{eq: C} does by the poles of the Gamma function, i.e., $\frac{|m|+\beta}{2}+s \notin -\mathbb{N}_0$ and $\frac{|m|-\beta}{2}+1\notin -\mathbb{N}_0$. The constant $\mu_{|m|}(s,\beta)$ vanishes if and only if the denominator of \eqref{eq: C} diverges by the poles of the Gamma function, i.e., $\frac{|m|-\beta}{2}+1-s \in -\mathbb{N}_0$ and $\frac{|m|+\beta}{2} \in -\mathbb{N}_0$. 
\end{proof}

\subsection{Monotonicity and asymptotics for $|m|$}

We show that $\mu_{|m|}(s,\beta)$ is increasing for $|m|$ larger than the largest pole in \eqref{eq: P} and has the asymptotics $|m|^{-2s}\mu_{|m|}(s,\beta)\to 1$ as $|m|\to\infty$ by using properties of the digamma function $\psi(z)=\Gamma'(z)/\Gamma(z)$; see Figure \ref{fig:5}.

\vspace{5pt}

\begin{prop}[Monotonicity]\label{p: monotonem}
The constant $\mu_{|m|}(s,\beta)$ is increasing for 

\begin{align}
|m|> \begin{cases}
& -\beta-2s,\quad \textrm{for}\ \beta\leq -2s,   \\
& \beta-2,\quad \textrm{for}\ \beta\geq 2,\\
& 2s,\quad \textrm{for}\ -2s<\beta<2.
\end{cases}
\label{eq: mufinite}
\end{align}
\end{prop}

\begin{proof}
We set

\begin{align*}
\mu_{|m|}(s,\beta)=2^{2s}\frac{\Gamma(t+s)\Gamma(l+s)}{\Gamma(t)\Gamma(l)},\quad t=\frac{|m|+\beta}{2},\ l=\frac{|m|-\beta}{2}+1-s.
\end{align*}\\
For $\beta\leq -2s$ and $|m|>-\beta-2s$, we observe that $t>-s$ and set $f(t)=\Gamma(t+s)/\Gamma(t)$. By properties of the digamma function, the function $f(t)$ is increasing (We give a proof in Lemma \ref{l: monotonegamma}). Moreover, the function  

\begin{align*}
\frac{\Gamma(t+s)\Gamma(l+s)}{\Gamma(t)\Gamma(l)}=f(t)f(t-\beta+1-s)
\end{align*}\\
is also increasing for $t>-s$ by Lemma \ref{l: monotonegamma2}. Thus, $\mu_{|m|}(s,\beta)$ is increasing for $|m|>-\beta-2s$. 

Similarly, for $\beta\geq 2$ and $|m|>\beta-2$, we see that $l>-s$ and $f(l+\beta-1+s)f(l)$ is increasing by Lemma \ref{l: monotonegamma2}. Thus $\mu_{|m|}(s,\beta)$ is increasing for $|m|>\beta-2$. For $-2s<\beta<2$ and $|m|>0$, $t+s>|m|/2>0$ and $l+s>|m|/2>0$. By Lemma \ref{l: monotonegamma}, both $f(t)$ and $f(l)$ are positive and increasing. Thus, $f(t)f(t-\beta+1-s)$ is also increasing and so is $\mu_{|m|}(s,\beta)$ for $|m|>0$.
\end{proof}

\vspace{5pt}

\begin{prop}[Asymptotics]\label{p: asymptoticsm}
The constant $\mu_{|m|}(s,\beta)$ has the following asymptotics:

\begin{align}
\lim_{|m|\to \infty}|m|^{-2s}\mu_{|m|}(s,\beta)&=1,   \label{eq: A1}\\
\lim_{|m|\to \infty}|m|^{-2s+1}\frac{d}{d |m|}\mu_{|m|}(s,\beta)&=2s,  \label{eq: A2} \\
\lim_{|m|\to \infty}|m|^{-2s+1}\frac{d}{d \beta}\mu_{|m|}(s,\beta)&=0,  \label{eq: A3}  \\
\lim_{|m|\to \infty}|m|^{-2s+2}\frac{d}{d |m|}\frac{d}{d \beta}\mu_{|m|}(s,\beta)&=0.  \label{eq: A4}  
\end{align}
\end{prop}

\vspace{5pt}

\begin{proof}
We may assume that $\mu_{|m|}(s,\beta)$ is positive by Proposition \ref{p: monotonem}. By integrating the identity,

\begin{align*}
\frac{d}{dz}\log\left(\frac{\Gamma(z+s)}{\Gamma(z)}\right)=\psi(z+s)-\psi(z),\quad z>0,
\end{align*}\\
we find that

\begin{align*}
\frac{\Gamma(z+s)}{\Gamma(z)}=\exp\left(\int_{z}^{z+s}\psi(\eta)d\eta \right).
\end{align*}\\
We thus obtain 

\begin{align*}
\mu_{|m|}(s,\beta)=2^{2s}\exp\left(\int_{\frac{\beta}{2}}^{\frac{\beta}{2}+s}\psi\left(\frac{|m|}{2}+\eta\right)d\eta+\int_{\frac{-\beta}{2}+1-s}^{\frac{-\beta}{2}+1}\psi\left(\frac{|m|}{2}+\eta\right)d\eta  \right). 
\end{align*}\\
The asymptotics \eqref{eq: A1} follows from that of the digamma function $\psi(z)=\log{z}-(2z)^{-1}+O(z^{-2})$ as $z\to\infty$ \cite[6.3.18]{AbS} and  
\begin{align*}
\int_{a}^{a+s}\psi(t+\eta)d\eta=s\log{t}+o(1)\quad \textrm{as}\ t\to\infty,
\end{align*}\\
for $a=\beta/2$ and $-\beta/2+1-s$. By differentiating $\log \mu_{|m|}(s,\beta)$ for $|m|$,    

\begin{align*}
2\mu_{|m|}(s,\beta)^{-1}\frac{d}{d|m|}\mu_{|m|}(s,\beta)=\psi\left(\frac{|m|}{2}+\frac{\beta}{2}+s\right)-\psi\left(\frac{|m|}{2}+\frac{\beta}{2}\right)+\psi\left(\frac{|m|}{2}-\frac{\beta}{2}+1\right)-\psi\left(\frac{|m|}{2}-\frac{\beta}{2}+1-s\right).
\end{align*}\\ 
By the asymptotics of the digamma function, 

\begin{align*}
\lim_{|m|\to\infty}\frac{|m|}{2}\left(\psi\left(\frac{|m|}{2}+\frac{\beta}{2}+s\right)-\psi\left(\frac{|m|}{2}\right)\right)=s.
\end{align*}\\
The asymptotics \eqref{eq: A2} follows from \eqref{eq: A1}. Similarly, by differentiating $\log \mu_{|m|}(s,\beta)$ for $\beta$,    

\begin{align*}
2\mu_{|m|}(s,\beta)^{-1}\frac{d}{d\beta}\mu_{|m|}(s,\beta)=\psi\left(\frac{|m|}{2}+\frac{\beta}{2}+s\right)-\psi\left(\frac{|m|}{2}+\frac{\beta}{2}\right)-\psi\left(\frac{|m|}{2}-\frac{\beta}{2}+1\right)+\psi\left(\frac{|m|}{2}-\frac{\beta}{2}+1-s\right),
\end{align*}\\
and \eqref{eq: A3} follows from \eqref{eq: A1}. By differentiating this identity for $|m|$, 

\begin{align*}
&4\mu_{|m|}(s,\beta)^{-1}\frac{d}{d|m|}\mu_{|m|}(s,\beta)\frac{d}{d\beta}\mu_{|m|}(s,\beta)
+4\mu_{|m|}(s,\beta)^{-1}\frac{d^{2}}{d|m|d\beta}\mu_{|m|}(s,\beta) \\
&=
\psi'\left(\frac{|m|}{2}+\frac{\beta}{2}+s\right)-\psi'\left(\frac{|m|}{2}+\frac{\beta}{2}\right)-\psi'\left(\frac{|m|}{2}-\frac{\beta}{2}+1\right)+\psi'\left(\frac{|m|}{2}-\frac{\beta}{2}+1-s\right).
\end{align*}\\
By the asymptotics of the trigamma function $\psi'(z)=z^{-1}+2z^{-2}+O(z^{-3})$ as $z\to\infty$ \cite[6.4.12]{AbS}, the right-hand side equals $o(|m|^{-2})$ as $|m|\to\infty$. The asymptotics \eqref{eq: A4} follows from \eqref{eq: A1}-\eqref{eq: A3}.
\end{proof}

\subsection{The finite and positive intervals for $\beta$}

We next consider poles and zeros of $\mu_{|m|}(s,\beta)$ for $\beta$.

\begin{prop}[Poles and zeros for $\beta$]\label{l: poleszerosbeta}
Let $0<s<1$ and $m\in \mathbb{Z}$. 

\noindent
(i) (Poles) The constant $\mu_{|m|}(s,\beta)$ diverges if and only if $\beta$ is the following values:
\begin{align}
\beta=
\begin{cases}
& -|m|-2k-2s, \ k\in \mathbb{N}_{0},\\
& |m|+2k+2,\ k\in \mathbb{N}_{0}. 
\end{cases}
\label{eq: Pb}
\end{align}
(ii) (Zeros) The constant $\mu_{|m|}(s,\beta)$ vanishes if and only if $\beta$ is the following values:

\begin{align}
\beta=
\begin{cases}
& |m|+2k+2-2s,\ k\in \mathbb{N}_{0},\\
& -|m|-2k,\ k\in \mathbb{N}_{0},
\end{cases}
\label{eq: Zb}
\end{align}
\end{prop}

\vspace{5pt}

\begin{proof}
The proof follows the same way as that of Proposition \ref{p: poleszerosm}.
\end{proof}

\vspace{5pt}

We show that $\mu_{|m|}(s,\beta)$ is finite on the interval $-|m|-2s<\beta<|m|+2$ and  positive on the interval $-|m|<\beta<|m|+2-2s$; see Figure \ref{fig:6} 

\begin{lem}[The finite and positive intervals]\label{l: intervals}
Let $0<s<1$ and $m\in \mathbb{Z}$. 

\noindent
(i) The interval

\begin{align}
-|m|-2s<\beta<|m|+2   \label{eq: finite}
\end{align}\\
does not include poles of $\mu_{|m|}(s,\beta)$. 

\noindent
(ii) The interval

\begin{align}
-|m|<\beta<|m|+2-2s   \label{eq: positive}
\end{align}\\
does not include zeros of $\mu_{|m|}(s,\beta)$. The constant $\mu_{|m|}(s,\beta)$ is positive on the interval \eqref{eq: positive}.

\noindent
(iii) Set

\begin{align}
\mu_{|m|}(s,\beta)=\kappa_{|m|}(s,\beta)(|m|-\beta+2-2s)(|m|+\beta).   \label{eq: kappa}
\end{align}\\
The constant $\kappa_{|m|}(s,\beta)$ is positive on the interval \eqref{eq: finite} and 

\begin{align}
\mu_{|m|}(s,\beta)\begin{cases}
\ &>0,\quad -|m|<\beta<|m|+2-2s,\\
\ &=0,\quad \beta=-|m|, |m|+2-2s,\\
\ &<0,\quad -|m|-2s<\beta< -|m|,\ |m|+2-2s < \beta<|m|+2.
\end{cases}
\label{eq: Sign}
\end{align}
\end{lem}

\vspace{5pt}

\begin{proof}
The interval \eqref{eq: finite} does not include the poles \eqref{eq: Pb}. The interval \eqref{eq: positive} does not include zeros \eqref{eq: Zb}. By the property of the Gamma function $\Gamma(z+1)=z\Gamma(z)$,

\begin{align*}
\mu_{|m|}(s,\beta)&=2^{2s-2}\frac{\Gamma(\frac{|m|+\beta}{2}+s )\Gamma(\frac{|m|-\beta}{2}+1 )}{\Gamma(\frac{|m|-\beta}{2}+2-s )\Gamma(\frac{|m|+\beta}{2}+1) }(|m|-\beta+2-2s)(|m|+\beta) \\
&=\kappa_{|m|}(s,\beta)(|m|-\beta+2-2s)(|m|+\beta).
\end{align*}\\
The constant $\kappa_{|m|}(s,\beta)$ is positive on the interval \eqref{eq: finite} and the sign of $\mu_{|m|}(s,\beta)$ agrees with that of the quadratic function $(|m|-\beta+2-2s)(|m|+\beta)$ for $\beta$. 
\end{proof}

\vspace{5pt}

\begin{figure}[h]
  \begin{minipage}[b]{0.45\linewidth}
\hspace{-110pt}
\includegraphics[scale=0.65]{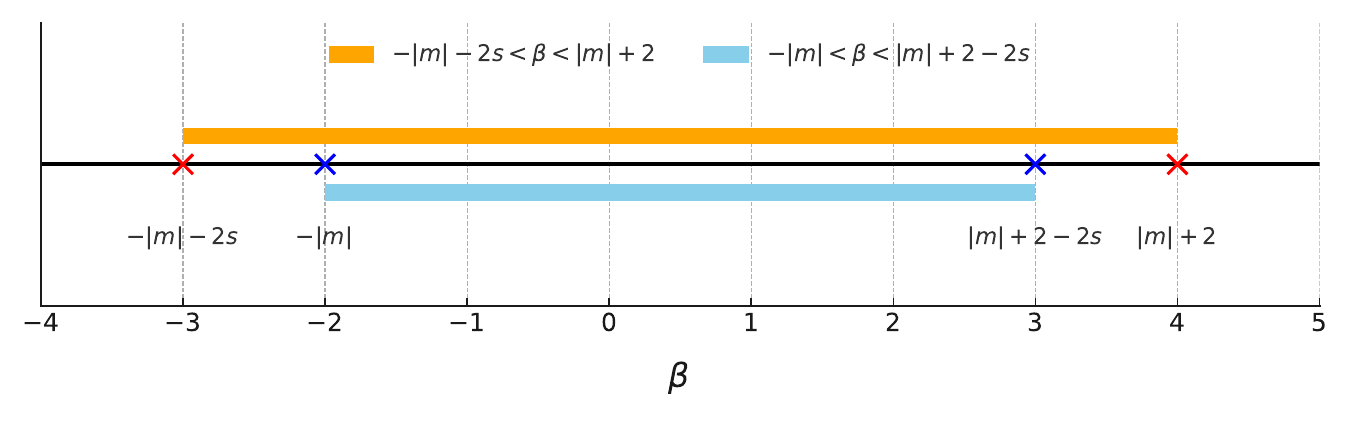}
  \end{minipage}
  \caption{The two intervals of $\beta$ in Lemma \ref{l: intervals} (for $s=1/2$ and $|m|=2$): (i) $-|m|-2s<\beta<|m|+2$ (Orange) over which $\mu_{|m|}(s,\beta)$ is finite and (ii) $-|m|-2s<\beta<|m|+2-2s$ (Lightblue) over which $\mu_{|m|}(s,\beta)$ is positive. The function $\mu_{|m|}(\beta,s)$ has poles at $\beta=|m|+2$ and $-|m|-2s$ (Red) and zeros at $\beta=|m|+2-2s$ and $-|m|$ (Blue).}
  \label{fig:6}
\end{figure}

\subsection{The Hankel transform}
We now prove the polar-mode formulas \eqref{eq: EV} and \eqref{eq: IEV} by using the Hankel transform, e.g., \cite[II, 15]{Deb}. We compute the constant \eqref{eq: C} by the Fourier transform and the inverse Fourier transform

\begin{align*}
\hat{f}(\xi)=\frac{1}{2\pi}\int_{\mathbb{R}^{2}}f(x)e^{-ix\cdot \xi}dx,\quad \check{f}(x)=\frac{1}{2\pi}\int_{\mathbb{R}^{2}}f(\xi)e^{ix\cdot \xi}d\xi.
\end{align*}\\
We use $\check{f}(x)=\hat{f}(-x)$ to compute the inverse Fourier transform from the Fourier transform and the symbol representation of the fractional operator

\begin{align*}
(-\Delta)^{s}f=(|\xi|^{2s} \hat{f})\ \check{}.
\end{align*}\\
We compute the fractional operator for separation variable functions by the Hankel transform. We use the generating function of the $m$-th order Bessel function of the first kind $J_m(x)$ \cite[9.1.41]{AbS}:

\begin{align*}
e^{\frac{x}{2}(z-\frac{1}{z})}=\sum_{m\in \mathbb{Z}}J_m(x)z^{m},\quad z\neq 0. 
\end{align*}
By taking $z=e^{i\theta}$, 
\begin{align*}
e^{i x \sin\theta}=\sum_{m\in \mathbb{Z}}J_m(x)e^{i m\theta},\quad J_m(x)=\frac{1}{2\pi}\int_{0}^{2\pi}e^{ix \sin\theta}e^{-im\theta}d\theta.
\end{align*}\\
The Hankel transform of the radial function $g(r)\in L^{1}_{\textrm{loc}}[0,\infty)$ is as follows:
\begin{align*}
H_{|m|}[g](\rho)=\int_{0}^{\infty}g(r)rJ_{|m|}(\rho r)dr.
\end{align*}
For $g(r)=r^{-\beta}$, the trace at $\rho=1$ yields the following constant \cite[11.4.16]{AbS}:
\begin{align*}
H_{|m|}[r^{-\beta}](1)=\int_{0}^{\infty}r^{-\beta+1}J_{|m|}(r)dr=\frac{2^{-\beta+1}\Gamma(\frac{|m|-\beta+2}{2})}{\Gamma(\frac{|m|+\beta}{2})},\quad -\frac{1}{2}<\beta<|m|+2.
\end{align*}\\
For negative $m$, we set $H_m[g]=(-1)^{m}H_{|m|}[g]$.

\begin{prop}
Let $0<s<1$ and $m\in \mathbb{Z}$. Let $x=re^{i\theta}$ and $\xi=\rho e^{i\varphi}$. Then, 

\begin{align}
\widehat{(g(r)e^{im\theta})}(\xi)&=(-i)^{m}H_m[g](\rho)e^{im\varphi}, \label{eq: FTH}\\
(g(\rho)e^{im\varphi})\ \check{}\ (x)&=(i)^{m}H_m[g](r)e^{im\theta}, \label{eq: IFTH}\\
(-\Delta)^{s}(g(r)e^{im\theta})&=H_{|m|}[\rho^{2s}H_{|m|}[g]](r)e^{im\theta},  \label{eq: FOH}
\end{align}\\
for $g(r)\in L^{1}_{\textrm{loc}}[0,\infty)$.
\end{prop}

\begin{proof}
We observe that 

\begin{align*}
\hat{f}(\xi)=\frac{1}{2\pi}\int_{\mathbb{R}^{2}}f(x)e^{-ix\cdot \xi}dx
=\frac{1}{2\pi}\int_{0}^{\infty}\int_{0}^{2\pi}e^{im\theta}g(r) e^{-ir\rho \cos(\theta-\varphi)}r d\theta dr. 
\end{align*}\\
By changing the variable $\alpha=\theta-\varphi-\pi/2$ and using the symmetry $J_{-m}(-x)=J_m(x)$, 

\begin{align*}
\frac{1}{2\pi}\int_{0}^{\infty}\int_{0}^{2\pi}e^{im\theta}g(r) e^{-ir\rho \cos(\theta-\varphi)}r d\theta dr 
&=\frac{1}{2\pi}\int_{0}^{\infty}\int_{0}^{2\pi}e^{im(\alpha-\frac{\pi}{2}+\varphi )}g(r) e^{-ir\rho \sin\alpha}r d\alpha dr \\
&=(-i)^{m}e^{im\varphi}\int_{0}^{\infty}J_{-m}(-r\rho)g(r)rdr=(-i)^{m}H_m[g](\rho)e^{im\varphi}.
\end{align*}\\
We obtain \eqref{eq: FTH}. The inverse Fourier transform \eqref{eq: IFTH} follows from \eqref{eq: FTH}. The expression \eqref{eq: FOH} follows from \eqref{eq: FTH} and \eqref{eq: IFTH}.
\end{proof}

\begin{proof}[Proof of Lemma \ref{l: FML}]
We show the formula \eqref{eq: EV}. The same computation yields \eqref{eq: IEV}. For $-1/2<\beta<\min\{5/2-2s,2\}$, we take $m$ so that $-1/2<\beta<|m|+2$ and $-|m|-2s<\beta<5/2-2s$. By the scaling,  

\begin{align*}
H_{|m|}[r^{-\beta}](\rho)&=\int_{0}^{\infty}r^{-\beta+1}J_{|m|}(r\rho)dr=\frac{1}{\rho^{-\beta+2}}H_{|m|}[r^{-\beta}](1),\\
H_{|m|}[\rho^{2s+\beta-2}](r)&=\frac{1}{r^{2s+\beta}}H_{|m|}[\rho^{2s+\beta-2}](1).
\end{align*}\\
The constants $H_{|m|}[r^{-\beta}](1)$ is finite by $-1/2<\beta<|m|+2$. The constant $H_{|m|}[\rho^{2s+\beta-2}](1)$ is finite by $-|m|-2s<\beta<5/2-2s$. By applying \eqref{eq: FOH}, 

\begin{align*}
\mu_{|m|}(s,\beta)=H_{|m|}[\rho^{2s+\beta-2}](1)H_{|m|}[r^{n}](1)
&=\frac{2^{(2s+\beta-2)+1}\Gamma(\frac{|m|+(2s+\beta-2)+2}{2})}{\Gamma(\frac{|m|-(2s+\beta-2)}{2})}\frac{2^{-\beta+1}\Gamma(\frac{|m|-\beta+2}{2})}{\Gamma(\frac{|m|+\beta}{2})} \\
&=2^{2s}\frac{\Gamma(\frac{|m|+\beta}{2}+s )\Gamma(\frac{|m|-\beta}{2}+1 )}{\Gamma(\frac{|m|-\beta}{2}+1-s )\Gamma(\frac{|m|+\beta}{2})}.
\end{align*}\\
Thus \eqref{eq: EV} holds for $-1/2<\beta<\min\{5/2-2s,2\}$ and $|m|>\max\{\beta-2,-\beta-2s\}$. The formula \eqref{eq: EV} is extendable for all $0<s<1$, $\beta<2$, and $m\in \mathbb{Z}$.
\end{proof}

\vspace{5pt}

\begin{rem}[The kernel]
The kernel of $(-\Delta)^{s}$ for homogeneous functions are of the form

\begin{equation}
\begin{aligned}
r^{-|m|-2k-2+2s}e^{im\theta}, \quad
r^{|m|+2k}e^{im\theta},\quad m\in \mathbb{Z},\ k\in \mathbb{N}_0.
\end{aligned}
\label{eq: Ker}
\end{equation}\\
The integers of the first eigenfunctions should satisfy the condition $|m|+2k+2s<0$ so that the eigenfunctions are locally integrable in $\mathbb{R}^{2}$. Indeed, the condition $(-\Delta)^{s}(r^{-\beta}e^{im\theta})=0$ implies that $\mu_{|m|}(s,\beta)=0$ and the degrees $\beta$ are given by the zeros \eqref{eq: Zb}. 
\end{rem}

\section{The non-local operators on high frequency spaces}\label{s:3}

We define the $2s$-th order non-local differential operator $\mathcal{L}(s,\beta)$ for $0<s<1$ by using the constant $\mu_{|m|}(s,\beta)$ of the polar-mode formula \eqref{eq: C}. For non-positive integers $m_0\in \mathbb{N}$ and $-m_0-2s<\beta<m_0+2$, the constant $\mu_{|m|}(s,\beta)$ is finite for all $|m|\geq m_0$ by Lemma \ref{l: intervals} and we define $\mathcal{L}(s,\beta)$ by the Fourier expansion for functions with vanishing Fourier modes for $|m|<m_0$.

\subsection{The high frequency Sobolev space}

We use the Fourier expansion

\begin{align*}
f(\theta)=\sum_{|m|\geq 0}\hat{f}_{m}e^{im\theta},\quad \hat{f}_{m}=\frac{1}{2\pi}\int_{0}^{2\pi}f(\theta)e^{-im\theta}d\theta,
\end{align*}\\
and define the $H^l$ scalar product between $f$ and $g$ by

\begin{align*}
(f,g)_{H^{l}}=2\pi \sum_{|m|\geq 0}\langle m\rangle ^{2l}\hat{f}_{m}\overline{\hat{g}_{m}},\quad l\in \mathbb{R},
\end{align*}\\
and the associated space of high-frequency functions 

\begin{align*}
H^{l}_{m_0}(\mathbb{T})=\left\{f\in H^{l}(\mathbb{T})\ \middle|\ f=\sum_{|m|\geq m_0}\hat{f}_{m}e^{im\theta}\ \right\},\quad m_0\in \mathbb{N}\cup \{0\}.
\end{align*}

\begin{prop}\label{p: betam_0}
Let $0<s<1$ and $m_0\in \mathbb{N}_0$. 

\noindent
(i) Assume that 

\begin{align}
-2s-m_0<\beta<m_0+2.   \label{eq: finitebm_0}
\end{align}\\
Then, $\beta$ belongs to the interval \eqref{eq: finite} for all $|m|\geq m_0$. In particular, $\mu_{|m|}(s,\beta)$ is finite for all $|m|\geq m_0$.\\
(ii) Assume that 

\begin{align}
-m_0<\beta<m_0+2-2s.   \label{eq: positivebm_0}
\end{align}\\
Then, $\beta$ belongs to the interval \eqref{eq: positive} for all $|m|\geq m_0$. In particular, $\mu_{|m|}(s,\beta)$ is positive for all $|m|\geq m_0$.\\
\end{prop}

\begin{proof}
This follows from Lemma \eqref{l: intervals}.
\end{proof}

\vspace{5pt}

We show that the non-positive constants $\mu_{|m|}(s,\beta)$ are at most two for $|m|\geq m_0$.

\vspace{5pt}

\begin{prop}\label{p: Symbolsign}
Let $0<s<1$ and $m_0\in \mathbb{N}_0$. Let $\beta$ satisfy \eqref{eq: finitebm_0}. Then, the following holds:\\
 
 \begin{itemize}
 \item[(i)] If $\beta$ satisfies \eqref{eq: positivebm_0}, $\mu_{m_0}(s,\beta)>0$\\
\item[(ii)] If $\beta$ does not satisfy \eqref{eq: positivebm_0}, the following holds:\\
 \begin{itemize}
 \item[(a)] For $0<s\leq \frac{1}{2}$, $\mu_{m_0}(s,\beta)\leq 0<\mu_{m_0+1}(s,\beta)$ \\
 \item[(b)] For $\frac{1}{2}<s<1$, $\mu_{m_0}(s,\beta)\leq 0<\mu_{m_0+1}(s,\beta)$ or $\mu_{m_0}(s,\beta)<\mu_{m_0+1}(s,\beta)\leq 0$
 \end{itemize}
 \end{itemize}
 \end{prop}

\vspace{5pt}

\begin{proof}
If $\beta$ satisfies \eqref{eq: positivebm_0}, $\mu_{m_0}(s,\beta)>0$ by \eqref{eq: Sign}. If $\beta$ does not satisfy \eqref{eq: positivebm_0}, $-m_0-2s<\beta\leq -m_0$ or $m_0+2-2s\leq \beta<m_0+2$ and $\mu_{m_0}(s,\beta)\leq 0$ by \eqref{eq: Sign}. We consider the case $-m_0-2s<\beta\leq -m_0$. The case $m_0+2-2s\leq \beta<m_0+2$ is parallel. The function $\mu_{|m|}(s,\beta)$ is increasing for $|m|>-\beta-2s$ and vanishes at $|m|=-\beta$. Since $m_0$ is in $(-\beta-2s,-\beta]$, $-\beta<m_0+2$ and $\mu_{m_0+2}(s,\beta)>0$. For $0<s\leq 1/2$, $-\beta<m_0+1$ and $\mu_{m_0+1}(s,\beta)>0$. 
\end{proof}

\vspace{5pt}

\subsection{The non-local $2s$-th order differential operator}

We now define the operator $\mathcal{L}(s,\beta)$ on high-frequency spaces by the constant $\mu_{|m|}(s,\beta)$ and the Fourier expansion.

\begin{defn}[The non-local operators on the torus]
Let $0<s<1$ and $m_0\in \mathbb{N}_0$.  

\noindent
(i) Let $\beta$ satisfy \eqref{eq: finitebm_0}. We set the operator $\mathcal{L}(s,\beta)$ on $H^{2s}_{m_0}(\mathbb{T})$ by

\begin{align}
\mathcal{L}(s,\beta)w= \sum_{|m|\geq m_0}\mu_{|m|}(s,\beta)\hat{w}_{m}e^{im\theta},\quad \hat{w}_m=\frac{1}{2\pi}\int_{\mathbb{T}}w(\theta)e^{-im\theta}d\theta.   \label{eq: OPL}
\end{align}\\
(ii) Let $\beta$ satisfy \eqref{eq: positivebm_0}. We set the inverse operator $\mathcal{L}(s,\beta)^{-1}$ on $L^{2}_{m_0}(\mathbb{T})$ by 

\begin{align}
\mathcal{L}(s,\beta)^{-1}g= \sum_{|m|\geq m_0}\mu_{|m|}(s,\beta)^{-1}\hat{g}_{m}e^{im\theta},\quad \hat{g}_m=\frac{1}{2\pi}\int_{\mathbb{T}}g(\theta)e^{-im\theta}d\theta.   \label{eq: OPLI}
\end{align}
\end{defn}

\begin{rems}\label{r: SP}
(i) The operator $\mathcal{L}(s,\beta)$ agrees with the operator $r^{\beta+2s}(-\Delta)^{s}\!\bigl(r^{-\beta} \cdot \bigr)$ on $H^{2s}_{m_0}(\mathbb{T})$. Indeed, For $w=e^{im\theta}$ and $|m|\geq m_0$, the identity \eqref{eq: EV} for $\beta<2$ implies 

\begin{align*}
\mathcal{L}(s,\beta)e^{im\theta}
=\mu_{|m|}(s,\beta)e^{im\theta}
=r^{\beta+2s}(-\Delta)^{s}\!\bigl(r^{-\beta} e^{i m \theta}\bigr).
\end{align*}\\
(ii) The operator $\mathcal{L}(s,\beta)$ agrees with the composition operator $-L(\beta)\mathcal{K}(s,\beta)$ for $m_0=0$ and $-2s<\beta<0$ where $L(\beta)=\partial_{\theta}^{2}+\beta^{2}$ and $\mathcal{K}(s,\beta)f=K(\cdot, s,\beta)*f$ for the kernel 

\begin{equation}
\begin{aligned}
K(\theta;s,\beta)=C(1-s)\int_{0}^{\infty}\frac{d \rho}{(\rho^{2}+1-2\rho\cos\theta )^{s}\rho^{\beta+1}}.
\end{aligned}
\label{eq:Kernel}
\end{equation}\\
Indeed, by using $(-\Delta)^{s}=(-\Delta)^{-(1-s)}(-\Delta)$ and $(-\Delta) (r^{-\beta}e^{im\theta})=(m^{2}-\beta^{2})r^{-\beta-2}e^{im\theta}$, 

\begin{align*}
\mathcal{L}(s,\beta)e^{im\theta}
&=
r^{\beta+2s}(-\Delta)^{s}(r^{-\beta}e^{im\theta}) \\
&=(m^{2}-\beta^{2})r^{\beta+2s}(-\Delta)^{-(1-s)}\left(r^{-\beta-2} e^{im\theta} \right) \\
&=(m^{2}-\beta^{2})r^{\beta+2s}C(1-s)\int_{\mathbb{R}^{2}}\frac{e^{im \theta_y}}{|x-y|^{2s}|y|^{\beta+2} }dy\\
&=(m^{2}-\beta^{2})\int_{0}^{2\pi}\left(C(1-s)\int_{0}^{\infty}\frac{d \rho}{(\rho^{2}+1-2\rho \cos(\theta-\theta'))^{s}\rho^{\beta+1} }\right)e^{im \theta'}d \theta' \\
&=(m^{2}-\beta^{2}) \mathcal{K}(s,\beta)e^{im\theta} 
=-L(\beta)\mathcal{K}(s,\beta)e^{im\theta}.
\end{align*}\\
This means that the symbol $K_{|m|}(s,\beta)$ of the operator $\mathcal{K}(s,\beta)$ satisfies

\begin{align}
\mu_{|m|}(s,\beta)=(m^{2}-\beta^{2})K_{|m|}(s,\beta).   \label{eq: muK}
\end{align}\\
By using the identity \eqref{eq: kappa},

\begin{align}
K_{|m|}(s,\beta)=\frac{\kappa_{|m|}(s,\beta)(|m|-\beta+2-2s)}{|m|-\beta}.   \label{eq: Km}
\end{align}\\
The identity \eqref{eq: muK} holds for $m$ and $\beta$ satisfying \eqref{eq: finite} though the kernel $K(\theta; s,\beta)$ is available under the condition $-2s<\beta<0$.\\ 
(iii) The operator $\mathcal{L}(s,\beta)^{-1}$ agrees with $r^{\beta}(-\Delta)^{-s}(r^{-\beta-2s}\cdot )$. Indeed, For $w=e^{im\theta}$ and $|m|\geq m_0$, the identity \eqref{eq: IEV} for $\beta<2-2s$ implies 

\begin{align*}
\mathcal{L}(s,\beta)^{-1}e^{im\theta}
=\mu_{|m|}(s,\beta)^{-1}e^{im\theta}
=r^{\beta}(-\Delta)^{-s}(r^{-\beta-2s}e^{i m \theta}).
\end{align*}\\
(iv) The operator $\mathcal{L}(s,\beta)^{-1}$ agrees with $\mathcal{K}(1-s,\beta+2s-2)$ for $m_0=0$ and $0<\beta<2-2s$. Indeed,

\begin{align*}
\mathcal{L}(s,\beta)^{-1}e^{im\theta}
&=
r^{\beta}(-\Delta)^{-s}(r^{-\beta-2s}e^{im\theta} ) \\
&=r^{\beta}C(s)\int_{\mathbb{R}^{2}}\frac{e^{im \theta_y}}{|x-y|^{2-2s}|y|^{\beta+2s} }dy\\
&=\int_{0}^{2\pi}\left(C(s)\int_{0}^{\infty}\frac{d \rho}{(\rho^{2}+1-2\rho \cos(\theta-\theta'))^{1-s}\rho^{\beta+2s-1} }\right)e^{im \theta'}d \theta' 
=\mathcal{K}(1-s,\beta+2-2s)e^{im\theta}.
\end{align*}\\
By the identity \eqref{eq: muK}, 

\begin{align}
K_{|m|}(1-s,\beta+2-2s)=\frac{1}{(m^{2}-\beta^{2})K_{|m|}(s,\beta)}.  \label{eq: Kidentity}
\end{align}
\end{rems}

\subsection{The bilinear form} 

We also prepare the bilinear form estimates associated with the operator \eqref{eq: OPL}. We show that the bilinear form $B(w,w)$ consists of the $H^{s}$-norm and lower order norms.

\begin{lem}\label{l:Bilinear}
Let $0<s<1$ and $m_0\in \mathbb{N}_{0}$. Let $\beta$ satisfy \eqref{eq: finitebm_0}. Set  

\begin{align}
B(w,\eta)=2\pi  \sum_{|m|\geq m_0}\mu_{|m|}(s,\beta)\hat{w}_m\overline{\hat{\eta}_m},\quad w,\eta\in H^{s}_{m_0}(\mathbb{T}). \label{eq:Bilinear}
\end{align}\\
Then, $B(\cdot,\cdot): H^{s}_{m_0}(\mathbb{T})\times H^{s}_{m_0}(\mathbb{T})\to \mathbb{C}$ is a bounded operator. Moreover, for the positive constant $\kappa_{|m|}(s,\beta)$ in \eqref{eq: kappa},   

\begin{equation}
\begin{aligned}
||w||_{H^{s}}^{2}&\lesssim B(w,w)-2\pi (1-s)\sum_{|m|\geq m_0}\kappa_{|m|}(s,\beta)|m| |\hat{w}_m|^{2} \\
&-2\pi (-\beta+1-s)\beta \sum_{|m|\geq m_0}\kappa_{|m|}(s,\beta)|\hat{w}_m|^{2}
\lesssim ||w||_{H^{s}}^{2},\quad w\in H^{s}_{m_0}(\mathbb{T}),  
\end{aligned}
\label{eq:bilinear1}
\end{equation}
\begin{align}
 ||w||_{H^{-\frac{1}{2}+s}}^{2}\lesssim
\sum_{|m|\geq m_0}\kappa_{|m|}(s,\beta)(1+|m|) |\hat{w}_m|^{2} \lesssim ||w||_{H^{-\frac{1}{2}+s}}^{2}, \quad w\in H^{-\frac{1}{2}+s}_{m_0}(\mathbb{T}).  \label{eq:bilinear2} 
\end{align}
\end{lem}

\vspace{5pt}

\begin{proof}
By the asymptotics of the symbol \eqref{eq: A1}, $|\mu_{|m|}(s,\beta)|\lesssim  \langle m\rangle^{2s}$ for all $|m|\geq m_0$ and 

\begin{align*}
|B(w,\eta)|\leq 2\pi  \sum_{|m|\geq m_0}\mu_{|m|}(s,\beta)\langle m\rangle^{2}|\hat{w}_m| |\hat{\eta}_m|
\lesssim \sum_{|m|\geq m_0} \langle m\rangle^{2s} |\hat{w}_m| |\hat{\eta}_m|
\lesssim  ||w||_{H^{s}}||\eta||_{H^{s}}.
\end{align*}\\
By the asymptotics \eqref{eq: A1}, 

\begin{align*}
\langle m\rangle^{2s-2} \lesssim 
\kappa_{|m|}(s,\beta)\lesssim \langle m\rangle^{2s-2},\quad |m|\geq m_0,
\end{align*}\\
and the estimates \eqref{eq:bilinear1} and \eqref{eq:bilinear2} follow.
\end{proof}

\vspace{5pt}

\section{Linear regularity estimates}\label{s:4}

From this point on, we fix parameters $m_0\in \mathbb{N}_0$, $0<s<1$, and $\beta$ satisfying \eqref{eq: finitebm_0}, and consider solutions $w\in H^{s}_{m_0}(\mathbb{T})$ to the linear $2s$-th order differential equation

\begin{align}
\mathcal{L}(s,\beta)w=g,  \label{eq:Linear}
\end{align}\\
for $g\in H^{-s}_{m_0}(\mathbb{T})$ by using the bilinear form $B(\cdot,\cdot)$ in Lemma \ref{l:Bilinear}. We establish both $L^{p}$ and H\"older regularity estimates to \eqref{eq:Linear} by using the pointwise symbol estimtes \eqref{eq: A1} and \eqref{eq: A2}. We show $L^{p}$ estimates by using a Fourier multiplier theorem and H\"older estimates by combining kernel estimates and a Fourier multiplier theorem.

\vspace{5pt}

\subsection{The Fourier multiplier theorem}

Let ${\mathcal{S}}(\mathbb{R})$ be a space of rapidly decaying functions on $\mathbb{R}$. For $f\in {\mathcal{S}}(\mathbb{R})$, we set the Fourier transform  

\begin{align*}
(\mathcal{F}f)(\xi)=\hat{f}(\xi)=\frac{1}{\sqrt{2\pi}}\int_{\mathbb{R}}f(x)e^{-ix\xi}d x,\quad \xi\in \mathbb{R},
\end{align*}\\
and the inverse Fourier transform $(\mathcal{F}^{*}f)(x)= \check{f}(x)=({\mathcal{F}}f)(-x)$. We say that $a\in L^{\infty}(\mathbb{R})$ is an $L^{p}$ multiplier if the operator $f\longmapsto (a\hat{f})\ \check{}$ is a bounded operator on $L^{p}(\mathbb{R})$ and denote the space of all $L^{p}$ multipliers on $L^{p}(\mathbb{R})$ by $\mathcal{M}_p(\mathbb{R})$. The following Lemma \ref{l:FMMT} (i) is a weaker version of the sufficient condition $|a'(\xi)|\leq C/|\xi|$ for $\xi\in \mathbb{R}\backslash \{0\}$ to guarantee that $a$ is an $L^{p}$ multiplier for all $1<p<\infty$ \cite[Theorem 6.2.7]{Grafakos}. 

Likewise, a sequence $\{a_m\}_{m\in \mathbb{Z}}\in l^{\infty}(\mathbb{Z})$ is an $L^{p}$ multiplier if the operator $f\longmapsto (a_m\hat{f}_m)^{\lor}=\sum_{m\in \mathbb{Z}}a_m \hat{f}_me^{im\theta}$ is a bounded operator on $L^{p}(\mathbb{T})$ and we denote the space of all $L^{p}$ multipliers on $L^{p}(\mathbb{T})$ by $\mathcal{M}_p(\mathbb{Z})$. The following Lemma \ref{l:FMMT} (ii) states that a sequence $\{a_m\}_{m\in \mathbb{Z}}$ belongs to $\mathcal{M}_p(\mathbb{Z})$ if it admits a continuous extension belonging to $\mathcal{M}_p(\mathbb{R})$ \cite[Theorem 4.3.7]{Grafakos}.

\vspace{5pt}

\begin{lem}\label{l:FMMT}
(i) Let $a(\xi)$ be a complex-valued bounded function in $\mathbb{R}\backslash \{0\}$. Assume that there exists $A>0$ such that 

\begin{align*}
\int_{R\leq |\xi|\leq 2R}|a'(\xi)|^{2}d \xi \leq \frac{A}{R},\quad R>0.
\end{align*} \\
Then, $a\in \mathcal{M}_p(\mathbb{R})$ for all $p\in (1,\infty)$ and $||a||_{\mathcal{M}_p(\mathbb{R})}\lesssim \max\left\{p,(p-1)^{-1}\right\}(A+||a||_{L^{\infty}(\mathbb{R}) })$.

\noindent
(ii) Assume that $a(t)\in \mathcal{M}_{p}(\mathbb{R})$ is continuous at every points $m\in \mathbb{Z}$. Then, $\{a(m)\}_{m\in \mathbb{Z}}\in \mathcal{M}_{p}(\mathbb{Z})$ and $||\{a(m)\}_{m\in \mathbb{Z}}\}||_{\mathcal{M}_{p}(\mathbb{Z})}\leq ||a||_{\mathcal{M}_{p}(\mathbb{R})}$.
\end{lem}

\vspace{5pt}

\subsection{$L^{p}$ estimates}
We establish $L^{p}$ regularity estimates for \eqref{eq:Linear} on the Bessel potential space \cite{Triebel1}, \cite{BO13}

\begin{align*}
H^{l,p}(\mathbb{T})=\left\{f\in L^{p}(\mathbb{T})\ \middle|\ ||f||_{H^{l,p}(\mathbb{T})}=||(\langle m \rangle^{l}\hat{f}_{m})^{\lor}||_{L^{p}(\mathbb{T})}<\infty \right\},\quad 1<p<\infty,\quad l> 0.
\end{align*}

\vspace{5pt}

\begin{thm}\label{t:LP}
Let $m_0\in \mathbb{N}_0$ and $0<s<1$. Let $\beta$ satisfy \eqref{eq: finitebm_0}. Let  $1<p<\infty$. Let $w\in H^{s}_{m_0}(\mathbb{T})$ be a solution to \eqref{eq:Linear} for $g\in H^{-s}_{m_0}(\mathbb{T})$. Assume that $w\in L^{2}(\mathbb{T})$ and $g\in L^{p}(\mathbb{T})$. Then, $w\in H^{2s,p}(\mathbb{T})$ and 

\begin{align}
||w||_{H^{2s,p}(\mathbb{T})}\leq C\left(||w||_{L^{2}}+||g||_{L^{p}(\mathbb{T})} \right)  \label{eq:LP}
\end{align}\\
holds for some constant $C$, independent of $w$.
\end{thm}

\vspace{5pt}

\begin{proof}
We first show the a priori estimate \eqref{eq:LP} for a smooth solution. By \eqref{eq:Bilinear}, the Fourier coefficients of $w$ and $g$ satisfy 

\begin{align*}
\mu_{|m|}(s,\beta)\hat{w}_{m}=\hat{g}_{m},\quad |m|\geq m_0.
\end{align*}\\
{For $|m|\geq m_0+2$, $\mu_{|m|}(s,\beta)> 0$ and 
\begin{align*}
(\langle m \rangle^{2s}\hat{w}_{m})^{\lor}
=\sum_{|m|\geq m_0}\langle m \rangle^{2s}\hat{w}_{m}e^{im\theta} 
&=\sum_{m_0+1\geq |m|\geq m_0}\langle m \rangle^{2s}\hat{w}_{m}e^{im\theta}
+\sum_{|m|\geq m_0+2}\langle m \rangle^{2s}\hat{w}_{m}e^{im\theta} \\
&=\sum_{m_0+1\geq |m|\geq m_0}\langle m \rangle^{2s}\hat{w}_{m}e^{im\theta}
+\sum_{|m|\geq m_0+2}\langle m \rangle^{2s}\mu_{|m|}(s,\beta)^{-1} \hat{g}_{m}e^{im\theta}.
\end{align*}\\
We estimate $L^{p}$-norms of the first two terms by the $L^{2}$-norm of $w$. We take an even $C^{1}$-function $a(t)$ on $\mathbb{R}$ such that 
\begin{align*}
a(t)=
\begin{cases}
\ \displaystyle \langle t \rangle^{2s} \mu_{t}(s,\beta)^{-1},\quad & t\geq m_0+2, \\
\ 0,\quad & 0 \le t \le m_0+1. 
\end{cases}
\end{align*}\\}
By the asymptotics \eqref{eq: A1} and \eqref{eq: A2}, $\mu_{t}(s,\beta)=O(|t|^{2s})$ and $\partial_t \mu_{t}(s,\beta)=O(|t|^{2s+1})$ as $|t|\to\infty$. Thus, $a(t) \in \mathcal{M}_{p}(\mathbb{R})$. By Lemma \ref{l:FMMT}, $\{a(m)\}_{m\in \mathbb{Z}}\in \mathcal{M}_{p}(\mathbb{Z})$ and \eqref{eq:LP} holds. 

For a solution $w\in H^{s}_{m_0}(\mathbb{T})$ with $g\in H^{-s}_{m_0}(\mathbb{T})$, the partial sums of Fourier series $w^{(N)}=\sum_{N\geq |m|\geq m_0}\hat{w}_{m}e^{im\theta}$ and $g^{(N)}=\sum_{N\geq |m|\geq m_0}\hat{g}_{m}e^{im\theta}$ is a smooth solution to $\mathcal{L}(s,\beta)w^{(N)}=g^{(N)}$. Since $w^{(N)}\to w$ in $L^{2}(\mathbb{T})$ and $g^{(N)}\to g$ in $L^{p}(\mathbb{T})$, e.g., \cite[Theorem 4.3.14]{Grafakos}, applying \eqref{eq:LP} yields $w^{(N)}\to w$ in $H^{2s,p}(\mathbb{T})$ and \eqref{eq:LP} holds for $w$ and $g$.
\end{proof}

\subsection{H\"older estimates}

We define the spaces of H\"older continuous functions

\begin{align*}
C^{\nu}(\mathbb{T})&=\{f\in L^{\infty}(\mathbb{T})\ |\ ||f||_{C^{\nu}(\mathbb{T})}<\infty\  \},\quad 0<\nu<1, \\
||f||_{C^{\nu}(\mathbb{T})}&=||f||_{L^{\infty}(\mathbb{T})}+[f]^{(\nu)}_{\mathbb{T}},\quad 
[f]^{(\nu)}_{\mathbb{T}}=\sup_{t>0}\frac{||f(\cdot +t)-f(\cdot)||_{L^{\infty}(\mathbb{T})} }{t^{\nu}},\\
C^{k+\nu}(\mathbb{T})&=\{f\in L^{\infty}(\mathbb{T})\ |\ ||f||_{C^{k+\nu}(\mathbb{T})}<\infty \},\quad k\in \mathbb{N}, \\
||f||_{C^{k+\nu}(\mathbb{T})}&=||f||_{C^{k}(\mathbb{T})}+\sum_{|\alpha|=k}[\partial^{\alpha}f ]^{(\nu)}_{\mathbb{T}}.
\end{align*}\\
The inverse operator $\mathcal{L}(s,\bar{\beta})^{-1}$ agrees with the convolution operator $\mathcal{K}(1-s,\bar{\beta}+2s-2)$ for $m_0\geq 0$ and $0<\bar{\beta}<2-2s$ in Remarks \ref{r: SP} (iv). The following estimates show that the kernel $K(\theta; 1-s,\bar{\beta}+2s-2)$ of $\mathcal{K}(1-s,\bar{\beta}+2s-2)$ for $0<s<1/2$ has the same singularity as that of the one-dimensional Riesz potential ($n=1$):
\begin{align*}
\psi=(-\Delta)^{-s}\omega=C(n,s)\frac{1}{|x|^{n-2s}}*\omega,\quad C(n,s)=\frac{\displaystyle\Gamma\left(\frac{n}{2}-s\right)}{4^{s}\pi^{\frac{n}{2}} \Gamma(s)}.
\end{align*}\\
For $s=1/2$, the kernel $K(\theta; 1-s,\bar{\beta}+2s-2)$ has a logarithmic singularity, and this was already obtained in \cite[Lemmas 4.11 and 4.12]{CCG20}. For $1/2<s<1$, the kernel $K(\theta; 1-s,\bar{\beta}+2s-2)$ is H\"older continuous of the exponent $2s-1$.

\vspace{5pt}

\begin{prop}
Let $0<s<1$ and $0<\bar{\beta}<2-2s$. The following holds for $|\theta|\leq \pi/2$: 

\begin{align}
K(\theta; 1-s,\bar{\beta}+2s-2)&\lesssim 
\begin{cases}
\ \displaystyle\frac{1}{|\theta|^{1-2s}},\quad &0<s<\displaystyle\frac{1}{2},\\
\ -\log{|\theta|}+C,\quad &s=\displaystyle\frac{1}{2},\quad 
\end{cases}  \label{eq:PK1}\\
|K^{(m)}(\theta; 1-s,\bar{\beta}+2s-2)|&\lesssim \frac{1}{|\theta|^{1-2s+m}},\quad 0<s<1,\quad  m\in \mathbb{N}. \label{eq:PK2}
\end{align}
\end{prop}

\vspace{5pt}

\begin{proof}
We set $t=1-\cos\theta$ and 

\begin{align*}
K(\theta; 1-s,\bar{\beta}+2s-2)=\int_{0}^{\infty}\frac{d \rho}{( (\rho-1)^{2}+2t\rho )^{1-s}\rho^{\bar{\beta}+2s-1} }
=:f(t). 
\end{align*}\\
Since $\lim_{\theta\to 0}t/\theta^{2}=1/2$, it suffices to show the following estimates for $0<t\leq 1$:

\begin{align}
f(t)&\lesssim 
\begin{cases}
\ \displaystyle\frac{1}{t^{\frac{1}{2} -s} }\quad 0<s<\frac{1}{2},\\
\ -\log{t}+C\quad s=\displaystyle\frac{1}{2},
\end{cases}\label{eq:PF1}\\
|f^{(m)}(t)|&\lesssim \frac{1}{t^{\frac{3}{2}-s}},\quad 0<s<1,\quad m\in \mathbb{N}. \label{eq:PF2}
\end{align}\\
We set 

\begin{align*}
f(t)=\int_{0}^{\infty}\frac{d \rho}{( (\rho-1)^{2}+2t\rho )^{1-s}\rho^{\bar{\beta}+2s-1} }=\int_{0}^{1/2}+\int_{1/2}^{3/2}+\int_{3/2}^{\infty}
=:f_1(t)+f_2(t)+f_3(t).
\end{align*}\\
The functions $f_1$ and $f_3$ are bounded for $0\leq t\leq 1$. By changing the variable by $\eta=t^{-1/2}(\rho-1)$,

\begin{align*}
f_2(t)=\int_{1/2}^{3/2}\frac{d \rho}{( (\rho-1)^{2}+2t\rho )^{1-s}\rho^{\bar{\beta}+2s-1} }
\lesssim \int_{1/2}^{3/2}\frac{d \rho}{( (\rho-1)^{2}+t)^{1-s}} 
\lesssim \frac{1}{t^{1/2-s}}\int_{0}^{1/(2\sqrt{t})}\frac{d \rho}{( \eta^{2}+1)^{1-s}}. 
\end{align*}\\
The function $(\eta^{2}+1)^{-1+s}$ for $0<s<1/2$ is integrable in $(0,\infty)$. For $s=1/2$, 

\begin{align*}
\int_{0}^{1/(2\sqrt{t})}\frac{d \rho}{\sqrt{ \eta^{2}+1}}
=\left[\log(\eta+\sqrt{\eta^{2}+1}) \right]_{0}^{1/(2\sqrt{t})}
=\log\left(\frac{1}{2}+\sqrt{\frac{1}{4}+t } \right)-\frac{1}{2}\log{t}.
\end{align*}\\
We demonstrated \eqref{eq:PF1}. The functions $f_1'$ and $f_3'$ are bounded for $0\leq t\leq 1$ and 

\begin{align*}
f_2'(t)=2(s-1) \int_{1/2}^{3/2}\frac{d \rho}{( (\rho-1)^{2}+2t\rho )^{2-s}\rho^{\bar{\beta}+2s-2} }.
\end{align*}\\
By changing the variable by $\eta=t^{-1/2}(\rho-1)$,

\begin{align*}
|f_2'(t)|\lesssim  \int_{1/2}^{3/2}\frac{d \rho}{( (\rho-1)^{2}+t )^{2-s}}
\lesssim  \frac{1}{t^{3/2-s}}\int_{0}^{1/(2\sqrt{t})}\frac{d \eta}{( \eta^{2}+1 )^{2-s}}
\leq  \frac{1}{t^{3/2-s}}\int_{0}^{\infty}\frac{d \eta}{( \eta^{2}+1 )^{2-s}}. 
\end{align*}\\
We obtained \eqref{eq:PF2} for $m=1$. The case $m\geq 2$ is similar.
\end{proof}

\vspace{5pt}

We show the following H\"older regularity estimates for the non-local $2s$-th order differential operator $\mathcal{L}(s,\beta)$ involving the parameter $\beta$, cf. \cite[Proposition 2.8]{Sil}, \cite[Lemma 1.8.3]{Ros24} for the case of the fractional Laplace operator.

\vspace{5pt}

\begin{thm}\label{t:Holder}
Let $m_0\in \mathbb{N}_0$ and $0<s<1$. Let $\beta$ satisfy \eqref{eq: finitebm_0}. Let $0<\nu<1$ satisfy $2s+\nu\notin \mathbb{N}$. Let $w\in H^{s}_{m_0}(\mathbb{T})$ be a solution to \eqref{eq:Linear} for $g\in H^{-s}_{m_0}(\mathbb{T})$. Assume that $w\in L^{\infty}(\mathbb{T})$ and $g\in C^{\nu}(\mathbb{T})$. Then, $w\in C^{2s+\nu}(\mathbb{T})$ and 

\begin{align}
||w||_{C^{2s+\nu}(\mathbb{T})}\leq C\left(||w||_{L^{\infty}}+||g||_{C^{\nu}(\mathbb{T})} \right)   \label{eq:Holder}
\end{align}\\
holds for some constant $C$, independent of $w$.
\end{thm}

\vspace{5pt}

\begin{prop}  \label{p:KernelHolder}
Let $0<s<1$, $0<{\beta}_0<2-2s$, and $0<\nu<1$ such that $2s+\nu\notin \mathbb{N}$. Then, the estimate \eqref{eq:Holder} holds for $w=\mathcal{L}(s,{\beta}_0)^{-1}g$ and $g\in C^{\nu}(\mathbb{T})$. 
\end{prop}

\vspace{5pt}

\begin{proof}
The estimate \eqref{eq:Holder} follows from the kernel estimates \eqref{eq:PK1} and \eqref{eq:PK2} and the same potential estimate argument for the case of the Riesz potential \cite[Lemma 1.8.3]{Ros24}. 
\end{proof}

\vspace{5pt}

\begin{prop}\label{p:SymbolHolder}
Let $m_0\in \mathbb{N}_0$, $0<s<1$, and $-m_0<{\beta}_1< m_0+2-2s$. Let $0<\nu<1$ satisfy $2s+\nu\notin \mathbb{N}$. Then, the estimate \eqref{eq:Holder} holds for $w=\mathcal{L}(s,{\beta}_1)^{-1}g$ and $g\in H^{-s}_{m_0}\cap C^{\nu}(\mathbb{T})$. 
\end{prop}

\vspace{5pt}

\begin{proof}
We take $0<{\beta}_0<2-2s$ and set

\begin{align*}
w&=\mathcal{L}(s,{\beta}_0)^{-1}g+\left(\mathcal{L}(s,{\beta}_1)^{-1}-\mathcal{L}(s,{\beta}_0)^{-1}\right)g\\
&=\mathcal{L}(s,{\beta}_0)^{-1}g+ \sum_{|m|\geq m_0}\left( \frac{1}{\mu_{|m|}(s,\beta_1)}-\frac{1}{\mu_{|m|}(s,\beta_0)}\right)\hat{g}_me^{im\theta}\\
&=\mathcal{L}(s,{\beta}_0)^{-1}g+ \sum_{|m|\geq m_0}\left( \frac{1}{\mu_{|m|}(s,\beta_1)\mu_{|m|}(s,\beta_0)}\int_{\beta_1}^{\beta_0}\frac{d}{d\beta}\mu_{|m|}(s,\beta)d\beta \right)\hat{g}_me^{im\theta}
=:w_1+w_2.
\end{align*}\\
The function $w_1$ satisfies the desired estimate by Proposition \ref{p:KernelHolder}. By the asymptotics \eqref{eq: A1}, \eqref{eq: A3}, and \eqref{eq: A4}, the sequence 

\begin{align*}
\left\{\langle m \rangle^{2s+1}\left( \frac{1}{\mu_{|m|}(s,\beta_1)\mu_{|m|}(s,\beta_0)}\int_{\beta_1}^{\beta_0}\frac{d}{d\beta}\mu_{|m|}(s,\beta)d\beta \right)\right\}
\end{align*}\\
is an $L^{p}$ multiplier on $L^{p}(\mathbb{Z})$. Thus 

\begin{align*}
||w_2||_{H^{2s+1,p}(\mathbb{T})}=\left\| \left(\langle m \rangle^{2s+1}\left( \frac{1}{\mu_{|m|}(s,\beta_1)\mu_{|m|}(s,\beta_0)}\int_{\beta_1}^{\beta_0}\frac{d}{d\beta}\mu_{|m|}(s,\beta)d\beta \right)  \hat{g}_{m} \right)^{\lor}  \right\|_{L^{p}(\mathbb{T})}\lesssim ||g||_{L^{p}(\mathbb{T})}\lesssim ||g||_{L^{\infty}(\mathbb{T})}.
\end{align*}\\
We take large $1<p<\infty$ so that $2s+1-1/p>2s+\nu$ and $H^{2s+1,p}(\mathbb{T})\subset C^{2s+1-\frac{1}{p}}(\mathbb{T})\subset C^{2s+\nu }(\mathbb{T})$ by the Sobolev embedding \eqref{eq:S3}. Thus, $w_2$ also satisfies the desired estimate. 
\end{proof}

\vspace{5pt}
	
\begin{proof}[Proof of Theorem \ref{t:Holder}]
We take an integer $m_1>m_0$ such that $-m_1<\beta< m_1+2-2s$. For a solution $w\in H^{s}_{m_0}(\mathbb{T})$ to \eqref{eq:Linear} for $g \in H^{-s}_{m_{0}}(\mathbb{T})$, we set 

\begin{align*}
w&=\sum_{m_1\geq |m|\geq m_0}\hat{w}_me^{im\theta}+\sum_{|m|\geq m_1}\hat{w}_me^{im\theta}=w_1+w_2,\\
g&=\sum_{m_1\geq |m|\geq m_0}\hat{g}_me^{im\theta}+\sum_{|m|\geq m_1}\hat{g}_me^{im\theta}=g_1+g_2.
\end{align*}\\
Since $w_1$ and $g_1$ are with finite Fourier modes, we have 
 
 \begin{align*}
 ||w_1||_{C^{2s+\nu}(\mathbb{T})}&\lesssim ||w||_{L^{\infty}(\mathbb{T})},\\
  ||g_1||_{C^{\nu}(\mathbb{T})}&\lesssim ||g||_{L^{\infty}(\mathbb{T})}.
 \end{align*}\\
 Since $w_2\in H^{s}_{m_1}(\mathbb{T})$ and $g_2\in H^{-s}_{m_1}\cap C^{\nu}(\mathbb{T})$ satisfy $\mathcal{L}(s,\beta)w_2=g_2$ and $\mathcal{L}(s,\beta)$ is invertible, applying Proposition \ref{p:SymbolHolder} yields $||w_2||_{C^{2s+\nu}(\mathbb{T})} \lesssim ||g_2||_{C^{\nu}(\mathbb{T})}\lesssim ||g||_{C^{\nu}(\mathbb{T})}$. We obtained the desired estimate \eqref{eq:Holder}.
\end{proof}

\vspace{5pt}

\section{Non-existence}\label{s:5}

Let $0<s<1$ and $m_0\in \mathbb{N}_0$. Let $\beta$ satisfy \eqref{eq: finitebm_0}. We say that $w\in H^{s}_{m_0}(\mathbb{T})\cap C^{1}(\mathbb{T})$ and $g\in H^{-s}_{m_0}(\mathbb{T})\cap C^{1}(\mathbb{T})$ is a solution to the nonlinear problem on the torus: 

\begin{equation}
\begin{aligned}
\beta w \partial_{\theta}g =(\beta+2s)g\partial_{\theta}w,\quad  
\mathcal{L}(s,\beta)w=g.   
\end{aligned}
\label{eq:nonlinear}
\end{equation}\\
The Sobolev regularity is abundant to define solutions to \eqref{eq:nonlinear} since $H^{s}(\mathbb{T})\subset  C^{1}(\mathbb{T})$. We show non-existence of odd symmetric solutions for $-2s\leq \beta\leq 0$ and complete the proof of the non-existence statement in Theorem \ref{t:thm} (ii).

\subsection{The odd function spaces}

We set the spaces of odd symmetric functions 

\begin{align*}
H^{l}_{m_0,\textrm{odd}}(\mathbb{T})=\left\{f\in H^{l}_{m_0}(\mathbb{T})\ \middle|\  f=\sum_{|m|\geq m_0}^{\infty}\hat{f}_{m}e^{im\theta},\ \hat{f}_{m}=-\hat{f}_{-m} \right\}, \quad l\in \mathbb{R}.
\end{align*}\\
In the sequel, we take $m_0\in \mathbb{N}$ under odd symmetry since $H^{l}_{0,\textrm{odd}}(\mathbb{T})=H^{l}_{1,\textrm{odd}}(\mathbb{T})$. By the trigonometric Fourier expansion, 

\begin{align*}
f=\sum_{m\in \mathbb{Z}}^{\infty}\hat{f}_me^{im\theta}
=\hat{f}_0+\sum_{m=1}^{\infty}\hat{f}_me^{im\theta}+\sum_{m=1}^{\infty}\hat{f}_{-m}e^{-im\theta} 
=\hat{f}_0+\sum_{m=1}^{\infty}(\hat{f}_m+\hat{f}_{-m})\cos(m\theta)+\sum_{m=1}^{\infty}i(\hat{f}_m-\hat{f}_{-m})\sin(m\theta).
\end{align*}\\
A function $f$ is odd-symmetric if and only if $f$ is expressed as a sine series. By using the orthonormal basis $e_m=\sin(m\theta)/\sqrt{\pi}$ of the odd symmetric subspace of $L^{2}(\mathbb{T})$, we express the odd symmetric spaces as 

\begin{align*}
H^{l}_{m_0,\textrm{odd}}(\mathbb{T})=\left\{f\in H^{l}_{m_0}(\mathbb{T})\ \middle|\  f= \sum_{m=m_0}^{\infty}(f,e_m)_{L^{2}}e_m\  \right\}, \quad l\in \mathbb{R}.
\end{align*}

\vspace{5pt}

\subsection{Irrotational solutions}

We use the rigidity of irrotational solutions to prove the non-existence of odd symmetric solutions to \eqref{eq:nonlinear} for $-2s\leq \beta\leq 0$.

\begin{thm}[Irrotational solutions] \label{t:Irrotational}
Let $0<s<1$ and $m_0\in \mathbb{N}$. Let $\beta$ satisfy \eqref{eq: finitebm_0}. Assume that $w\in H^{s}_{m_0,\textrm{odd}}(\mathbb{T})$ is a solution to \eqref{eq:nonlinear} for $g=0$. Then, the following holds:\\
\begin{itemize}
\item[(i)] If $\beta$ satisfies \eqref{eq: positivebm_0}, $w=0$\\
\item[(ii)] If $\beta$ does not satisfy \eqref{eq: positivebm_0}, $w=0$ except for the following values:\\
\begin{itemize}
\item[(a)] $\beta=-m_0$, $m_0+2-2s$\\
\item[(b)] $\beta=-m_0-1$, $\beta=m_0+3-2s$ in case $1/2<s<1$\\
\end{itemize}
In case (a), $w$ is a constant multiple of $e_{m_0}=\sin(m_0\theta)/\sqrt{\pi}$. In case (b), $w$ is a constant multiple of $e_{m_0+1}=\sin((m_0+1)\theta)/\sqrt{\pi}$.  
\end{itemize}
\end{thm}

\vspace{5pt}

\begin{proof}
The Fourier coefficients of $w$ satisfy

\begin{align*}
\mu_{|m|} (s,\beta)\hat{w}_m=0,\quad |m|\geq m_0.
\end{align*}\\
If $\beta$ satisfies \eqref{eq: positivebm_0}, $\mu_{|m|}(s,\beta)$ is positive for all $|m|\geq m_0$ by Proposition \ref{p: Symbolsign}. Thus $w=0$. 

If $\beta$ does not satisfy \eqref{eq: positivebm_0}, $\mu_{|m|}(s,\beta)$ is positive for all $|m|\geq m_0+2$ by Proposition \ref{p: Symbolsign}. Thus $\hat{w}_m=0$ for $|m|\geq m_0+2$. For $0<s\leq 1/2$, $\mu_{m_0+1}(s,\beta)>0$ and $\hat{w}_{m_0+1}=0$. Thus, 

\begin{align*}
\mu_{m_0} (s,\beta)\hat{w}_{m_0}=0.
\end{align*}\\
Since $\mu_{m_0} (s,\beta)$ vanishes at $\beta=-m_0$ and $m_0+2-2s$, $w=0$ except for such $\beta$. For $1/2<s<1$, 

\begin{align*}
\mu_{|m|} (s,\beta)\hat{w}_m=0,\quad |m|= m_0, m_0+1.
\end{align*}\\
If $\beta=-m_0$ or $m_0+2-2s$, $\hat{w}_{m_0}$ may not vanish. If $\beta$ is not those constants, $\hat{w}_{m_0}=0$. The constant $\mu_{m_0+1} (s,\beta)$ can vanish at $\beta=-m_0-1$ and $m_0+3-2s$. If $\beta=-m_0-1$ or $m_0+3-2s$, $\hat{w}_{m_0}$ may not vanish. If $\beta$ is not those constants, $w=0$.
\end{proof}

For $-2s\leq s\leq 0$, irrotational solutions are only $e_1=\sin\theta/\sqrt{\pi}$.

\vspace{5pt}

\begin{prop}\label{p:Irrotational2}
Let $0<s<1$ and $m_0\in \mathbb{N}$. Let $\beta$ satisfy \eqref{eq: finitebm_0}. Let $w\in H^{s}_{m_0,\textrm{odd}}(\mathbb{T})\cap C^{1}(\mathbb{T})$ be a solution to \eqref{eq:nonlinear} for $g=0$. Assume that $-2s\leq \beta\leq 0$. Then, the following holds for $w$:\\

\begin{itemize}
\item[(i)] For $m_0\geq 2$, $w=0$\\
\item[(ii)] For $m_0=1$, $w=0$ except for $1/2\leq s<1$ and $\beta=-1$ for which $w$ is a constant multiple of $e_1=\sin\theta/\sqrt{\pi}$
\end{itemize}
\end{prop}

\vspace{5pt}

\begin{proof}
For $m_0\geq 2$, $-m_0<-2<-2s\leq \beta\leq 0<m_0+2-2s$ and $\beta$ satisfies \eqref{eq: positivebm_0}. Thus, $w=0$ by Theorem \ref{t:Irrotational}. 

We consider the case $m_0=1$. If $\beta$ does not satisfy \eqref{eq: positivebm_0}, $w=0$ except for $\beta=-1$ or $-2$ by Theorem \ref{t:Irrotational}. Since $-2s\leq \beta\leq 0$, $\beta\neq -2$. For $0<s<1/2$, $\beta\neq -1$. Thus $w$ is a constant multiple of $e_1$ for $\beta=-1$ and $1/2\leq s<1$.
\end{proof}

\subsection{Rotational solutions}

We show that all odd symmetric solutions to \eqref{eq:nonlinear} for $-2s\leq \beta\leq 0$ are irrotational. \\

\begin{thm}\label{t:Rotational}
Let $0<s<1$ and $m_0\in \mathbb{N}$. Let $\beta$ satisfy \eqref{eq: finitebm_0}. Let $w\in H^{s}_{m_0,\textrm{odd}}(\mathbb{T})\cap C^{1}(\mathbb{T})$ and $g\in H^{-s}_{m_0,\textrm{odd}}(\mathbb{T})\cap C^{1}(\mathbb{T})$ be a solution to \eqref{eq:nonlinear} for $-2s\leq \beta\leq 0$. For $\beta=-2s$, assume in addition that $w,g\in C^{2}(\mathbb{T})$. Then, $g=0$ and the following holds for $w$:\\

\begin{itemize}
\item[(i)] For $m_0\geq 2$, $w=0$\\
\item[(ii)] For $m_0=1$, $w=0$ except for $1/2\leq s<1$ and $\beta=-1$ for which $w$ is a constant multiple of $e_1=\sin\theta/\sqrt{\pi}$
\end{itemize}
\end{thm}

\vspace{5pt}

\begin{proof}
It suffices to show $g=0$ by Proposition \ref{p:Irrotational2}. We first consider the case $-2s<\beta<0$. We show that $wg=0$ on $\mathbb{T}$ . On the contrary, suppose that $w(\theta_0)g(\theta_0)\neq 0$ for some point $\theta_0\in \mathbb{T}$. By odd symmetry \eqref{eq:Odd}, $w$ and $g$ vanish at $\theta=0$ and $\pi$. We may assume that $\theta_0\in (0,\pi)$ and take an open interval $J\subset (0,\pi)$ such that $wg\neq 0$ on $J$. By dividing the first equation of \eqref{eq:nonlinear} by $wg$ and integrating it, 

\begin{align*}
|w|^{\beta+2s}|g|^{-\beta}=C,
\end{align*}\\
for some constant $C$ on $J$. By $-\beta>0$ and $\beta+2s>0$, $wg\neq 0$ on $\overline{J}$. By continuing this argument, $wg\neq 0$ on $(0,\pi)$ and the above equality holds on $(0,\pi)$. Since $w(0)=0$, $C=0$. This contradicts $w(\theta_0)g(\theta_0)\neq 0$. Thus $wg=0$ on $\mathbb{T}$. 

We observe that $\partial_{\theta}w\partial_{\theta}g=0$ on $\mathbb{T}$ holds. Indeed, if $\partial_{\theta}w(\theta_0)\neq 0$ at some point $\theta_0$, $w$ is monotone near $\theta_0$ and $\partial_{\theta}g(\theta_0)=0$ by continuity of $\partial_{\theta}g$. Thus $\partial_{\theta}w\partial_{\theta}g=0$ on $\mathbb{T}$. The Fourier coefficients of $w$ and $g$ satisfy

\begin{align*}
\mu_{|m|}(s,\beta)\hat{w}_m=\hat{g}_m,\quad |m|\geq m_0.
\end{align*}\\
By the symbol identity \eqref{eq: muK},

\begin{align*}
(m^{2}-\beta^{2})\hat{w}_m=\frac{1}{K_{|m|}(s,\beta)}\hat{g}_m,\quad |m|\geq m_0.
\end{align*}\\
By $-2s\leq \beta\leq 0$ and \eqref{eq: Km}, the constant $K_{|m|}(s,\beta)$ is positive for all $|m|\geq m_0$. By multiplying $2\pi \overline{\hat{g}_m}$ by this, and summing up for all $|m|\geq m_0$,

\begin{align*}
2\pi \sum_{|m|\geq m_0}\left(m^{2}-\beta^{2}\right)\hat{w}_{m}\overline{\hat{g}_{m}}
= 2\pi\sum_{|m|\geq m_0}\frac{1}{K_{|m|}(s,\beta)}|{\hat{g}_{m}}|^{2}.
\end{align*}\\
By Parseval's identity, the left-hand side equals $-(\partial_{\theta}w,\partial_{\theta}g)_{L^{2}}-\beta^{2}(w,g)_{L^{2}}=0$. Thus $g=0$.

We consider the case $\beta=0$. We show that $\partial_{\theta}w\partial_{\theta}g=0$ on $\mathbb{T}$. By the first equation of \eqref{eq:nonlinear}, $g\partial_{\theta}w=0$ on $\mathbb{T}$. {If $\partial_{\theta}g(\theta_0)\neq 0$ for some point $\theta_0$, $g$ is monotone near $\theta_0$. If $g(\theta_0)\neq 0$, $\partial_{\theta}w(\theta_0)=0$. If $g(\theta_0)= 0$, $g(\theta)\neq 0$ for small $|\theta-\theta_0|\neq 0$. By $g\partial_{\theta}w=0$, $\partial_{\theta}w(\theta)=0$ for small $|\theta-\theta_0|\neq 0$. Since $\partial_{\theta}w(\theta)$ is continuous $\partial_{\theta}w(\theta_0)=0$. Thus $\partial_{\theta}w\partial_{\theta}g=0$ on $\mathbb{T}$.}

By multiplying $2\pi m^{2}\overline{\hat{w}_m}$ by 

\begin{align*}
K_{|m|}(s,0) m^{2}\hat{w}_{m}=\hat{g}_{m},
\end{align*}\\
summing up it for all $|m|\geq m_0$,

\begin{align*}
2\pi \sum_{|m|\geq m_0}K_{|m|}(s,0) m^{4}|\hat{w}_m|^{2}
=2\pi \sum_{|m|\geq m_0}m^{2}\overline{\hat{w}_m} \hat{g}_m
=-(\partial_{\theta}w,\partial_{\theta}g)_{L^{2}}=0.
\end{align*}\\
We conclude $w=0$.

It remains to consider the case $\beta=-2s$. By the first equation of \eqref{eq:nonlinear}, $w\partial_{\theta}g=0$. By the same argument as above for $(w,g)\in C^{2}(\mathbb{T})$, $\partial_{\theta} w\partial_{\theta}g=0$ and $\partial_{\theta}^{2} w\partial_{\theta}^{2}g=0$ on $\mathbb{T}$.  By multiplying $2\pi m^{2}\overline{\hat{g}_{m}}$ by  

\begin{align*}
(m^{2}-4s^{2})\hat{w}_{m}=\frac{1}{K_{|m|}(s,-2s)}\hat{g}_{m},
\end{align*}\\
and summing up it all $|m|\geq m_0$, 

\begin{align*}
2\pi\sum_{|m|\geq m_0}(m^{2}-4s^{2})m^{2}\hat{w}_{m}\overline{\hat{g}_{m}}
=2\pi\sum_{|m|\geq m_0}\frac{1}{K_{|m|}(s,-2s)} m^{2} |\hat{g}_{m}|^{2}.
\end{align*}\\
The left-hand side equals $(\partial_{\theta}^{2}w,\partial_{\theta}^{2}g)_{L^{2}}-4s^{2}(\partial_{\theta}w,\partial_{\theta}g)_{L^{2}}=0$. Thus, $g=0$.
\end{proof}

\vspace{5pt}

\begin{proof}[Proof of Theorem \ref{t:thm} (Non-existence)]
The result follows from Theorem \ref{t:Rotational}.
\end{proof}

\vspace{5pt}

\vspace{5pt}

\section{Existence}\label{s:6}

Let $0<s<1$ and $m_0\in \mathbb{N}$. Let $\beta$ satisfy \eqref{eq: finitebm_0}. We construct odd symmetric solutions to the nonlinear problem 

\begin{equation}
\begin{aligned}
\beta w \partial_{\theta}g =(\beta+2s)g\partial_{\theta}w,\quad  
\mathcal{L}(s,\beta)w=g,
\end{aligned}
\label{eq:nonlinear2}
\end{equation}\\
for $\beta<-2s$ and $0<\beta$ (we assume $1/2-s<\beta$ for $0<s<1/2$) 
and complete the proof of Theorem \ref{t:thm}. We choose $w$ and $g$ by

\begin{align}
g=cw |w|^{\frac{2s}{\beta}},\quad c>0,\quad \beta< -2s\quad \textrm{or}\quad 0<\beta,  \label{eq:ansatz}
\end{align}\\
so that they satisfy the first equation of \eqref{eq:nonlinear2}. The relationship $g=cw |w|^{\frac{2s}{\beta}}$ can be obtained by integrating the first equation of \eqref{eq:nonlinear2}. We construct odd symmetric solutions to \eqref{eq:nonlinear2} for such $g$ via the semilinear problem
 
\begin{align}
\mathcal{L}(s,\beta)w=cw|w|^{\frac{2s}{\beta}}.  \label{eq:SEP}
\end{align}

\vspace{5pt}

\subsection{The functional setup}
We seek solutions to \eqref{eq:SEP} as the critical points of the functional 

\begin{align}
I[w]=\frac{1}{2}B(w,w)-\frac{c\beta }{2(\beta+s)}\int_{0}^{2\pi}|w|^{2+\frac{2s}{\beta}}d \theta,\quad w\in H^{s}_{m_0,\textrm{odd}}(\mathbb{T})=:X.  \label{eq:F}
\end{align}\\
By the Sobolev inequality on the torus (Theorem \ref{t:Sobolev}),

\begin{equation}
\begin{aligned}
&H^{s}(\mathbb{T})\subset C^{s-\frac{1}{2}}(\mathbb{T}),\quad \frac{1}{2}<s<1,\\
&H^{\frac{1}{2}}(\mathbb{T})\subset L^{q}(\mathbb{T}),\quad 1\leq q<\infty,\\
&H^{s}(\mathbb{T})\subset L^{q}(\mathbb{T}),\quad \frac{1}{q}=\frac{1}{2}-s,\quad 0<s<\frac{1}{2}.
\end{aligned}
\label{eq:Sobolev}
\end{equation}\\
The embeddings into subcritical spaces $H^{s}(\mathbb{T})\subset \subset L^{r}(\mathbb{T})$ for $r<q$ are compact. In the case $0<s<1/2$, we choose $1/2-s< \beta$ and define the $L^{2+2s/\beta}(\mathbb{T})$ norm for functions in $H^{s}(\mathbb{T})$. In the sequal, we consider the following $(m_0,s,\beta)$: $m_0\in \mathbb{N}$ and

\begin{equation}
\begin{aligned}
&  \frac{1}{2}\leq s<1; \quad -m_0-2s<\beta< -2s\quad \textrm{or}\quad 0<\beta<m_0+2,\\
&  0< s<\frac{1}{2};\quad -m_0-2s<\beta< -2s\quad \textrm{or}\quad \frac{1}{2}-s< \beta<m_0+2.
\end{aligned}
\label{eq:AD}
\end{equation}\\
We first show that $I\in C^{1}(X; \mathbb{R})$ and critical points $w\in X$ of $I$ are solutions to \eqref{eq:SEP} for $g=cw |w|^{\frac{2s}{\beta}}\in H^{-s}_{m_0,\textrm{odd}}(\mathbb{T})$.

\vspace{5pt}

\begin{prop}
Let $(m_0,s,\beta)$ satisfy \eqref{eq:AD}. Then, 

\begin{align}
H^{s}(\mathbb{T})\subset\subset L^{2+\frac{2s}{\beta}}(\mathbb{T}),  \label{eq:CE}
\end{align}\\
and $I\in C(X; \mathbb{R})$. Moreover, 

\begin{align}
\left\|\ |w|^{1+\frac{2s}{\beta}}\ \right\|_{H^{-s}(\mathbb{T}) }\lesssim ||w||_{H^{s}(\mathbb{T})}^{1+\frac{2s}{\beta}},\quad w\in H^{s}(\mathbb{T}). \label{eq:gestimate}
\end{align}
\end{prop}

\vspace{5pt}

\begin{proof}
By the Sobolev inequality \eqref{eq:Sobolev}, the compact embedding \eqref{eq:CE} holds for $1/2\leq s<1$. For $0<s<1/2$, $r=2+2s/\beta$ satisfies $1/r>1/2-s=1/q$ by the condition $1/2-s<\beta$ and \eqref{eq:CE} also holds.

By H\"older's inequality for $1/p+1/q=1$, 

\begin{align*}
\left|\int_{\mathbb{T}}|w|^{1+\frac{2s}{\beta}}\eta d\theta\right|\leq \left(\int_{\mathbb{T}}|w|^{(1+\frac{2s}{\beta})p}d\theta\right)^{\frac{1}{p}}
\left(\int_{\mathbb{T}}|\eta|^{q}d\theta\right)^{\frac{1}{q}}
\lesssim ||w||_{L^{(1+\frac{2s}{\beta})p}}^{1+\frac{2s}{\beta}}
||\eta ||_{L^{q}(\mathbb{T})},\quad w,\eta\in H^{s}(\mathbb{T}).
\end{align*}\\
For $1/2\leq s<1$, \eqref{eq:gestimate} follows from the Sobolev inequality and duality $H^{s}(\mathbb{T})^{*}=H^{-s}(\mathbb{T})$. For $0<s<1/2$, we apply the Sobolev inequality \eqref{eq:Sobolev} for $1/q=1/2-s$. Since $(1+2s/\beta)p\leq q$, \eqref{eq:gestimate} holds. 
\end{proof}

\vspace{5pt}

\begin{prop}\label{p:FD}
Let $(m_0,s,\beta)$ satisfy \eqref{eq:AD}. Then, $I\in C^{1}(X; \mathbb{R} )$ and 

\begin{align}
\langle I'[w],\eta \rangle=B(w,\eta)- \langle cw|w|^{\frac{2s}{\beta}},\eta \rangle,\quad \eta\in X.   \label{eq:FD}
\end{align}
\end{prop}

\vspace{5pt}

\begin{proof}
We set
 
\begin{align*}
I[w]
=\frac{1}{2}B(w,w)-\frac{c\beta }{2(\beta+s)}\int_{0}^{2\pi}|w|^{2+\frac{2s}{\beta}}d \theta 
=: I_0[w]-J[w],\quad w\in X.   
\end{align*}\\
For $\eta\in X$ and $\varepsilon>0$,

\begin{align*}
\frac{B(w+\varepsilon \eta,w+\varepsilon \eta)-B(w,w)}{\varepsilon}
=2B(w,\eta)+\varepsilon B(\eta,\eta).
\end{align*}\\
By letting $\varepsilon\to0$, the G\^{a}teaux derivative $D_GI_0[w]$ exists and $\langle D_GI_0[w],\eta\rangle=B(w,\eta)$. The G\^{a}teaux derivative $D_GI_0[\cdot ]: X\to X^{*}$ is continuous by continuity of the bilinear form $B: X\times X\to \mathbb{R}$. Hence the Fr\'{e}chet derivative $I_0'[w]=D_GI_0[w]$ exists and $I_0\in C^{1}(X; \mathbb{R})$.

We differentiate the functional $J$. By 

\begin{align*}
|w+\varepsilon \eta|^{2+\frac{2s}{\beta}}-|w|^{2+\frac{2s}{\beta}}
=\int_{0}^{\varepsilon}\frac{d}{dt}|w+t\eta|^{2+\frac{2s}{\beta}}dt
&=\frac{2(\beta+s)}{\beta}\int_{0}^{\varepsilon}(w+t\eta)|w+t\eta|^{\frac{2s}{\beta}}\eta d t \\
&=\frac{2\varepsilon (\beta+s)}{\beta}\int_{0}^{1}(w+\varepsilon\sigma\eta)|w+\varepsilon\sigma\eta|^{\frac{2s}{\beta}}\eta d \sigma,
\end{align*}\\
we express the differential quotient as 

\begin{align*}
\frac{J[w+\varepsilon \eta]-J[w] }{\varepsilon}
=c\int_{0}^{2\pi}\left( \int_{0}^{1}(w+\varepsilon\sigma\eta)|w+\varepsilon\sigma\eta|^{\frac{2s}{\beta}}\eta d \sigma \right)  d \theta.
\end{align*}\\
By \eqref{eq:CE}, $|w+\varepsilon \sigma \eta|^{1+2s/\beta}|\eta|\leq (|w|+|\eta|)^{2+2s/\beta}\in L^{1}(\mathbb{T})$. By letting $\varepsilon\to 0$ and the dominated convergence theorem, the G\^{a}teaux derivative $D_GJ[w]$ exists and $\langle D_GJ[w],\eta\rangle=(cw|w|^{2s/\beta},\eta)_{L^{2}}$. Since $D_{G}J[\cdot]: X\to X^{*}$ is continuous by \eqref{eq:gestimate}, the Fr\'{e}chet derivative $J'[w]=D_GJ[w]$ exists and $J\in C^{1}(X; \mathbb{R})$. Hence $I\in C^{1}(X; \mathbb{R})$ and \eqref{eq:FD} holds.
\end{proof}

\vspace{5pt}

\begin{lem}\label{l:CPisSol}
Let $(m_0,s,\beta)$ satisfy \eqref{eq:AD}. Assume that $w\in X$ is a critical point of $I$. Then, $g=cw|w|^{\frac{2s}{\beta}}\in H^{-s}_{m_0,\textrm{odd}}(\mathbb{T})$ and 

\begin{align}
B(w,\eta)=\langle cw|w|^{\frac{2s}{\beta}},\eta\rangle,\quad \eta\in H^{s}_{m_0}(\mathbb{T}).   \label{eq:WF}
\end{align}
\end{lem}

\vspace{5pt}

\begin{proof}
By Proposition \ref{p:FD}, the critical point $w\in X$ satisfies 

\begin{align*}
2\pi \sum_{|m|\geq m_0}\mu_{|m|}(s,\beta)\hat{w}_{m}\overline{\hat{\eta}_{m}}
=\sum_{|m|\geq m_0}\hat{g}_{m}\overline{\hat{\eta}_{m}},
\end{align*}\\
for all $\eta=\sum_{|m|\geq m_0}\hat{\eta}_{m}e^{im\theta}\in H^{s}_{m_0}(\mathbb{T})$ satisfying $\hat{\eta}_{m}=-\hat{\eta}_{-m}$. We take an arbitrary $\eta=\sum_{|m|\geq m_0}\hat{\eta}_{m}e^{im\theta}\in H^{s}_{m_0}(\mathbb{T})$ and set 

\begin{align*}
\hat{\eta}_{m}=\hat{\eta}_{m,\textrm{even}}+\hat{\eta}_{m,\textrm{odd}},\quad
\hat{\eta}_{m,\textrm{even}}=\frac{1}{2}(\hat{\eta}_{m}+\hat{\eta}_{-m}), \quad
\hat{\eta}_{m,\textrm{odd}}=\frac{1}{2}(\hat{\eta}_{m}-\hat{\eta}_{-m}),
\end{align*}\\
so that $\hat{\eta}_{m,\textrm{even}}=\hat{\eta}_{-m,\textrm{even}}$ and $\hat{\eta}_{m,\textrm{odd}}=-\hat{\eta}_{-m,\textrm{odd}}$. Then, by odd symmetry of $g$, 

\begin{align*}
\sum_{|m|\geq m_0}\hat{g}_{m}\overline{\hat{\eta}_{m,\textrm{even}}}
&=\sum_{|m|\geq m_0}\hat{g}_{-m}\overline{\hat{\eta}_{-m,\textrm{even}}}
=-\sum_{|m|\geq m_0}\hat{g}_{m}\overline{\hat{\eta}_{m,\textrm{even}}}, \\
\sum_{|m|\geq m_0}\hat{g}_{m}\overline{\hat{\eta}_{m}}
&=\sum_{|m|\geq m_0}\hat{g}_{m}\overline{\hat{\eta}_{m,\textrm{even}}}+\sum_{|m|\geq m_0}\hat{g}_{m}\overline{\hat{\eta}_{m,\textrm{odd}}}
=\sum_{|m|\geq m_0}\hat{g}_{m}\overline{\hat{\eta}_{m,\textrm{odd}}}.
\end{align*}\\
Similarly, since $\mu_{|m|}(s,\beta)\hat{w}_{m}$ is odd symmetric, 

\begin{align*}
\sum_{|m|\geq m_0}\mu_{|m|}(s,\beta)\hat{w}_{m}\overline{\hat{\eta}_{m}}
=\sum_{|m|\geq m_0}\mu_{|m|}(s,\beta)\hat{w}_{m}\overline{\hat{\eta}_{m,\textrm{odd}}}.
\end{align*}\\
Thus, \eqref{eq:WF} holds.
\end{proof}

\subsection{Regularity of critical points}

We show that critical points $w\in H^{s}_{m_0}(\mathbb{T})$ and $g=cw |w|^{\frac{2s}{\beta}}\in H^{-s}_{m_0}(\mathbb{T})$ have the regularity $w \in C^{2s+1+\frac{2s}{\beta}}(\mathbb{T})$ and $g=cw |w|^{\frac{2s}{\beta}}\in C^{1+\frac{2s}{\beta}}(\mathbb{T})$ for $1/2<s<1$ by applying the H\"older estimate \eqref{eq:Holder}. \\ 


\begin{lem}[$1/2< s<1$]\label{l:1/2<s}
Let $(m_0,s,\beta)$ satisfy \eqref{eq:AD}. Assume that $w\in X$ is a critical point of $I$. Then, \\

\noindent
(i) If $2s+1+\frac{2s}{\beta}\notin \mathbb{N}$ and $1+\frac{2s}{\beta}\notin \mathbb{N}$, $w \in C^{2s+1+\frac{2s}{\beta}}(\mathbb{T})$ and $g=cw |w|^{\frac{2s}{\beta}}\in C^{1+\frac{2s}{\beta}}(\mathbb{T})$ \\ 
(ii) Otherwise, $w \in C^{2s+1+\frac{2s}{\beta}-\varepsilon}(\mathbb{T})$ and $g=cw |w|^{\frac{2s}{\beta}}\in C^{1+\frac{2s}{\beta}-\varepsilon}(\mathbb{T})$ for arbitrary small $\varepsilon>0$.
\end{lem}

\vspace{5pt}

\begin{proof}
We set $f(t)=ct|t|^{\frac{2s}{\beta}}\in C^{1+\frac{2s}{\beta}}[0,\infty)$. By the Sobolev embedding \eqref{eq:Sobolev}, $w\in H^{s}(\mathbb{T})\subset C^{s-1/2}(\mathbb{T})$. Since $w$ and $f$ are H\"older continuous, $g=f(w)\in C^{\nu_0}(\mathbb{T})$ for some $\nu_0\in (0,1)$. By Theorem \ref{t:Holder}, $w\in C^{2s+\nu_0}(\mathbb{T})$. In particular, $w\in C^{1+\nu_0}(\mathbb{T})$. 

For $\beta<-2s$, $f(t)$ is H\"older continuous of exponent $0<1+2s/\beta<1$ and hence $g=f(w)\in C^{1+\frac{2s}{\beta}}(\mathbb{T})$. By Theorem \ref{t:Holder}, $w\in C^{2s+1+\frac{2s}{\beta}}(\mathbb{T})$ if $2s+1+\frac{2s}{\beta}\notin \mathbb{N}$. Otherwise, $w\in C^{2s+1+\frac{2s}{\beta}-\varepsilon}(\mathbb{T})$ for $\varepsilon>0$.

For $0<\beta$, $f(t)$ is a $C^{1}$-function. For  $0<2s/\beta<1$, $g=f(w)\in C^{1+\nu_1}(\mathbb{T})$ for some $\nu_1\in (0,1)$. By differentiating the second equation of \eqref{eq:nonlinear2}, $\mathcal{L}(s,\beta)\partial_{\theta}w=\partial_{\theta}g$. By Theorem \ref{t:Holder}, $w\in C^{1+2s+v_1}(\mathbb{T})$. In particular, $w\in C^{2+v_1}(\mathbb{T})$. By $g=f(w)\in C^{1+\frac{2s}{\beta}}(\mathbb{T})$. By Theorem \ref{t:Holder}, we obtain $w\in C^{2s+1+\frac{2s}{\beta}}(\mathbb{T})$ if $2s+1+\frac{2s}{\beta}\notin \mathbb{N}$. Otherwise, $w\in C^{2s+1+\frac{2s}{\beta}-\varepsilon}(\mathbb{T})$ for $\varepsilon>0$.

For $2s/\beta=1$, we obtain $g=f(w)\in C^{1+\frac{2s}{\beta}-\varepsilon}(\mathbb{T})$ and $w\in C^{2s+1+\frac{2s}{\beta}-\varepsilon}(\mathbb{T})$ for arbitrary small $\varepsilon>0$. The case $1\leq 2s/\beta$ is similar.
\end{proof}

\vspace{5pt}
We obtain the same regularity for $s=1/2$ by using the $L^{p}$-estimate \eqref{eq:LP}.
\vspace{5pt}

\begin{lem}[$s=1/2$]\label{l:s=1/2}
Let $(m_0,s,\beta)$ satisfy \eqref{eq:AD}. Assume that $w\in X$ is a critical point of $I$. Then,\\

\noindent
(i) For $\frac{1}{\beta}\notin \mathbb{N}$, $w \in C^{2+\frac{1}{\beta}}(\mathbb{T})$ and $g=cw |w|^{\frac{1}{\beta}}\in C^{1+\frac{1}{\beta}}(\mathbb{T})$. \\
(ii) For $\frac{1}{\beta}\in \mathbb{N}$, $w \in C^{2+\frac{1}{\beta}-\varepsilon}(\mathbb{T})$ and $g=cw |w|^{\frac{1}{\beta}}\in C^{1+\frac{1}{\beta}-\varepsilon}(\mathbb{T})$ for arbitrary small $\varepsilon>0$.
\end{lem}

\vspace{5pt}

\begin{proof}
By the Sobolev embedding \eqref{eq:Sobolev}, $w\in H^{\frac{1}{2}}(\mathbb{T})\subset L^{q}(\mathbb{T})$ for all $1\leq q<\infty$. Thus, $g\in L^{p}(\mathbb{T})$ for all $1< p<\infty$. By Theorem \ref{t:LP}, $w\in H^{1,p}(\mathbb{T})$. By the Sobolev embedding \eqref{eq:Sobolev}, $w\in C^{\nu_{2}}(\mathbb{T})$ for all $\nu_2\in (0,1)$. Since $w$ and $f$ are H\"older continuous, $g=f(w)\in C^{\nu_3}(\mathbb{T})$ for some $\nu_3\in (0,1)$. By Theorem \ref{t:Holder}, $w\in C^{1+\nu_{3}}(\mathbb{T})$. 

For $\beta<-1$, $g=f(w)\in C^{1+\frac{1}{\beta}}(\mathbb{T})$. By Theorem \ref{t:Holder}, $w\in C^{2+\frac{1}{\beta}}(\mathbb{T})$. For $0<\beta$, $f(t)$ is a $C^{1+\frac{1}{\beta}}$-function. By applying the same argument as the proof of Lemma \ref{l:1/2<s}, we obtain the desired regularity results. 
\end{proof}

\vspace{5pt}

For $0<s<1/2$, we first show H\"older continuity of critical points by an iteration argument using the $L^{p}$-estimate \eqref{eq:LP}.

\vspace{5pt}

\begin{prop}[$0<s<1/2$] \label{p:LPtoHolder}
Let $(m_0,s,\beta)$ satisfy \eqref{eq:AD}. Assume that $w\in X$ is a critical point of $I$. Then, $w\in C^{\gamma}(\mathbb{T})$ for some $\gamma\in (0,1)$.
\end{prop}

\vspace{5pt}

\begin{proof}
By the Sobolev embedding $H^{s}(\mathbb{T})\subset L^{q_1}(\mathbb{T})$ for $1/q_1=1/2-s$, $g(w)=cw|w|^{\frac{2s}{\beta}}\in L^{p_1}(\mathbb{T})$ for $p_1=q_1/a$ and $a=1+2s/\beta>0$. For $\beta<-2s$, $0<a<1$ and $p_1>q_1>2$. For $1/2-s< \beta$, $a<1+2s/(1/2-s)$ and $p_1=q_1/a>2/(1+2s)>1$. By Theorem \ref{t:LP}, $w\in H^{2s,p_1}(\mathbb{T})$. If $1/p_1-2s<0$, $w$ is H\"older continuous by the Sobolev embedding. We may assume that $1/p_1-2s>0$.   

By the Sobolev embedding $H^{2s,p_1}(\mathbb{T})\subset L^{q_2}(\mathbb{T})$ for $1/q_2=1/p_1-2s$. By the same argument as above, $g(w)\in L^{p_2}(\mathbb{T})$ for $p_2=q_2/a$ and $w\in H^{2s,p_2}(\mathbb{T})$. If $1/p_2-2s<0$, $w$ is H\"older continuous by the Sobolev embedding.

By repeating this argument, we obtain a sequence $\{p_n\}\subset (1,\infty)$ such that 

\begin{align*}
w\in H^{2s,p_n}(\mathbb{T}),\quad \frac{1}{p_n}-2s>0,\quad \frac{1}{p_{n+1}}=a\left(\frac{1}{p_{n}}-2s \right), \quad n=1,2,\cdots.
\end{align*}\\
By solving the geometric sequence,

\begin{align*}
\frac{1}{p_{n}}=a^{n-1}\left(\frac{1}{p_{1}}-\alpha\right)+\alpha,\quad \alpha=\beta+2s.
\end{align*}\\
For $\beta<-2s$, $a<1$ and $\alpha<0$. Thus, $1/p_N<2s$ for some $N\geq 1$ and $w$ is H\"older continuous by the Sobolev embedding.

For $1/2-s< \beta$, $a=1+2s/\beta>1$ and 

\begin{align*}
\frac{1}{p_1}-\alpha=\frac{a}{q_1}-\frac{2as}{a-1}
=a\left(\frac{1}{q_1}- \frac{2s}{a-1}\right)=a\left(\frac{1}{2}-s- \beta\right)<0.
\end{align*}\\
Thus, $1/p_N<2s$ for some $N\geq 1$ and $w$ is H\"older continuous by the Sobolev embedding. 
\end{proof}

\vspace{5pt}

We next apply an iteration argument using the H\"older estimate \eqref{eq:Holder} and obtain the regularity $w\in C^{-\beta-\varepsilon}(\mathbb{T})$. If $\beta<-1$, $w$ is a $C^{1}$-function and we obtain the regularity $w\in C^{2s+1+\frac{2s}{\beta}}(\mathbb{T})$. If $-1\leq \beta<-2s$, $w$ merely lies in $C^{-\beta-\varepsilon}(\mathbb{T})$ with $g=cw|w|^{\frac{2s}{\beta}}\in C^{-\beta-2s-\varepsilon}(\mathbb{T})$.

\vspace{5pt}

\begin{lem}[$0<s<1/2$]\label{l:s<1/2} 
Let $(m_0,s,\beta)$ satisfy \eqref{eq:AD}. Assume that $w\in X$ is a critical point of $I$. Then, \\

\noindent
(I) For $\beta<-1$ and $1/2-s< \beta$,\\
(i) If $2s+1+\frac{2s}{\beta}\notin \mathbb{N}$ and $1+\frac{2s}{\beta}\notin \mathbb{N}$, $w \in C^{2s+1+\frac{2s}{\beta}}(\mathbb{T})$ and $g=cw |w|^{\frac{2s}{\beta}}\in C^{1+\frac{2s}{\beta}}(\mathbb{T})$.\\
(ii) Otherwise, $w \in C^{2s+1+\frac{2s}{\beta}-\varepsilon}(\mathbb{T})$ and $g=cw |w|^{\frac{2s}{\beta}}\in C^{1+\frac{2s}{\beta}-\varepsilon}(\mathbb{T})$ for arbitrary small $\varepsilon>0$.

\noindent
(II) For $-1\leq  \beta<-2s$, $w \in C^{-\beta-\varepsilon}(\mathbb{T})$ and $g=cw  |w|^{\frac{2s}{\beta}}\in C^{-\beta-2s-\varepsilon}(\mathbb{T})$ for arbitrary small $\varepsilon>0$.
\end{lem}

\vspace{5pt}

\begin{proof}
We first consider negative $\beta<-2s$. Observe that $f(t)\in C^{a}[0,\infty)$ for $0<a=1+2s/\beta<1$. By Proposition \ref{p:LPtoHolder}, $w\in C^{\gamma_1}(\mathbb{T})$ for some $\gamma_1\in (0,1)$ and $g(w)\in C^{a\gamma_1}(\mathbb{T})$. By Theorem \ref{t:Holder}, $w\in C^{2s+a\gamma_1}(\mathbb{T})$. If $\gamma_2=2s+a\gamma_1>1$, $w$ is a $C^{1}$-function. We may assume that $\gamma_2=2s+a\gamma_1<1$. By repeating this argument, we obtain a sequence $\{\gamma_n\}$ such that 

\begin{align*}
w\in C^{\gamma_n}(\mathbb{T}),\quad \gamma_{n+1}=a\gamma_n+2s,\quad n=1,2,\cdots.
\end{align*}\\
By solving the geometric sequence,

\begin{align*}
\gamma_n=a^{n-1}(\gamma_1+\beta)-\beta,
\end{align*}\\
and $\gamma_n$ approaches $-\beta$. If $\beta<-1$, $\gamma_N>1$ for some $N\geq 1$ and $w$ is a $C^{1}$-function. Then, $g(w)\in C^{a}(\mathbb{T})$ and $w\in C^{2s+a}(\mathbb{T})$ if $2s+a\notin \mathbb{N}$ and $w\in C^{2s+a-\varepsilon}(\mathbb{T})$ for $\varepsilon>0$ if $2s+a\in \mathbb{N}$ by Theorem \ref{t:Holder}. If $-1\leq \beta<-2s$, $w\in C^{-\beta-\varepsilon}(\mathbb{T})$ and $g(w)\in C^{-\beta-2s-\varepsilon}(\mathbb{T})$ for arbitrary small $\varepsilon>0$.

In the case $1/2-s<\beta$, $f(t)\in C^{1}[0,\infty)$ and $w\in C^{\gamma_1}(\mathbb{T})$ implies that $g(w)\in C^{\gamma_1}(\mathbb{T})$. By Theorem \ref{t:Holder}, $w\in C^{2s+\gamma_1}(\mathbb{T})$. By repeating this argument, $w\in C^{1+\nu_0}(\mathbb{T})$ for some $\nu_0\in (0,1)$. Since $f(t)\in C^{1+\frac{2s}{\beta}}[0,\infty)$, the desired regularity follows from the same argument as in the proof of Lemma \ref{l:1/2<s}. 
\end{proof}

\vspace{5pt}
\subsection{Variational principles}

In the rest of this paper, we find critical points of the functional \eqref{eq:F} by variational principles. For a functional $I\in C^{1}(X;\mathbb{R})$ on a Banach space $X$, we say that $w\in X$ is a critical point of $I$ if 

\begin{align*}
\langle I'[w],\eta\rangle=0\quad \textrm{for all}\ \eta\in X.
\end{align*}\\
The constant $c\in \mathbb{R}$ is a critical value if a critical point $w\in X$ exists at $c=I[w]$. We say that a sequence $\{w_n\}\subset X$ is a Palais--Smale sequence at level $c\in \mathbb{R}$ if 

\begin{align*}
I[w_n]&\to c,\\ 
I'[w_n]&\to 0\quad \textrm{in}\ X^{*}.
\end{align*}\\
We say that $I$ satisfies the $(\textrm{PS})_c$ condition if any Palais--Smale sequence at level $c\in \mathbb{R}$ has a convergent subsequence in $X$. We apply the following three variational principles for the functional $I\in C^{1}(X; \mathbb{R})$ \cite[Corollary 2.5, Theorems 2.10, 2.11, 2.12]{Willem}:\\

\begin{lem}[Minimizing method]\label{l:Min}
Assume that $I$ is bounded from below and satisfies the $(\textrm{PS})_c$ condition with $c=\inf_{X}I$. Then, there exists a minimizer for $I$.
\end{lem}
 
\begin{lem}[Mountain pass theorem]\label{l:MP}
Assume that there exist $w_0\in X$ and $r_0>0$ such that $||w_0||_{X}>r_0$ and 

\begin{align*}
\inf_{||w||_{X}=r_0}I[w]>I[0]\geq I[w_0]. 
\end{align*}\\
Assume that $I$ satisfies the $(\textrm{PS})_c$ condition with 

\begin{align*}
c&=\inf_{\gamma\in \Lambda}\max_{0\leq t\leq 1}I[\gamma(t)], \\
\Lambda&=\left\{\gamma\in C([0,1]; X )\ \middle|\ \gamma(0)=0,\ \gamma(1)=w_0\  \right\}.
\end{align*}\\
Then, $c$ is a critical value of $I$.
\end{lem}

We apply linking and saddle point theorems when $X$ admits a direct sum decomposition $X=Y\oplus Z$ with a finite-dimensional subspace $Y$ and a subspace $Z$.

\begin{lem}[Linking theorem]\label{l:LT}
For $\rho_0>r_0>0$ and $z_0\in Z$ such that $||z_0||_{X}=r_0$, set 

\begin{align*}
M&=\left\{w=y+\lambda z_0\in Y\oplus \mathbb{R}z_0\ |\ ||w||_{X}\leq \rho_0,\ \lambda\geq 0,\ y\in Y\right\},\\
M_0&=\left\{w=y+\lambda z_0\in Y\oplus \mathbb{R}z_0\ |\ \lambda=0\ \textrm{and}\ ||y||_{X}\leq \rho_0,\ \textrm{or}\ \lambda> 0\ \textrm{and}\ ||w||_{X}= \rho_0 \right\},\\
N&=\left\{z\in Z\ |\ ||z||_{X}=r_0\  \right\}.
\end{align*}\\
Assume that there exist $\rho_0>r_0>0$ and $z_0\in Z$ satisfying $||z_0||_{X}=r_0$ such that 

\begin{align*}
\inf_{N}I>\max_{M_0}I.
\end{align*}\\
Assume that $I$ satisfies the $(\textrm{PS})_c$ condition with 

\begin{align*}
c&=\inf_{\gamma\in \Lambda}\max_{w\in M}I[\gamma[w]],\\
\Lambda&=\left\{\gamma\in C(M; X)\ \middle|\ \gamma|_{M_0}=id\  \right\}.
\end{align*}\\
Then, $c$ is a critical value of $I$.
\end{lem}

\begin{lem}[Saddle-point theorem]\label{l:SP}
Let $X=Y\oplus Z$ for a finite-dimensional subspace $Y$ and a subspace $Z$ of $X$. For $\rho_0>0$, set 

\begin{align*}
M&=\left\{w\in Y\ \middle|\ ||w||_{X}\leq \rho_0\ \right\},\\
M_0&=\left\{w\in Y\ \middle|\ ||w||_{X}= \rho_0\ \right\},\\
N&=Z.
\end{align*}\\
Assume that there exists $\rho_0>0$ such that 

\begin{align*}
\inf_{N}I>\max_{M_0}I.
\end{align*}\\
Assume that $I$ satisfies the $(\textrm{PS})_c$ condition with 

\begin{align*}
c&=\inf_{\gamma\in \Lambda}\max_{w\in M}I[\gamma[w]],\\
\Lambda&=\left\{\gamma\in C(M; X)\ \middle|\ \gamma|_{M_0}=id\  \right\}.
\end{align*}\\
Then, $c$ is a critical value of $I$.
\end{lem}

\subsection{The direct sum decomposition}

For $(m_0,s,\beta)$ satisfying \eqref{eq:AD} and the functional $I\in C^{1}(X;\mathbb{R})$ for $X=H^{s}_{m_0,\textrm{odd}}(\mathbb{T})$ in \eqref{eq:F}, we show the $(\textrm{PS})_c$ condition for all $c\in \mathbb{R}$ and functional estimates on proper subsets. We then apply four different variational principles for the following $\beta$:\\

\noindent
(i) $-m_0<\beta<-2s$; minimizing method (Lemma \ref{l:Min})\\
(ii) $0<\beta< m_0+2-2s$ (We assume $1/2-s<\beta< m_0+2-2s$ for $0<s<1/2$); mountain pass theorem (Lemma \ref{l:MP})\\
(iii) $m_0+2-2s\leq \beta< m_0+2$; linking theorem (Lemma \ref{l:LT})\\
(iv) $-m_0-2s<\beta\leq - m_0$; saddle point theorem (Lemma \ref{l:SP})\\

Those divisions are due to signs of eigenvalues of $\mathcal{L}(s,\beta)$ and sub/supernonlinearity in \eqref{eq:SEP}. We first prepare bilinear form estimates under odd symmetry.\\

\begin{lem}\label{l:DSD}
Let $m_0\in \mathbb{N}$ and $0<s<1$. Let $\beta$ satisfy \eqref{eq: finitebm_0}. The function  $e_m=\sin m\theta/\sqrt{\pi}$ for $m\geq m_0$ is an eigenfunction of $\mathcal{L}(s,\beta)$ with the eigenvalue $\mu_{m}(s,\beta)$ and $\{e_m\}_{m=m_0}^{\infty}$ are orthonormal basis on $L^{2}_{m_0,\textrm{odd}}(\mathbb{T})$. Moreover, the bilinear form \eqref{eq:Bilinear} is expressed as 

\begin{align}
B(w,\eta)=\sum_{m=m_0}^{\infty}\mu_m(s,\beta)(w,e_m)_{L^{2}}(\eta,e_m)_{L^{2}},\quad w,\eta\in H^{s}_{m_0,\textrm{odd}}(\mathbb{T}).  \label{eq:BilinearSine}
\end{align}\\
(i) For $\beta$ satisfying \eqref{eq: positivebm_0}, 

\begin{align}
B(w,w)\geq \mu_{m_0}(s,\beta)||w||_{L^{2}}^{2},\quad w\in L^{2}_{m_0,\textrm{odd}}(\mathbb{T}). \label{eq:LB1}
\end{align}\\
(ii) For $\beta$ not satisfying \eqref{eq: positivebm_0}, 

\begin{equation}
\begin{aligned}
H^{s}_{m_0,\textrm{odd}}(\mathbb{T})&=Y\oplus Z,\quad 
Y=\textrm{span}\ (e_{m_0},e_{m_0+l}),\quad 
Z=H^{s}_{m_0+l+1,\textrm{odd}}(\mathbb{T}),\\
B(w,w)&=B(y,y)+B(z,z),\quad B(y,z)=0, \\
\mu_{m_0+l}(s,\beta)||y||_{L^{2}}^{2} &\leq B(y,y)\leq \mu_{m_0}(s,\beta)||y||_{L^{2}}^{2},\\
B(z,z)&\geq \mu_{m_0+l+1}(s,\beta)||z||_{L^{2}}^{2},\quad w=y+z\in Y\oplus Z,
\end{aligned}
\label{eq:LB2}
\end{equation}
for some $l\in \{0,1\}$. 
\end{lem}

\vspace{5pt}

\begin{proof}
Since $2\sqrt{\pi}i\hat{w}_{m}=(w,e_m)_{L^{2}}$ for $w\in L^{2}_{m_0,\textrm{odd}}(\mathbb{T})$, 

\begin{align*}
B(w,\eta)=2\pi \sum_{|m|\geq m_0}^{\infty}\mu_m(s,\beta)\hat{w}_m\overline{\hat{\eta}_m}
=\sum_{m=m_0}^{\infty}\mu_m(s,\beta)(w,e_m)_{L^{2}}(\eta,e_m)_{L^{2}},
\end{align*}\\
and \eqref{eq:BilinearSine} holds. We show (i). For $-m_0<\beta<m_0+2-2s$, $\mu_{m}(s,\beta)$ is positive and increasing by Proposition \ref{p: Symbolsign}. Thus,  

\begin{align*}
B(w,w)\geq \mu_{m_0} (s,\beta)\sum_{m= m_0}^{\infty}\left|(w,e_n)_{L^{2}}\right|^{2}=\mu_{m_0}(s,\beta) ||w||_{L^{2}}^{2}.
\end{align*}\\
We show (ii). For $-m_0-2s<\beta\leq -m_0$ or $m_0+2-2s\leq \beta< m_0+2$, the first two $\mu_{m}(s,\beta)$ can be non-positive by Proposition \ref{p: Symbolsign}. We set non-positive eigenvalues by $\mu_{m_0}(s,\beta)$ and $\mu_{m_0+l}(s,\beta)$ and the subspace by $Y=\textrm{span}\ (e_{m_0}, e_{m_0+l})$ for some $l\in\{0,1\}$. We apply the direct sum decomposition $L^{2}_{m_0,\textrm{odd}}(\mathbb{T})=Y\oplus Y^{\perp}$ and decompose $w \in H^{s}_{m_0,\textrm{odd}}(\mathbb{T})$ into $w=y+z\in Y\oplus Y^{\perp}$. Since $y\in H^{s}_{m_0,\textrm{odd}}(\mathbb{T})\oplus H^{s}_{m_0+l,\textrm{odd}}(\mathbb{T})$, $z=w-y\in H^{s}_{m_0+l+1,\textrm{odd}}(\mathbb{T})=Z$. Thus, the direct sum decomposition of \eqref{eq:LB2} holds. Observe that 

\begin{align*}
B(e_n,e_k)=\sum_{m=m_0}^{\infty}\mu_m(s,\beta) (e_n,e_m)_{L^{2}}(e_k,e_m)_{L^{2}}
=\mu_n(s,\beta) \delta_{n,k},\quad n,k\geq m_0.
\end{align*}\\
By using this identity, for $y\in Y$ and $z\in Z$,

\begin{align*}
B(y,z)&=B\left(\sum_{m=m_0}^{m_0+l}(y,e_{m})_{L^{2}}e_{m} , \sum_{k=m_0+l+1}^{\infty}(z,e_k)_{L^{2}}e_k\right)=\sum_{m=m_0}^{m_0+l}\sum_{k=m_0+l+1}^{\infty}(y,e_{m})_{L^{2}}(z,e_k)_{L^{2}} B\left(e_{m} , e_k\right)=0,\\
B(y,y)&=B\left(  \sum_{m=m_0}^{m_0+l}(y,e_{m})_{L^{2}}e_{m}  , \sum_{k=m_0}^{m_0+l}(y,e_{k})_{L^{2}}e_{k}  \right)
=\sum_{m=m_0}^{m_0+l}|(y,e_{m})_{L^{2}}|^{2}B(e_{m},e_{m})
\leq \mu_{m_0}(s,\beta) ||y||_{L^{2}}^{2},\\
B(y,y)&=\sum_{m=m_0}^{m_0+l}|(y,e_{m})_{L^{2}}|^{2}B(e_{m},e_{m})
\geq \mu_{m_0+l}(s,\beta) ||y||_{L^{2}}^{2},\\
B(z,z)&= \sum_{m=m_0+l+1}^{\infty}\mu_m(s,\beta) |(z,e_m)_{L^{2}}|^{2}
\geq \mu_{m_0+l+1}(s,\beta) ||z||_{L^{2}}^{2}.
\end{align*}\\
Thus, \eqref{eq:LB2} holds. 
\end{proof}

\vspace{5pt}

We show a lower bound for the bilinear form on the subspace associated with positive eigenvalues in $H^{s}_{m_0,\textrm{odd}}(\mathbb{T})$.

\vspace{5pt}

\begin{lem}\label{SLB}
Let $0<s<1$ and $m_0\in \mathbb{N}$. Let $\beta$ satisfy \eqref{eq: finitebm_0}.\\
\noindent
(i) For $\beta$ satisfying \eqref{eq: positivebm_0}, there exists $\delta_0>0$ such that

\begin{align}
B(w,w)\geq \delta_0 ||w||_{H^{s}}^{2},\quad w\in H^{s}_{m_0,\textrm{odd}}(\mathbb{T}).  \label{eq:SLB1}
\end{align}\\
(ii) For $\beta$ not satisfying \eqref{eq: positivebm_0}, there exists $\delta_1>0$ such that

\begin{align}
B(z,z)\geq \delta_1 ||z||_{H^{s}}^{2},\quad z\in Z.  \label{eq:SLB2}
\end{align}
\end{lem}

\vspace{5pt}

\begin{proof}
We give a proof for \eqref{eq:SLB1}. The proof for \eqref{eq:SLB2} is similar. By \eqref{eq:LB1}, 

\begin{align*}
\delta_0=\inf\{B(w,w)\ |\ w\in H^{s}_{m_0,\textrm{odd}}(\mathbb{T}),\ ||w||_{H^{s}}=1\ \}\geq 0.
\end{align*}\\
We show that $\delta_0$ is positive. We take a sequence $\{w_n\}\subset H^{s}_{m_0,\textrm{odd}}(\mathbb{T})$ such that $||w_n||_{H^{s}}=1$ and $B(w_n,w_n)\to \delta_0$. By the compact embedding $H^{s}(\mathbb{T})\subset \subset H^{r}(\mathbb{T})$ for $r<s$, we take a subsequence (still denoted by $\{w_n\}$) such that $w_n\rightharpoonup w$ in $H^{s}_{m_0,\textrm{odd}}(\mathbb{T})$ and $w_n\to w$ in $H^{r}(\mathbb{T})$ for some $w\in H^{s}_{m_0,\textrm{odd}}(\mathbb{T})$. For $N\geq 1$, 

\begin{align*}
B(w_n,w_n)
=\sum_{m=m_0}^{\infty}\mu_m(s,\beta) |(w_n,e_m)_{L^{2}}|^{2}
\geq \sum_{m=m_0}^{N}\mu_m(s,\beta) |(w_n,e_m)_{L^{2}}|^{2}.
\end{align*}\\
By letting $n\to\infty$ and $N\to\infty$ and by \eqref{eq:LB1}, 

\begin{align*}
\delta_0=\lim_{n\to\infty} B(w_n,w_n)\geq B(w,w)\geq \mu_{m_0}(s,\beta)||w||_{L^{2}}^{2}.
\end{align*}\\
If $w\neq 0$, $\delta_0$ is positive. We may assume that $w=0$. By \eqref{eq:bilinear2} and letting $n\to\infty$,

\begin{align*}
\sum_{m=m_0}^{\infty}\kappa_m(s,\beta)(1+|m|) |(w_n,e_m)_{L^{2}}|^{2}
\lesssim  ||w_n||_{H^{-\frac{1}{2}+s}}^{2} \to 0.
\end{align*}\\
By \eqref{eq:bilinear1} and $||w_n||_{H^{s}}=1$,  

\begin{align*}
\delta_0=\lim_{n\to\infty}B(w_n,w_n)\geq C\lim_{n\to\infty}||w_n||_{H^{s}}
= C>0.
\end{align*}\\ 
We thus conclude.
\end{proof}

\vspace{5pt}

\subsection{The Palais--Smale condition}

We show the (PS$)_{c}$ condition for the functional \eqref{eq:F} for all $c\in \mathbb{R}$. We apply the lower bounds \eqref{eq:SLB1} for $\beta$ satisyfing \eqref{eq: positivebm_0} and \eqref{eq:SLB2} for $\beta$ not satisfying \eqref{eq: positivebm_0}, respectively.

\vspace{5pt}

\begin{prop}\label{p:superlinearPSS}
Let $(m_0,s,\beta)$ satisfy \eqref{eq:AD}. The functional $I\in C^{1}(X; \mathbb{R})$ in \eqref{eq:F} satisfies 

\begin{align}
I[w]-\sigma \langle I'[w],w\rangle=\left(\frac{1}{2}-\sigma \right)B(w,w)+c\left(\sigma-\frac{\beta}{2(\beta+s)} \right)\int_{\mathbb{T}}|w|^{2+\frac{2s}{\beta}}d\theta,\quad w\in H^{s}_{m_0,\textrm{odd}}(\mathbb{T}), \quad \sigma\in \mathbb{R}.  \label{eq:FI}
\end{align} \\
Assume that $\beta$ satisfies \eqref{eq: positivebm_0}. Then,  

\begin{align}
I[w]-\sigma \langle I'[w],w \rangle\ \geq \left(\frac{1}{2}-\sigma \right)\delta_0||w||_{H^{s}}^{2},\quad w\in H^{s}_{m_0,\textrm{odd}}(\mathbb{T}),\quad \frac{\beta}{2(\beta+s)}<\sigma<\frac{1}{2}.  \label{eq:Superlinear}
\end{align}
\end{prop}

\vspace{5pt}

\begin{proof}
The identity \eqref{eq:FI} follows from \eqref{eq:F} and \eqref{eq:FD}. The estimate \eqref{eq:Superlinear} follows from \eqref{eq:FI}.
\end{proof}

\vspace{5pt}

\begin{prop}\label{p:PSbounded}
Let $(m_0,s,\beta)$ satisfy \eqref{eq:AD}. Then any sequence $\{w_n\}\subset H^{s}_{m_0,\textrm{odd}}(\mathbb{T})$ satisfying 

\begin{equation}
\begin{aligned}
&\sup_{n\geq 1}I[w_n]<\infty,\\
&I'[w_n]\to 0\quad \textrm{on}\ H^{-s}_{m_0,\textrm{odd}}(\mathbb{T}),
\end{aligned}
\label{eq:PSS}
\end{equation}\\
are bounded on $H^{s}_{m_0,\textrm{odd}}(\mathbb{T})$.
\end{prop}

\vspace{5pt}

\begin{proof}
For $\beta$ satisfying \eqref{eq: positivebm_0}, the result follows from Proposition \ref{p:superlinearPSS}. We consider $\beta$ not satisfying \eqref{eq: positivebm_0}. Namely, $-m_0-2s<\beta\leq -m_0$ or $m_0+2-2s\leq \beta< m_0+2$. Suppose that $M_n=||w_n||_{H^{s}}$ diverges. Then, $\tilde{w}_n=w_n/M_n$ satisfies $||\tilde{w}_n||_{H^{s}}=1$. By the direct sum decomposition in Lemma \ref{l:DSD}, $\tilde{w}_n=\tilde{y}_n+\tilde{z}_n\in Y\oplus Z$ satisfies $||\tilde{w}_n||_{L^{2}}^{2}=||\tilde{y}_n||_{L^{2}}^{2}+||\tilde{z}_n||_{L^{2}}^{2}$. Since $Y$ is finite-dimensional, $\tilde{y}_n$ and $\tilde{z}_n$ are bounded in $H^{s}(\mathbb {T})$.

We first consider the superlinear case $m_0+2-2s\leq \beta< m_0+2$. We apply the identity \eqref{eq:FI} for $w_n$ and $\beta/(2\beta+2s)<\beta<1/2$ and divide it by $M_m^{2+2s/\beta}$. Then, $\tilde{w}_{n}\to 0$ in $L^{2+2s/\beta}(\mathbb{T})$. In particular, $\tilde{w}_{n}\to 0$ in $L^{2}(\mathbb{T})$. This implies that $\tilde{y}_n, \tilde{z}_n\to 0$ in $L^{2}(\mathbb{T})$. Since $Y$ is finite-dimensional, $\tilde{y}_n\to 0$ in $H^{s}(\mathbb{T})$. Since $\tilde{w}_n$ and $\tilde{z}_n$ are bounded in $H^{s}(\mathbb{T})$ and vanish on $L^{2}(\mathbb{T})$, $\tilde{w}_n$ and $\tilde{z}_n$ also vanish on $H^{r}(\mathbb{T})$ for $r<s$. By applying the bilinear estimates \eqref{eq:bilinear1} and \eqref{eq:bilinear2} for $\tilde{w}_n$, we obtain $\lim_{n\to\infty}B(\tilde{w}_n,\tilde{w}_n)\neq 0$. By using \eqref{eq:LB2}, $B(\tilde{w}_n,\tilde{w}_n)=B(\tilde{y}_n,\tilde{y}_n)+B(\tilde{z}_n,\tilde{z}_n)$ and $\lim_{n\to\infty}B(\tilde{z}_n,\tilde{z}_n)\neq 0$. By the bilinear estimates \eqref{eq:bilinear1}, $\lim_{n\to\infty}||\tilde{z}_n||_{H^{s}}^{2}\geq C$ for some constant $C>0$. By dividing \eqref{eq:FI} for $w_n$ by $M_n$ and using \eqref{eq:LB2} and \eqref{eq:SLB2},

\begin{align*}
\frac{1}{M_n^{2}}\left(I[w_n]-\sigma \langle I'[w_n],w_n\rangle\right)\geq \left(\frac{1}{2}-\sigma \right)B(\tilde{w}_n,\tilde{w}_n)\geq \left(\frac{1}{2}-\sigma \right)\left( -\mu_{m_0}||\tilde{y}_n||_{H^{s}}^{2}+\delta_1||\tilde{z}_n||_{H^{s}}^{2} \right).
\end{align*}\\
Letting $n\to\infty$ yields the contradiction $0\geq (1/2-\sigma)\delta_1 C>0$.

We next consider the sublinear case $-m_0-2s<\beta\leq -m_0$. We substitute $\tilde{\eta}_n=-\tilde{y}_n+\tilde{z}_n$ into \eqref{eq:FD} for $w_n$ and observe that 

\begin{align*}
\frac{1}{M_n}\langle I_n'[w_n],\tilde{\eta}_{n}\rangle
=B(\tilde{w}_n,\tilde{\eta}_n)
-\frac{1}{M_{n}^{-\frac{2s}{\beta}}}\langle c\tilde{w}_n |\tilde{w}_n|^{\frac{2s}{\beta}},\tilde{\eta}_{n}\rangle,
\end{align*}\\
and $\lim_{n\to\infty}B(\tilde{w}_n,\tilde{\eta}_n)=0$. By \eqref{eq:LB2}, there exists $C>0$ such that $-B(y,y)\geq C||y||_{L^{2}}^{2}$ for $y\in Y$. By \eqref{eq:SLB2},

\begin{align*}
0=\lim_{n\to\infty}B(\tilde{w}_n,\tilde{\eta}_n)
=\lim_{n\to\infty}\left(-B(\tilde{y}_n,\tilde{y}_n)+B(\tilde{z}_n,\tilde{z}_n)\right)
\geq \lim_{n\to\infty}\left(C||\tilde{y}_n||_{L^{2}}+\delta_1 ||\tilde{z}_n ||_{H^{s}}\right).
\end{align*}\\
This yields the contradiction $1=\lim_{n\to\infty}||\tilde{w}_n||_{H^{s}}=0$.

We thus conclude that $\{w_n\}$ are bounded on $H^{s}_{m_0,\textrm{odd}}(\mathbb{T})$ for $-m_0-2s<\beta\leq -m_0$ and $m_0+2-2s\leq \beta< m_0+2$.
\end{proof}

\vspace{5pt}

\begin{lem}\label{l:PSC}
Let $(m_0,s,\beta)$ satisfy \eqref{eq:AD}. The functional $I\in C^{1}(X; \mathbb{R})$ in \eqref{eq:F} satisfies (PS$)_c$ condition for any $c\in \mathbb{R}$.   
\end{lem}

\vspace{5pt}

\begin{proof}
We take a Palais--Smale sequence $\{w_m\}\subset H^{s}_{m_0,\textrm{odd}}(\mathbb{T})$ at level $c\in \mathbb{R}$. By Propositions \ref{p:PSbounded}, $\{w_m\}$ is bounded in $H^{s}(\mathbb{T})$. By possibly choosing a subsequence, $w_n\rightharpoonup w$ in $H^{s}(\mathbb{T})$ and $w_n\to w$ in $L^{2}(\mathbb{T})$. By the weak convergence,

\begin{align*}
\lim_{n\to\infty} \langle I'[w],w-w_n\rangle =0.
\end{align*}\\
By the second condition of \eqref{eq:PSS},

\begin{align*}
\lim_{n\to\infty}\langle I'[w_n],w-w_n\rangle=0.
\end{align*}\\
By \eqref{eq:FD},

\begin{align*}
\langle I'[w]-I'[w_n],w-w_n\rangle =B(w-w_n,w-w_n)-c\left\langle w|w|^{\frac{2s}{\beta}}-w_n|w_n|^{\frac{2s}{\beta}},w-w_n\right\rangle.
\end{align*}\\
By the compact embedding \eqref{eq:CE}, the second term on the right-hand side vanishes as $n\to\infty$. By \eqref{eq:bilinear1} and \eqref{eq:bilinear2}, $w_n\to w$ in $H^{s}(\mathbb{T})$. Thus, $I$ satisfies the (PS$)_c$ condition.  
\end{proof}

\vspace{5pt}

\subsection{Functional estimates on subsets}

It remains to show estimates of the functional \eqref{eq:F} on proper subsets in the following three different regimes.\\

\begin{lem}\label{FE}
Let $(m_0,s,\beta)$ satisfy \eqref{eq:AD}. Then, the functional $I\in C^{1}(X; \mathbb{R})$ satisfies the following:

\noindent
(i) For $-m_0<\beta<-2s$, 

\begin{align}
\inf_{X}I>-\infty. \label{eq:FE1}
\end{align}\\
(ii) For $0<\beta<m_0+2-2s$, there exists $r_0>0$ and $w_0\in X$ such that $||w_0||_{X}>r_0$ and 

\begin{align}
\inf_{||w||_{X}=r_0 }I[w]> I[0]\geq I[w_0]. \label{eq:FE2}
\end{align}\\
(iii) For $m_0+2-2s\leq \beta<m_0+2$, there exists $\rho_0>r_0>0$ and $z_0\in Z$ such that 

\begin{equation}
\begin{aligned}
&\inf\left\{I[z]\ |\ ||z||_{H^{s}}=r_0,\ z\in Z\  \right\} \\
&>0\geq \max\left\{I[w]\ |\ w=y+\lambda z\in Y\oplus \mathbb{R}z_0,\ \lambda=0 \  \textrm{and}\ ||y||_{H^{s}}\leq \rho_0,\ \textrm{or}\ \lambda>0\ \textrm{and}\ ||w||_{H^{s}}=\rho_0\ \right\}
\end{aligned}
\label{eq: FE3}
\end{equation}\\
(iv) For $-m_0-2s<\beta\leq - m_0$, there exists $\rho_0>0$ such that 

\begin{align}
\inf_{Z}I>\max\{I[y]\ |\ ||y||_{H^{s}}=\rho_0,\ y\in Y\ \}.  \label{eq:FE4}
\end{align}
\end{lem}

\vspace{5pt}

\begin{proof}
We show (i). We take arbitrary $w\in H^{s}_{m_0,\textrm{odd}}(\mathbb{T})$. Since $1<2+2s/\beta<2$ for $-m_0<\beta<-2s$, we set $p=2/(2+2s/\beta)>1$ and apply H\"older's inequality to estimate 

\begin{align*}
\int_{\mathbb{T}}|w|^{2+\frac{2s}{\beta}}d\theta\leq ||w||_{L^{2}}^{\frac{2}{p}}(2\pi)^{\frac{1}{q}},
\end{align*}\\
where $1/p+1/q=1$. By Young's inequality, 

\begin{align*}
I[w]=\frac{1}{2}B(w,w)-\frac{c\beta}{2(\beta+s)}\int_{\mathbb{T}}|w|^{2+\frac{2s}{\beta}}d\theta
\geq \frac{1}{2}B(w,w)-\varepsilon||w||_{L^{2}}^{2}-\frac{C}{\varepsilon},
\end{align*}\\
for $\varepsilon>0$ and some constant $C>0$. Since  $|\beta|<m_0$, applying \eqref{eq:SLB1} implies that  

\begin{align*}
I[w]\geq \left(\frac{\delta_0}{2}-\varepsilon\right)||w||_{H^{s}}^{2}-\frac{C}{\varepsilon}.
\end{align*}\\
Thus, $\inf_{X}I$ is bounded from below. 

We show (ii). We take $r_0>0$ and $w\in H^{s}_{m_0,\textrm{odd}}(\mathbb{T})$ such that $||w||_{H^{s}}=r_0$. By the continuous embedding \eqref{eq:CE} and the lower bound of the bilinear form \eqref{eq:SLB2}, $||w||_{L^{2+2s/\beta}}\lesssim ||w||_{H^{s}}$ and 

\begin{align*}
I[w]=\frac{1}{2}B(w,w)-\frac{c\beta}{2(\beta+s)}\int_{\mathbb{T}}|w|^{2+\frac{2s}{\beta}}d\theta
\geq \frac{\delta_0}{2}||w||_{H^{s}}^{2}-C||w||_{H^{s}}^{2+\frac{2s}{\beta}}
=r_0^{2}\left(\frac{\delta_0}{2}-Cr_0^{\frac{2s}{\beta}}\right).
\end{align*}\\
By choosing sufficiently small $r_0>0$, $\inf_{||w||_{X}=r_0}I>0$. For $\tilde{w}_0\in H^{s}_{m_0,\textrm{odd}}(\mathbb{T})$ such that $||\tilde{w}_0||_{X}=r_0$,

\begin{align*}
I[\rho \tilde{w}_0]=\rho^{2}\left(\frac{1}{2}B(\tilde{w}_0,\tilde{w}_0)-\rho^{\frac{2s}{\beta}}\frac{c\beta}{2(\beta+s)}\int_{\mathbb{T}}|\tilde{w}_0|^{2+\frac{2s}{\beta}}d\theta\right)\to-\infty\quad \textrm{as}\ \rho\to\infty.
\end{align*}\\
Thus, there exists $\rho_0>1$ such that $0\geq I[w_0]$ for $w_0=\rho_0\tilde{w}_0$.

We show (iii). We take $r_0>0$ and $z\in Z$ such that $||z||_{H^{s}}=r_0$. In a similar way to (ii), we use the lower bound of the bilinear form \eqref{eq:SLB1} and estimate

\begin{align*}
I[z]=\frac{1}{2}B(z,z)-\frac{c\beta}{2(\beta+s)}\int_{\mathbb{T}}|z|^{2+\frac{2s}{\beta}}d\theta
\geq \frac{\delta_1}{2}||z||_{H^{s}}^{2}-C||z||_{H^{s}}^{2+\frac{2s}{\beta}}
=r_0^{2}\left(\frac{\delta_1}{2}-Cr_0^{\frac{2s}{\beta}}\right).
\end{align*}\\
By choosing sufficiently small $r_0>0$, $\inf\{ I[z]\ |\ z\in Z\ |\ ||z||_{H^{s}}=r_0  \}>0$. We set $z_0=r_0e_{m_0+l+1}/||e_{m_0+l+1}||_{H^{s}}$. For $y\in Y$ and $\lambda>0$, we set $w=y+\lambda z_0\in Y\oplus \mathbb{R}z_0$. Then, for $\rho>0$, $I[\rho w]\to -\infty$ as $\rho\to\infty$ since $\beta>0$ similarly as in the proof of (ii). Since $Y\oplus \mathbb{R}z_0$ is finite-dimensional, there exists $\rho_0>r_0$ such that 

\begin{align*}
\max\left\{I[w]\ |\ w=y+\lambda z\in Y\oplus \mathbb{R}z_0,\  \lambda>0\ \textrm{and}\ ||w||_{H^{s}}=\rho_0\ \right\}\leq 0.
\end{align*}\\
For $\lambda=0$, $B(y,y)\leq \mu_{m_0}(s,\beta)||y||_{L^{2}}^{2}\leq 0$ by \eqref{eq:LB2}. Thus,  

\begin{align*}
\max\left\{I[w]\ |\ w=y+\lambda z\in Y\oplus \mathbb{R}z_0,\ \lambda=0\ \textrm{and}\ ||w||_{H^{s}}\leq \rho_0\ \textrm{or}\   \lambda>0\ \textrm{and}\ ||w||_{H^{s}}=\rho_0\ \right\}\leq 0.
\end{align*}\\
It remains to show (iv). We take an arbitrary $z\in Z$. Since $2+2s/\beta<2$, we  apply Young's inequality for $p=2/(2+2s/\beta)$ and $\varepsilon>0$ to estimate

\begin{align*}
||z||_{H^{s}}^{2+\frac{2s}{\beta}}\leq \frac{1}{p}\varepsilon ||z||_{H^{s}}^{2}+\frac{1}{q\varepsilon^{\frac{q}{p}}},
\end{align*}\\
where $1/p+1/q=1$. By the continuous embedding \eqref{eq:CE} and \eqref{eq:SLB2}, 

\begin{align*}
I[z]=\frac{1}{2}B(z,z)-\frac{c\beta}{2(\beta+s)}\int_{\mathbb{T}}|z|^{2+\frac{2s}{\beta}}d\theta\geq \left(\frac{\delta_1}{2}-\varepsilon\right)||z||_{H^{s}}^{2} -C.
\end{align*}\\
Thus, $\inf_{Z}I>-\infty$. For an arbitrary $y\in Y$ and $\rho=||y||_{H^{s}}$, $B(y,y)\leq 0$ and  

\begin{align*}
I[y]=\frac{1}{2}B(y,y)-\frac{c\beta}{2(\beta+s)}\int_{\mathbb{T}}|y|^{2+\frac{2s}{\beta}}d\theta
\leq -\frac{c\beta}{2(\beta+s)}\rho^{2+\frac{2s}{\beta}}\inf\left\{||y||_{L^{2+\frac{2s}{\beta}}}^{2+\frac{2s}{\beta}}\ \middle|\ y\in Y,\ ||y||_{H^{s}}=1\   \right\}.
\end{align*}\\
Thus, there exists $\rho_0>0$ such that $\max\{I[y]\ |\ y\in Y,\ ||y||_{H^{s}}=\rho_0 \}<\inf_{Z}I$.
\end{proof}

\vspace{5pt}

\begin{proof}[Proof of Theorem \ref{t:thm} (Existence)]
For $(m_0,s,\beta)$ satisfying \eqref{eq:AD}, the functional $I\in C^{1}(X; \mathbb{R})$ satisfies (PS$)_c$ condition for any $c\in \mathbb{R}$ by Lemma \ref{l:PSC} and the inequalities \eqref{eq:FE1}-\eqref{eq:FE4} by Lemma \ref{FE}. We apply Lemmas \ref{l:Min}, \ref{l:MP}, \ref{l:LT}, and \ref{l:SP}, and deduce the existence of critical points of $I$. By Lemmas \ref{l:1/2<s}, \ref{l:s=1/2} and \ref{l:s<1/2}, we obtain the desired regularity of critical points. 
\end{proof}

\section{Numerical experiments}\label{s:7}

In this section we describe how to numerically compute some of the solutions proved in Theorem \ref{t:thm}.  See below for different depictions of the solutions.

Given the explicit expression of the multiplier, the computation has been reduced to solving \eqref{eq:HSgSQG} for any given $m_0$ in the subspace of functions with frequencies bounded by $M$. In order to account for potential cancellations between the Gamma functions, we use log-Gamma function evaluations instead of Gamma function ones. Previous expressions of the multiplier \eqref{eq:Formula} (which had a more cumbersome formula as an infinite series) were implemented using a Wynn $\varepsilon$-acceleration method.

To compute a solution for a fixed $(s,\beta,m_0)$, we used the Levenberg-Marquardt algorithm \cite{Levenberg:Levenberg-Marquardt,Marquardt:Levenberg-Marquardt}. 
Our discretization variables were the values of the Fourier coefficients of $w$. From those, we calculated the Fourier coefficients of $g$ , and we evaluated  equation \eqref{eq:HSgSQG} at  gridpoints $x_i = \frac{2\pi}{N}i, \, 0 \leq i \leq N-1$ via FFT with an additional normalization in $w$ to have $L^2$-norm equal to 1. The initial condition was taken to be mode $m_0$ equal to 1 and the rest of the modes equal to 0. We did not appreciate any convergence issues and the solver finished within a few iterations.

In Figure \ref{fig:fixed_beta} we fixed $\beta = -1.6$ and varied $s$, even going beyond the threshold $\beta = -2s$ that leads to the Bahouri-Chemin solution where $g$ is a step function. In the case where $-2s < \beta$ we see that $g$ is suggested to be unbounded. In all region $-2s \leq \beta$ the solver struggles with a Gibbs phenomenon due to the tails being unresolved and the $g$ part of the solution presents some oscillations. In Figure \ref{fig:fixed_s} (where we fix $s$ instead of $\beta$) we see the same phenomenon as in the previous figure. We remark that in both Figures \ref{fig:fixed_beta}-\ref{fig:fixed_s} both solutions have limited regularity due to the algebraic decay of the tails.

\begin{figure}[h]
\includegraphics[scale=0.3]{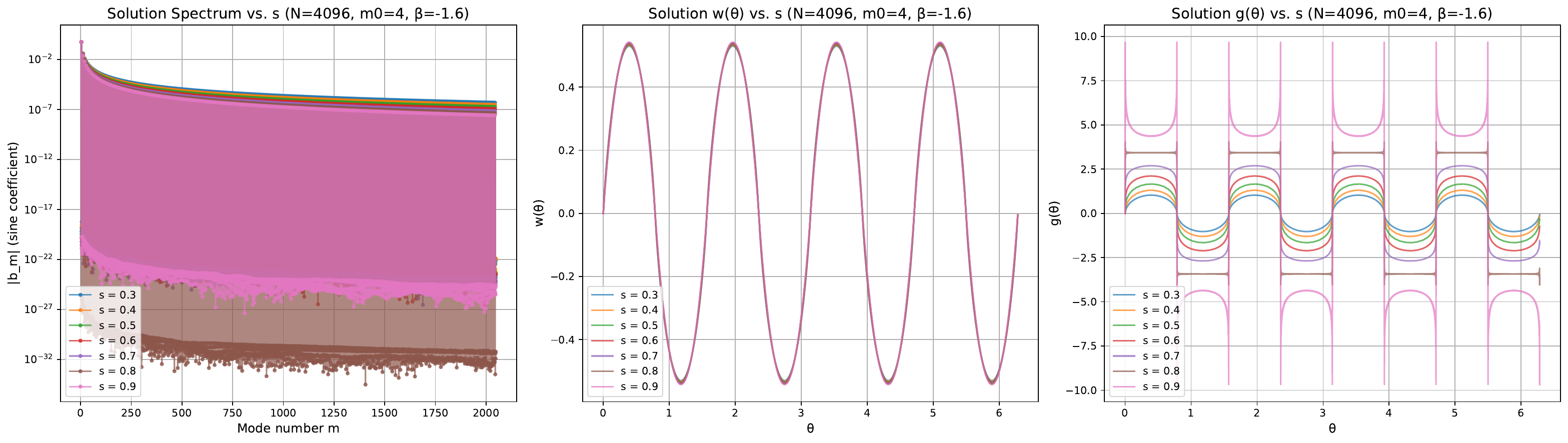}
\caption{Solution for $N = 4096$, $m_0 = 4$, $\beta = -1.6$ and $s \in \{0.3, 0.4, 0.5, 0.6, 0.7, 0.8, 0.9\}$. \\
Left: Fourier Spectrum of $w$, Middle: $w(\theta)$, Right: $g(\theta)$.}
\label{fig:fixed_beta}
\end{figure}

\begin{figure}[h]
\includegraphics[scale=0.3]{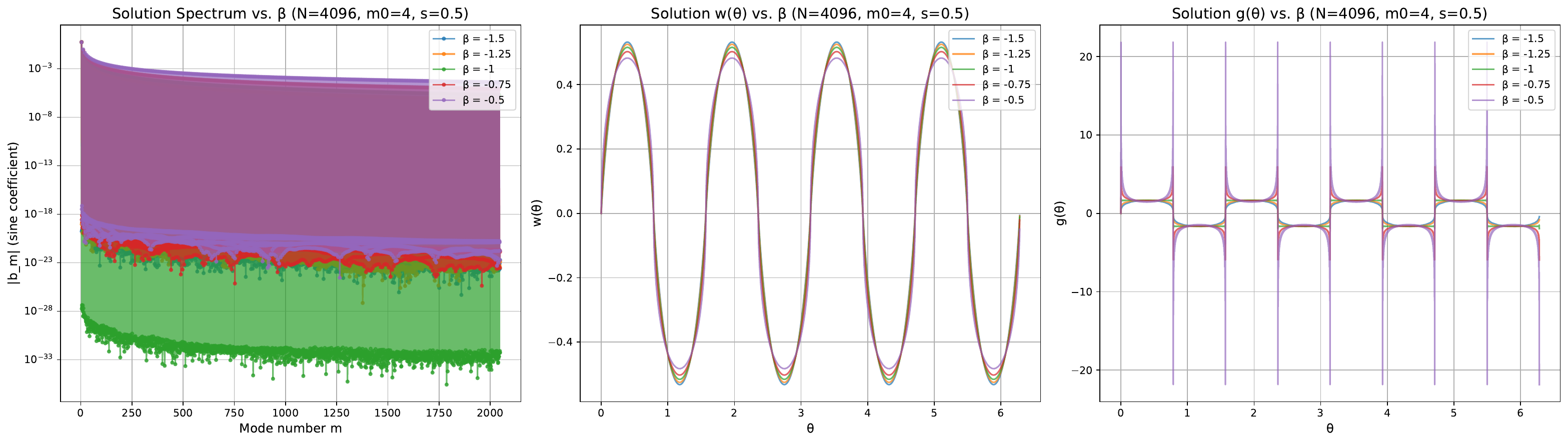}
\caption{Solution for $N = 4096$, $m_0 = 4$, $s = 0.5$ and $\beta \in \{-1.5, -1.25, -1, -0.75, -0.5\}$. \\
Left: Fourier Spectrum of $w$, Middle: $w(\theta)$, Right: $g(\theta)$.}
\label{fig:fixed_s}
\end{figure}

In Figure \ref{fig:fixed_beta_small} we tried to capture the limit $\beta \to 0$ that formally gives rise to $g(\theta)$ being a Dirac delta (or a sum of Dirac deltas). We fixed $N = 4096, m_0 = 4, s = 0.5$ and pushed $\beta \to 0$. We can visually see how $g$ converges to a Dirac delta while $w$ converges to some (potentially unbounded) limit.

\begin{figure}[h]
\includegraphics[scale=0.3]{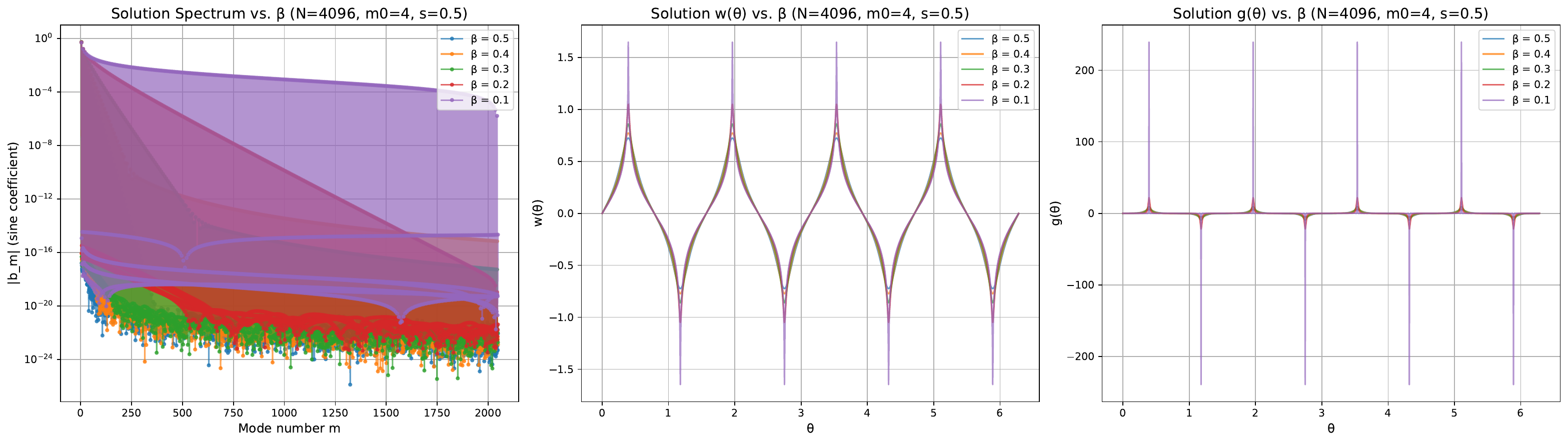}
\caption{Solution for $N = 4096$, $m_0 = 4$, $s = 0.5$ and $\beta \in \{0.1, 0.2, 0.3, 0.4, 0.5\}$. \\
Left: Fourier Spectrum of $w$, Middle: $w(\theta)$, Right: $g(\theta)$.}
\label{fig:fixed_beta_small}
\end{figure}

Finally, we performed a resolution study in $N$ by fixing $m_0 = 4, s = 0.2, \beta = -0.6$ and increasing $N$. We see a very accurate matching of the solutions (cf. Left panel) as well as a decay in the oscillations and convergence (Middle and Right panels). The leftmost panel suggests a limited amount of regularity due to the algebraic decay of the tails. In order to quantify that more accurately, we fit the Fourier coefficients to the function $C/m^\alpha$ using least squares. A side to side comparison can be found in Figure \ref{fig:spectrum}. We found that $\alpha = 1.589$, leading to a regularity of $w \in C^{0.589}$ which is very close to the expected $w \in C^{-\beta}$ (i.e. $w \in C^{0.6}$ in this case). These results suggest a possible pathway to improve the simulation by using other (non-Fourier) elements that can cancel out the algebraic decay of the Fourier coefficients (or similarly, the singular part of the solution). A good choice of complementary elements might be Clausen functions (see \cite{Dahne:highest-cusped-waves-fkdv,Dahne-GomezSerrano:highest-wave-burgershilbert} for instances where they have been successfully used) due to their well-suitedness with Fourier multipliers.

\begin{figure}[h]
\includegraphics[scale=0.3]{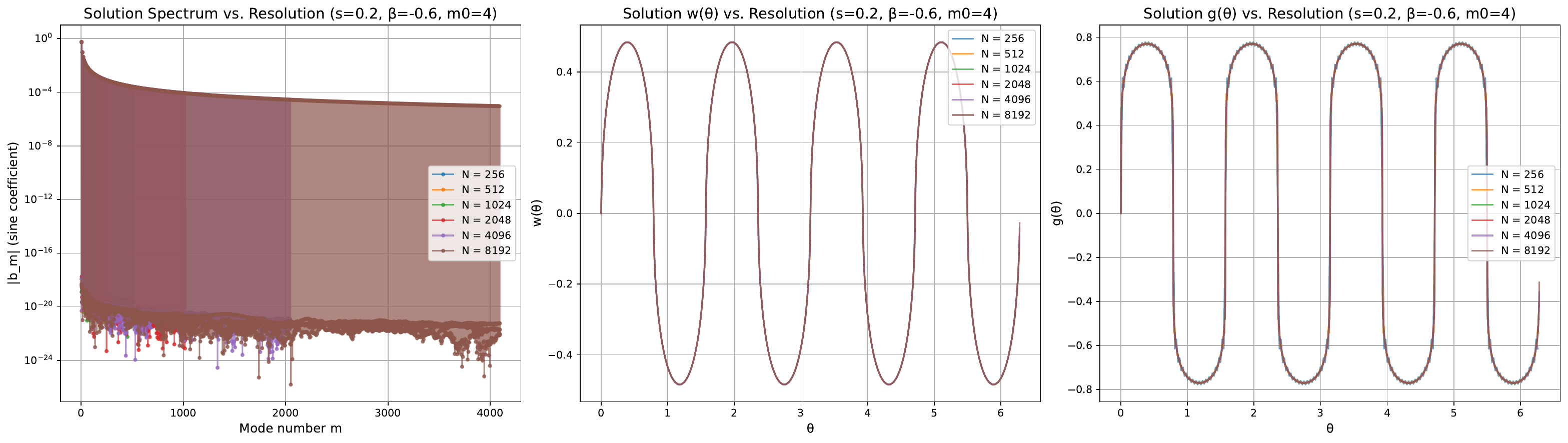}
\label{fig:fixed_sbeta_Nvariable}
\caption{Solution for $m_0 = 4$, $s = 0.2$, $\beta = -0.6$ and $N \in \{ 256, 512, 1024, 2048, 4096, 8192\}$. \\
Left: Fourier Spectrum of $w$, Middle: $w(\theta)$, Right: $g(\theta)$.}
\end{figure}

\begin{figure}[h]
\includegraphics[scale=0.5]{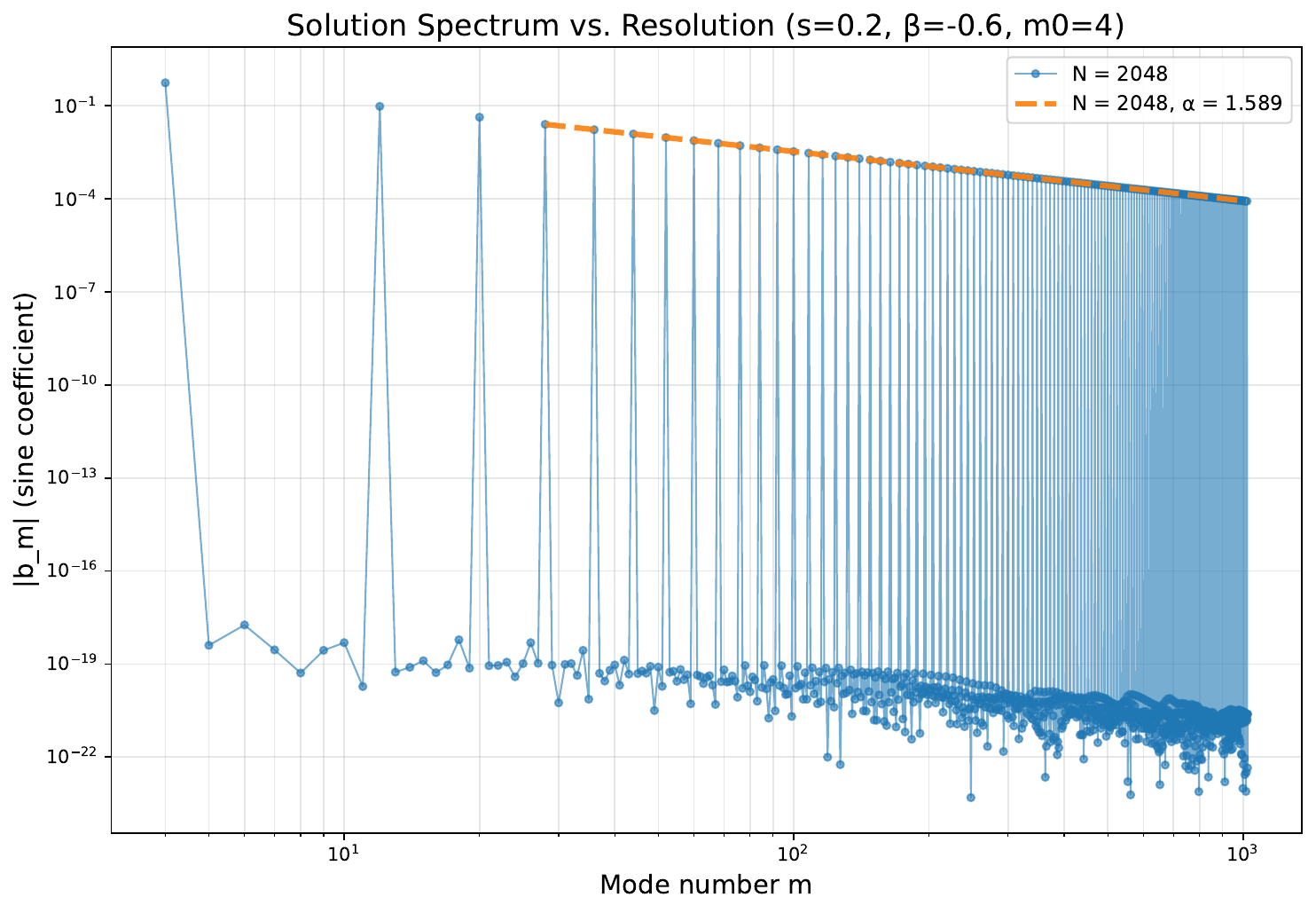}\caption{Solution for $m_0 = 4$, $s = 0.2$, $\beta = -0.6$ and $N = 2048$. Fit to a power law.}\label{fig:spectrum}
\end{figure}

\appendix

\section{Monotonicity of the ratio of the Gamma functions}\label{a:A}

We show monotonicity of the ratio of the Gamma functions used in the proof of Proposition \ref{p: monotonem} by using properties of the digamma function $\psi(x)$.

\begin{prop}\label{p: digammaestimate}
Let $0<s<1$. Then,

\begin{align}
0<\psi(r+1)-\psi(r+1-s)\leq \frac{s}{r},\quad r>0.  \label{eq: digammaesti}
\end{align}
\end{prop}

\begin{proof}
The left inequality follows from the monotonicity of the digamma function $\psi(x)$ for $x>0$. We use the representation formula \cite[6.3.22]{AbS},

\begin{align*}
\psi(x)=-\gamma+\int_{0}^{1}\frac{1-t^{x-1}}{1-t}dt,\quad x>0,
\end{align*}\\
with Euler's constant $\gamma=0.5772\cdots$. By using 

\begin{align*}
1-t^{s}=\int_{t}^{1}s\rho^{s-1}d\rho\leq s(1-t)t^{s-1},
\end{align*}\\
we obtain

\begin{align*}
0<\psi(r+1)-\psi(r+1-s)=\int_{0}^{1}\frac{t^{r-s}-t^{r}}{1-t}dt
=\int_{0}^{1}\frac{t^{r-s}(1-t^{s})}{1-t}dt\leq s\int_{0}^{1}t^{r-1}dt=\frac{s}{r}.
\end{align*}
\end{proof}

\begin{lem}\label{l: monotonegamma}
Let $0<s<1$. The function $f(t)=\Gamma(t+s)/\Gamma(t)$ is increasing for $t>-s$.
\end{lem}

\vspace{5pt}

\begin{proof}
The function $f(t)$ is positive for $t>0$. By differentiating $\log{f(t)}=\log{\Gamma(t+s)}-\log{\Gamma(t)}$, 

\begin{align*}
\frac{f'(t)}{f(t)}=\psi(t+s)-\psi(t).
\end{align*}\\
Since $\psi(t)$ is increasing for $t>0$, so is $f(t)$ for $t>0$. For $-s<t\leq 0$, we set $0<r=t+s\leq s$ and use the property $\Gamma(x+1)=x\Gamma(x)$ to observe that

\begin{align*}
f(t)=\frac{\Gamma(t+s)}{\Gamma(t)}=\frac{\Gamma(t+s+1)}{\Gamma(t+1)}\left(1-\frac{s}{t+s}\right)=\frac{\Gamma(r+1)}{\Gamma(r+1-1)}\left(1-\frac{s}{r}\right):=g(r).
\end{align*}\\
Since $h(r)=\Gamma(r+1)/\Gamma(r+1-s)$ is positive, we differentiate $\log{h(r)}=\log{\Gamma(r+1)}-\log{\Gamma(r+1-s)}$ and observe that

\begin{align*}
\frac{h'(r)}{h(r)}=\psi(r+1)-\psi(r+1-s).
\end{align*}\\
By applying the estimate \eqref{eq: digammaesti},

\begin{align*}
g'(r)=\left(\psi(r+1)-\psi(r+1-s) \right)h(r)\left(1-\frac{s}{r}\right)+h(r)\frac{s}{r^{2}} 
&=h(r)\left\{\left(\psi(r+1)-\psi(r+1-s) \right)\left(1-\frac{s}{r}\right)+\frac{s}{r^{2}}\right\}\\
&\geq sh(r)\left(\frac{1}{r}+\frac{1-s}{r^{2}}\right)>0.
\end{align*}\\
Thus, $f(t)$ is increasing for $-s<t\leq 0$.
\end{proof}

\begin{lem}\label{l: monotonegamma2}
Let $0<s<1$ and $\beta\leq -2s$. The function 

\begin{align}
f(t)f(t-\beta+1-s)=\frac{\Gamma(t+s)\Gamma(t-\beta+1)}{\Gamma(t)\Gamma(t-\beta+1-s)}   \label{eq: gammainc}
\end{align}\\
is increasing for $t>-s$.
\end{lem}

\vspace{5pt}

\begin{proof}
Since $f(t)$ and $f(t-\beta+1-s)$ are positive and increasing for $t>0$ by Lemma \ref{l: monotonegamma}, $f(t)f(t-\beta+1-s)$ is also increasing. We consider the case $-s<t\leq 0$. We set $0<r=t+s\leq s$ and observe that

\begin{align*}
f(t)f(t-\beta+1-s)=\frac{\Gamma(r+1)\Gamma(r-\beta+1-s)}{\Gamma(r-s+1)\Gamma(r-\beta+1-2s)}\left(1-\frac{s}{r}\right):=\tilde{g}(r).
\end{align*}\\
By differentiating the gamma functions,

\begin{align*}
&\frac{d}{dr}\left(\frac{\Gamma(r+1)\Gamma(r-\beta+1-s)}{\Gamma(r-s+1)\Gamma(r-\beta+1-2s)}\right)\\
&=(\psi(r+1)-\psi(r+1-s)+\psi(r-\beta+1-s)-\psi(r-\beta+1-2s) )\left(\frac{\Gamma(r+1)\Gamma(r-\beta+1-s)}{\Gamma(r-s+1)\Gamma(r-\beta+1-2s)}\right).
\end{align*}\\
By applying the estimate \eqref{eq: digammaesti} and using $-2s-\beta\geq 0$, 

\begin{align*}
&\tilde{g}'(r)\left(\frac{\Gamma(r+1)\Gamma(r-\beta+1-s)}{\Gamma(r-s+1)\Gamma(r-\beta+1-2s)}\right)^{-1} \\
&=(\psi(r+1)-\psi(r+1-s)+\psi(r-\beta+1-s)-\psi(r-\beta+1-2s) )\left(1-\frac{s}{r^{2}}\right)+\frac{s}{r^{2}} \\
&\geq \left(\frac{s}{r}+\frac{s}{r-\beta-s} \right) \left(1-\frac{s}{r^{2}}\right)+\frac{s}{r^{2}}
=\frac{s}{r}\left(\frac{1-s}{r}+\frac{(2r-2s-\beta)}{(r-\beta-s)} \right)>0.
\end{align*}\\
Thus, $f(t)f(t-\beta+1-s)$ is increasing for $-s<t\leq 0$.
\end{proof}

\vspace{5pt}

\section{Sobolev inequality on the torus}\label{a:B}

\begin{thm}[Sobolev inequality]\label{t:Sobolev}
Let $0<s<1$ and $1<p<\infty$.\\
\noindent  
(i) If $s<1/p$,

\begin{align}
H^{s,p}(\mathbb{T})&\subset L^{q}(\mathbb{T}),\quad \frac{1}{q}=\frac{1}{p}-s, \label{eq:S1}\\
H^{s,p}(\mathbb{T})&\subset\subset L^{r}(\mathbb{T}),\quad  1\leq r<q.  \label{eq:S2}
\end{align}
(ii) If $s>1/p$, 

\begin{align}
H^{s,p}(\mathbb{T})\subset C^{\gamma}(\mathbb{T}),\quad \gamma=s-\frac{1}{p}.  \label{eq:S3}
\end{align}
\end{thm}

\vspace{5pt}

The continuous embedding \eqref{eq:S1} is due to \cite[Proposition 1.1]{BO13}. We give a proof for the compact embedding \eqref{eq:S2} and the embedding into the H\"older space \eqref{eq:S3} by using the heat kernel and Bessel potential kernel in $\mathbb{R}$:

\begin{align*}
G_t(\theta)=\frac{1}{\sqrt{4\pi t}}e^{-\frac{\theta^{2}}{4t}},\quad F_s(\theta)=\frac{1}{2^{\frac{s}{2}-1 }\Gamma(\frac{s}{2})\sqrt{2\pi}|\theta|^{\frac{1-s}{2}}}K_{\frac{1-s}{2}}(|\theta|),
\end{align*}\\
where $K_{\nu}(|\theta|)$ denotes the $\nu$-th order modified Bessel function of the second kind, e.g., \cite[V, (26)]{Stein70}. We set the heat semigroup and Bessel potential on the torus by using periodized kernels, e.g., \cite{RS16}:

\begin{align*}
T(t)f&=G_{t}^{\textrm{per}}*f,\quad G_{t}^{\textrm{per}}(\theta)
=\sum_{m\in \mathbb{Z}}G_t(\theta+2\pi m)
=\frac{1}{2\pi}\sum_{m\in \mathbb{Z}}e^{-t|m|^{2}}e^{im\theta},\quad t>0,\\
S(s)f&=F_{s}^{\textrm{per}}*f,\quad F_{s}^{\textrm{per}}(\theta)
=\sum_{m\in \mathbb{Z}}F_s(\theta+2\pi m)
=\frac{1}{2\pi}\sum_{m\in \mathbb{Z}}\langle m \rangle^{-s}e^{im\theta},\quad 0<s<1.
\end{align*}\\
The $L^{1}$-norm of the kernel $G_t^{\textrm{per}}(\theta)$ is one and $T(t)$ is bounded on $L^{\sigma}(\mathbb{T})$ for $1<\sigma<\infty$. Moreover, $T(t)$ is a $C_0$-semigroup on $L^{\sigma}(\mathbb{T})$ for $1<\sigma<\infty$ by the continuity on $L^{2}(\mathbb{T})$ and density of smooth functions on $L^{\sigma}(\mathbb{T})$. The kernel $F_{s}$ satisfies $|F_{s}(\theta)|\lesssim \theta^{-1+s}$ and $|F_{s}'(\theta)|\lesssim \theta^{-2+s}$ near the origin and the same estimates hold for the periodized kernel $F_{s}^{\textrm{per}}$ for all $\theta\in \mathbb{T}$. In particular, $F_{s}^{\textrm{per}}(\theta)$ is $L^{\sigma}$ integrable on $\mathbb{T}$ for $1<\sigma <1/(1-s)$. The operator $S(s)$ and its inverse are expressed as 

\begin{align*}
S(s)f=F_{s}^{\textrm{per}}*f=2\pi \sum_{m\in \mathbb{Z}}\hat{F}^{\textrm{per}}_{s}(m)\hat{f}_me^{im\theta}=\sum_{m\in \mathbb{Z}}<m>^{s}\hat{f}_me^{im\theta}=(<m>^{s}\hat{f}_m)^{\lor},
\end{align*}\\
and $S(s)^{-1}f=S(-s)f=(<m>^{-s}\hat{f}_m)^{\lor}$. The kernel estimates for $F_{s}^{\textrm{per}}(\theta)$ yield the H\"older estimate

\begin{align}
||S(s)g||_{C^{\gamma}}\lesssim ||g||_{L^{p}},\quad \gamma=s-\frac{1}{p}>0,\label{eq:S4}
\end{align}\\
and the embedding into the H\"older space \eqref{eq:S3} by choosing $g=S(-s)f$ and using $||S(-s)f||_{L^{p}}=||f||_{H^{s,p}}$. We show the compact embedding \eqref{eq:S2}.

\vspace{5pt}

\begin{prop}
Let $0<s<1$ and $1<p<\infty$ satisfy $1/q=1/p-s>0$. Then, the estimate   

\begin{align}
||T(t)f-f||_{L^{r}}
\leq \left\|T(t)F_{s}^{\textrm{per}}-F_{s}^{\textrm{per}}\right\|_{L^{\sigma}}||f||_{H^{s,p}},\quad f\in H^{s,p}(\mathbb{T}),   \label{eq: Convergence}
\end{align}\\ 
holds for $r<q$ and $\sigma$ satisfying $1/r=1/\sigma+1/p-1$.
\end{prop}

\vspace{5pt}

\begin{proof}
We express $T(t)f-f$ by the convolution 

\begin{align*}
T(t)f-f
= \left(T(t)S(s)-S(s)\right)S(-s)f
=(T(t)F_{s}^{\textrm{per}}-F_{s}^{\textrm{per}})*(S(-s)f),
\end{align*}\\
apply Young's convolution inequality for $r<q$ and $\sigma$ satisfying $1/r=1/\sigma+1/p-1$,

\begin{align*}
||T(t)f-f||_{L^{r}}
\leq \left\|T(t)F_{s}^{\textrm{per}}-F_{s}^{\textrm{per}}\right\|_{L^{\sigma}}||S(-s)f||_{L^{p}}
= \left\|T(t)F_{s}^{\textrm{per}}-F_{s}^{\textrm{per}}\right\|_{L^{\sigma}}||f||_{H^{s,p}}.
\end{align*}\\
By $s=1/p-1/q>1/p-1/r=1-1/\sigma$ and $F_{s}^{\textrm{per}}\in L^{\sigma}(\mathbb{T})$ for $1<\sigma <1/(1-s)$, the right-hand side is finite.
\end{proof}

\vspace{5pt}

\begin{proof}[Proof of Theorem \ref{t:Sobolev}]
For an arbitrary bounded sequence $\{f_n\}\subset H^{s,p}(\mathbb{T})$, we set $f^{N}_{n}=T(1/N)f_n$ with integers $N\geq 1$. By choosing a subsequence for each $N$, the sequence $\{f_n^{N}\}$ converges in $L^{r}(\mathbb{T})$. By a diagonal argument, we take a subsequence such that $\{f_n^{N}\}$ converges for any $N\geq 1$. By applying the inequality \eqref{eq: Convergence}, 

\begin{align*}
||f_n-f_l||_{L^{r}}
&\leq  \left\|f_n-f_n^{N}\right\|_{L^{r}}+\left\|f_n^{N}-f_l^{N} \right\|_{L^{r}}+\left\|f_l^{N}-f_l\right\|_{L^{r}}\\
&\lesssim \left\|T\left(\frac{1}{N}\right)F_{s}^{\textrm{per}}-F_{s}^{\textrm{per}}\right\|_{L^{\sigma}}\sup_{n\geq 1}||f_n||_{H^{s,p}}+\left\|f_n^{N}-f_l^{N} \right\|_{L^{r}}.
\end{align*}\\
Letting $n,l\to\infty$ and then $N\to\infty$ imply that $\{f_n\}$ is a convergent sequence in $L^{r}(\mathbb{T})$.
\end{proof}

\vspace{5pt}


\begin{thebibliography}{WLGSB23}

\bibitem[AA24]{AA24}
A.~Ai and O.-N. Avadanei.
\newblock Well-posedness for the surface quasi-geostrophic front equation.
\newblock {\em Nonlinearity}, 37(5):Paper No. 055022, 41, (2024).

\bibitem[ABC{\etalchar{+}}24]{ABCDGK}
D.~Albritton, E.~Bru{\'e}, M.~Colombo, C.~De Lellis, V.~Giri, M.~Janisch, and
  H.~Kwon.
\newblock {\em {I}nstability and nonuniqueness for the 2d {E}uler equations in
  vorticity form, after {M}. {V}ishik}, volume 219.
\newblock Annals of Mathematics Studies, Princeton University Press, 2024.

\bibitem[Abe24]{Abe11}
K.~Abe.
\newblock {E}xistence of homogeneous {E}uler flows of degree $-\alpha\notin
  [-2,0]$.
\newblock {\em Arch. Rational Mech. Anal.}, 248(30), (2024).

\bibitem[AGJ]{AGJ}
K.~Abe, D.~Ginsberg, and I.-J. Jeong.
\newblock {S}tationary self-similar profiles for the two-dimensional inviscid
  {B}oussinesq equations.
\newblock \href{https://arxiv.org/abs/2410.21765}{arXiv:2410.21765}.

\bibitem[Aly94]{Aly94}
J.~J. Aly.
\newblock Asymptotic formation of a current sheet in an indefinitely sheared
  force-free field: an analytical example.
\newblock {\em Astronomy and Astrophysics}, 288:p.1012--1020, (1994).

\bibitem[AS64]{AbS}
M.~Abramowitz and I.~A. Stegun.
\newblock {\em Handbook of mathematical functions with formulas, graphs, and
  mathematical tables}, volume No. 55 of {\em National Bureau of Standards
  Applied Mathematics Series}.
\newblock U. S. Government Printing Office, Washington, DC, 1964.

\bibitem[BC94]{BC94}
H.~Bahouri and J.-Y. Chemin.
\newblock \'{E}quations de transport relatives \'a\ des champs de vecteurs
  non-lipschitziens et m\'ecanique des fluides.
\newblock {\em Arch. Rational Mech. Anal.}, 127(2):159--181, (1994).

\bibitem[BL15a]{BL1}
J.~Bourgain and D.~Li.
\newblock Strong ill-posedness of the incompressible {E}uler equation in
  borderline {S}obolev spaces.
\newblock {\em Invent. Math.}, 201(1):97--157, (2015).

\bibitem[BL15b]{BL2}
J.~Bourgain and D.~Li.
\newblock Strong illposedness of the incompressible {E}uler equation in integer
  {$C^m$} spaces.
\newblock {\em Geom. Funct. Anal.}, 25(1):1--86, (2015).

\bibitem[BO13]{BO13}
\'A. B\'enyi and T.~Oh.
\newblock The {S}obolev inequality on the torus revisited.
\newblock {\em Publ. Math. Debrecen}, 83(3):359--374, (2013).

\bibitem[CC10]{CC10}
{\'A}.~Castro and D.~C{\'o}rdoba.
\newblock Infinite energy solutions of the surface quasi-geostrophic equation.
\newblock {\em Adv. Math.}, 225(4):1820--1829, (2010).

\bibitem[CCC{\etalchar{+}}12]{CCCGW}
D.~Chae, P.~Constantin, D.~C\'{o}rdoba, F.~Gancedo, and J.~Wu.
\newblock Generalized surface quasi-geostrophic equations with singular
  velocities.
\newblock {\em Comm. Pure Appl. Math.}, 65(8):1037--1066, (2012).

\bibitem[CCGS20]{CCG20}
{\'A}.~Castro, D.~C{\'o}rdoba, and J.~G{\'o}mez-Serrano.
\newblock Global smooth solutions for the inviscid sqg equation.
\newblock {\em Memoirs of the AMS}, 266(1292):89 pages, (2020).

\bibitem[CCW11]{CCW3}
D.~Chae, P.~Constantin, and J.~Wu.
\newblock Inviscid models generalizing the two-dimensional {E}uler and the
  surface quasi-geostrophic equations.
\newblock {\em Arch. Ration. Mech. Anal.}, 202(1):35--62, (2011).

\bibitem[CCZ21]{CCZ21}
\'{A}. Castro, D.~C\'{o}rdoba, and F.~Zheng.
\newblock The lifespan of classical solutions for the inviscid surface
  quasi-geostrophic equation.
\newblock {\em Ann. Inst. H. Poincar\'{e} C Anal. Non Lin\'{e}aire},
  38(5):1583--1603, (2021).

\bibitem[CFMR05]{CFMR}
D.~C\'{o}rdoba, M.~A. Fontelos, A.~M. Mancho, and J.~L. Rodrigo.
\newblock Evidence of singularities for a family of contour dynamics equations.
\newblock {\em Proc. Natl. Acad. Sci. USA}, 102(17):5949--5952, (2005).

\bibitem[CFMS]{CFMS}
A.~Castro, D.~Faraco, F.~Mengual, and M.~Solera.
\newblock Unstable vortices, sharp non-uniqueness with forcing, and global
  smooth solutions for the {S}{Q}{G} equation.
\newblock \href{https://arxiv.org/abs/2502.10274}{arXiv:2502.10274}.

\bibitem[Che21]{Chen21}
J.~Chen.
\newblock On the slightly perturbed {D}e {G}regorio model on {$S^1$}.
\newblock {\em Arch. Ration. Mech. Anal.}, 241(3):1843--1869, (2021).

\bibitem[Che25]{Chen25}
J.~Chen.
\newblock On the regularity of the de gregorio model for the 3d euler
  equations.
\newblock {\em J. Eur. Math. Soc.}, 27(no. 4):pp. 1619--1677, (2025).

\bibitem[CHH21]{CHH21}
J.~Chen, T.~Y. Hou, and De~Huang.
\newblock {O}n the finite time blowup of the {D}e {G}regorio model for the 3{D}
  {E}uler equations.
\newblock {\em {C}omm. {P}ure {A}ppl. {M}ath.}, 74(6):1282--1350, (2021).

\bibitem[Cho]{Cho}
M.~Cho.
\newblock Long-time behavior of logarithmic spiral vortex sheets with two
  branches.
\newblock \href{https://arxiv.org/abs/2312.02072}{arXiv:2312.02072}.

\bibitem[CJK25]{CJK25}
Y.-P. Choi, J.~Jung, and J.~Kim.
\newblock On well/ill-posedness for the generalized surface quasi-geostrophic
  equation in {H}\"{o}lder spaces.
\newblock {\em J. Differential Equations}, 443:Paper No. 113521, 36, (2025).

\bibitem[CJNO25]{CJNO}
D.~Chae, I.-J. Jeong, J.~Na, and S.-J. Oh.
\newblock Well-posedness for {O}hkitani model and long-time existence for
  surface quasi-geostrophic equations.
\newblock {\em Comm. Math. Phys.}, 406(4):Paper No. 75, 25, (2025).

\bibitem[CJO]{CJO1}
D.~Chae, I.-J. Jeong, and S.-J. Oh.
\newblock Illposedness via degenerate dispersion for generalized surface
  quasi-geostrophic equations with singular velocities.
\newblock \href{https://arxiv.org/abs/2308.02120}{arXiv:2308.02120}.

\bibitem[CKO24a]{CKO24b}
T.~Cie\'{s}lak, P.~Kokocki, and W.~S. O\.{z}a\'{n}ski.
\newblock Linear instability of symmetric logarithmic spiral vortex sheets.
\newblock {\em J. Math. Fluid Mech.}, 26(2):Paper No. 21, 27, (2024).

\bibitem[CKO24b]{CKO24a}
T.~Cie\'{s}lak, P.~Kokocki, and W.~S. O\.{z}a\'{n}ski.
\newblock Well-posedness of logarithmic spiral vortex sheets.
\newblock {\em J. Differential Equations}, 389:508--539, (2024).

\bibitem[CKO25]{CKO25}
T.~Cie\'{s}lak, P.~Kokocki, and W.~S. O\.{z}a\'{n}ski.
\newblock Existence of nonsymmetric logarithmic spiral vortex sheet solutions
  to the 2{D} {E}uler equations.
\newblock {\em Ann. Sc. Norm. Super. Pisa Cl. Sci. (5)}, 26(1):301--340,
  (2025).

\bibitem[CLM85]{CLM85}
P.~Constantin, P.D. Lax, and A.~Majda.
\newblock A simple one-dimensional model for the three-dimensional vorticity
  equation.
\newblock {\em Commun. Pure Appl. Math.}, 38(6):715--724, (1985).

\bibitem[CMT94a]{CMT94}
P.~Constantin, A.~J. Majda, and E.~Tabak.
\newblock Formation of strong fronts in the 2-d quasigeostrophic thermal active
  scalar.
\newblock {\em Nonlinearity}, 7(6):1495--1533, (1994).

\bibitem[CMT94b]{CMT1}
P.~Constantin, A.~J. Majda, and E.~G. Tabak.
\newblock Singular front formation in a model for quasigeostrophic flow.
\newblock {\em Phys. Fluids}, 6(1):9--11, (1994).

\bibitem[CMZ22]{CMZ1}
D.~C\'{o}rdoba and L.~Mart\'{\i}nez-Zoroa.
\newblock Non existence and strong ill-posedness in {$C^k$} and {S}obolev
  spaces for {SQG}.
\newblock {\em Adv. Math.}, 407:Paper No. 108570, 74, (2022).

\bibitem[CMZ24]{CMZ2}
D.~C\'{o}rdoba and L.~Mart\'{\i}nez-Zoroa.
\newblock Non-existence and strong ill-posedness in {$C^{k,\beta}$} for the
  generalized surface quasi-geostrophic equation.
\newblock {\em Comm. Math. Phys.}, 405(7):Paper No. 170, 53, (2024).

\bibitem[CMZO24]{CMZO}
D.~C\'{o}rdoba, L.~Mart\'{\i}nez-Zoroa, and W.~S. O\.{z}a\'{n}ski.
\newblock Instantaneous gap loss of {S}obolev regularity for the 2{D}
  incompressible {E}uler equations.
\newblock {\em Duke Math. J.}, 173(10):1931--1971, (2024).

\bibitem[CMZO25]{CMZO2}
D.~C\'{o}rdoba, L.~Mart\'{\i}nez-Zoroa, and W.~S. O\.{z}a\'{n}ski.
\newblock Instantaneous continuous loss of regularity for the {S}{Q}{G}
  equation.
\newblock {\em Adv. Math.}, 481:110553, (2025).

\bibitem[CW12]{ChWu}
D.~Chae and J.~Wu.
\newblock Logarithmically regularized inviscid models in borderline {S}obolev
  spaces.
\newblock {\em J. Math. Phys.}, 53(11):115601, 15, (2012).

\bibitem[Dah24]{Dahne:highest-cusped-waves-fkdv}
J.~Dahne.
\newblock Highest cusped waves for the fractional {K}d{V} equations.
\newblock {\em J. Differential Equations}, 401:550--670, (2024).

\bibitem[DE23]{DE}
T.~D. Drivas and T.~M. Elgindi.
\newblock Singularity formation in the incompressible {E}uler equation in
  finite and infinite time.
\newblock {\em EMS Surv. Math. Sci.}, 10(no. 1):pp. 1--100, (2023).

\bibitem[Deb07]{Deb}
L.~Debnath.
\newblock {\em Integral Transforms and Their Applications}.
\newblock Texts in Applied Mathematics. Chapman and Hall/CRC, Boca Raton, FL,
  3rd edition, 2007.
\newblock Third Edition.

\bibitem[Den09]{Den09}
S.~A. Denisov.
\newblock Infinite superlinear growth of the gradient for the two-dimensional
  {E}uler equation.
\newblock {\em Discrete Contin. Dyn. Syst.}, 23(3):755--764, (2009).

\bibitem[Den15a]{Den15b}
S.~A. Denisov.
\newblock The centrally symmetric {$V$}-states for active scalar equations.
  {T}wo-dimensional {E}uler with cut-off.
\newblock {\em Comm. Math. Phys.}, 337(2):955--1009, (2015).

\bibitem[Den15b]{Den15a}
S.~A. Denisov.
\newblock The sharp corner formation in 2{D} {E}uler dynamics of patches:
  infinite double exponential rate of merging.
\newblock {\em Arch. Ration. Mech. Anal.}, 215(2):675--705, (2015).

\bibitem[DG90]{DeG}
S.~De~Gregorio.
\newblock On a one-dimensional model for the three-dimensional vorticity
  equation.
\newblock {\em J. Stat. Phys.}, 59((5--6)):1251--1263, (1990).

\bibitem[DGS23]{Dahne-GomezSerrano:highest-wave-burgershilbert}
J.~Dahne and J.~G\'{o}mez-Serrano.
\newblock Highest cusped waves for the {B}urgers-{H}ilbert equation.
\newblock {\em Arch. Ration. Mech. Anal.}, 247(5):Paper No. 74, 55, (2023).

\bibitem[DM25]{DoMe}
M.~Dolce and G.~Mescolinii.
\newblock Self-similar instability and forced nonuniqueness: an application to
  the 2{D} {E}uler equations.
\newblock {\em J. Lond. Math. Soc.}, 112:e70274, (2025).

\bibitem[EG19]{EG}
V.~Elling and M.~V. Gnann.
\newblock Variety of unsymmetric multibranched logarithmic vortex spirals.
\newblock {\em European J. Appl. Math.}, 30:23--38, (2019).

\bibitem[EGM21]{EGM21}
T.~M. Elgindi, T.-E. Ghoul, and N.~Masmoudi.
\newblock Stable self-similar blow-up for a family of nonlocal transport
  equations.
\newblock {\em Anal. PDE}, 14(3):891--908, (2021).

\bibitem[EH25]{EH}
T.~M. Elgindi and Y.~Huang.
\newblock Regular and singular steady states of 2d incompressible euler
  equations near the bahouri-chemin patch.
\newblock {\em Arch. Rational Mech. Anal.}, 249(1):2, (2025).

\bibitem[EJ]{ElJo}
T.~M. Elgindi and M.~J. Jo.
\newblock {C}usp {F}ormation in {V}ortex {P}atches.
\newblock \href{https://arxiv.org/abs/2504.02705}{arXiv:2504.02705}.

\bibitem[EJ17]{EJ17}
T.~M. Elgindi and I.-J. Jeong.
\newblock Ill-posedness for the {I}ncompressible {E}uler {E}quations in
  {C}ritical {S}obolev {S}paces.
\newblock {\em Ann. PDE}, 3(1):3:7, (2017).

\bibitem[EJ20a]{EJSVP2}
T.~M. Elgindi and I.-J. Jeong.
\newblock On singular vortex patches, {II}: long-time dynamics.
\newblock {\em Trans. Amer. Math. Soc.}, 373(9):6757--6775, (2020).

\bibitem[EJ20b]{EJ20c}
T.~M. Elgindi and I.-J. Jeong.
\newblock On the effects of advection and vortex stretching.
\newblock {\em Arch. Ration. Mech. Anal.}, 235:1763--1817, (2020).

\bibitem[EJ20c]{EJ20b}
T.~M. Elgindi and I.-J. Jeong.
\newblock Symmetries and critical phenomena in fluids.
\newblock {\em Comm. Pure Appl. Math.}, 73:257--316, (2020).

\bibitem[EJ23]{EJSVP1}
T.~M. Elgindi and I.-J. Jeong.
\newblock On singular vortex patches, {I}: {W}ell-posedness issues.
\newblock {\em Mem. Amer. Math. Soc.}, 283(1400):v+89, (2023).

\bibitem[Ell16]{Elling}
V.~Elling.
\newblock Self-similar 2d {E}uler solutions with mixed-sign vorticity.
\newblock {\em Comm. Math. Phys.}, 348:27--68, (2016).

\bibitem[EM20]{Elgindi20}
T.~M. Elgindi and N.~Masmoudi.
\newblock {$L^\infty$} ill-posedness for a class of equations arising in
  hydrodynamics.
\newblock {\em Arch. Ration. Mech. Anal.}, 235:1979--2025, (2020).

\bibitem[EMS]{EMS}
T.~M. Elgindi, R.~W. Murray, and A.~R. Said.
\newblock Wellposedness and singularity formation beyond the {Y}udovich class.
\newblock \href{https://arxiv.org/abs/2312.17610}{arXiv:2312.17610v1}.

\bibitem[EMS25]{Elgindi2022}
T.~M. Elgindi, R.~W. Murray, and A.~R. Said.
\newblock On the long-time behavior of scale-invariant solutions to the 2d
  {E}uler equation and applications.
\newblock {\em Ann. Sci. \'Ec. Norm. Sup\'er.}, 58(4):943--970, (2025).

\bibitem[FRRO24]{Ros24}
X.~Fern\'andez-Real and X.~Ros-Oton.
\newblock {\em Integro-differential elliptic equations}, volume 350 of {\em
  Progress in Mathematics}.
\newblock Birkh\"auser/Springer, Cham, [2024] \copyright 2024.

\bibitem[GGS24]{GGS24}
C.~Garc{\'\i}a and J.~G{\'o}mez-Serrano.
\newblock {S}elf-similar spirals for the generalized surface quasi-geostrophic
  equations.
\newblock {\em J. Eur. Math. Soc.}, (2024).

\bibitem[GNP22]{GNP22}
F.~Gancedo, H.~Q. Nguyen, and N.~Patel.
\newblock Well-posedness for {SQG} sharp fronts with unbounded curvature.
\newblock {\em Math. Models Methods Appl. Sci.}, 32(13):2551--2599, (2022).

\bibitem[GP21]{GP21}
F.~Gancedo and N.~Patel.
\newblock On the local existence and blow-up for generalized sqg patches.
\newblock {\em Ann. PDE}, 7(1):Paper No. 4, 63, (2021).

\bibitem[Gra08]{Grafakos}
L.~Grafakos.
\newblock {\em Classical {F}ourier analysis}, volume 249 of {\em Graduate Texts
  in Mathematics}.
\newblock Springer, New York, second edition, 2008.

\bibitem[HQWW24]{HQWW}
D.~Huang, X.~Qin, X.~Wang, and D.~Wei.
\newblock Self-similar finite-time blowups with smooth profiles of the
  generalized constantin--lax--majda model.
\newblock {\em Arch. Rational Mech. Anal.}, 248(22), (2024).

\bibitem[HTW23]{HTW}
D.~Huang, J.~Tong, and D.~Wei.
\newblock On self-similar finite-time blowups of the de gregorio model on the
  real line.
\newblock {\em Commun. Math. Phys.}, 402:2791--2829, (2023).

\bibitem[Jeo21]{Jeong21}
I.-J. Jeong.
\newblock Loss of regularity for the 2{D} euler equations.
\newblock {\em J. Math. Fluid Mech.}, 23(no. 4):Paper No. 95, 11 pp, (2021).

\bibitem[JJ23]{JeonJeong}
J.~Jeon and I.-J. Jeong.
\newblock On evolution of corner-like g{SQG} patches.
\newblock {\em J. Math. Fluid Mech.}, 25(2):Paper No. 35, 10, (2023).

\bibitem[JKY25]{JKY25}
I.-J. Jeong, J.~Kim, and Y.~Yao.
\newblock On well-posedness of {$\alpha$}-{SQG} equations in the half-plane.
\newblock {\em Trans. Amer. Math. Soc.}, 378(1):421--446, (2025).

\bibitem[JS24]{JS24}
I.-J. Jeong and A.~R. Said.
\newblock Logarithmic spirals in $2$d perfect fluids.
\newblock {\em J. \'{E}c. polytech. Math.}, 11:655--682, (2024).

\bibitem[JSS19]{JSS19}
H.~Jia, S.~Stewart, and V.~Sverak.
\newblock On the de gregorio modification of the constantin--lax--majda model.
\newblock {\em Archive for Rational Mechanics and Analysis}, 231(2):1269--1304,
  (2019).

\bibitem[KJ24]{KJ24}
J.~Kim and I.-J. Jeong.
\newblock Strong illposedness for {S}{Q}{G} in critical {S}obolev spaces.
\newblock {\em Analysis \& PDE}, 17(1):133--170, (2024).

\bibitem[KL25]{KL25}
A.~Kiselev and X.~Luo.
\newblock The {$\alpha$}-{SQG} patch problem is illposed in {$C^{2,\beta}$} and
  {$W^{2,p}$}.
\newblock {\em Comm. Pure Appl. Math.}, 78(4):742--820, (2025).

\bibitem[KRYZ16]{KRYZ}
A.~Kiselev, L.~Ryzhik, Y.~Yao, and A.~Zlato\v{s}.
\newblock Finite time singularity for the modified {S}{Q}{G} patch equation.
\newblock {\em Ann. of Math. (2)}, 184(3):909--948, (2016).

\bibitem[Kv14]{KS}
A.~Kiselev and V.~\v{S}ver\'{a}k.
\newblock Small scale creation for solutions of the incompressible
  two-dimensional euler equation.
\newblock {\em Ann. of Math. (2)}, 180(3):1205--1220, (2014).

\bibitem[Kwo21]{Kwon}
H.~Kwon.
\newblock Strong ill-posedness of logarithmically regularized 2{D} {E}uler
  equations in the borderline {S}obolev space.
\newblock {\em J. Funct. Anal.}, 280(7):108822, (2021).

\bibitem[KYZ17]{KYZ17}
A.~Kiselev, Y.~Yao, and A.~Zlato{\v{s}}.
\newblock Local regularity for the modified {SQG} patch equation.
\newblock {\em Comm. Pure Appl. Math.}, 70(7):1253--1315, (2017).

\bibitem[LBB94]{LB94}
D.~Lynden-Bell and C.~Boily.
\newblock Self-similar solutions up to flashpoint in highly wound
  magnetostatics.
\newblock {\em Mon. Not. R. Astron. Soc.}, 267:146--152, (1994).

\bibitem[Lev44]{Levenberg:Levenberg-Marquardt}
K.~Levenberg.
\newblock A method for the solution of certain non-linear problems in least
  squares.
\newblock {\em Quart. Appl. Math.}, 2:164--168, (1944).

\bibitem[LLR20]{LLR}
Z.~Lei, J.~Liu, and X.~Ren.
\newblock On the constantin--lax--majda model with convection.
\newblock {\em Communications in Mathematical Physics}, 375(1):765--783,
  (2020).

\bibitem[LSS21]{LSS}
P.~M. Lushnikov, D.~A. Silantyev, and M.~Siegel.
\newblock Collapse versus blow-up and global existence in the generalized
  constantin--lax--majda equation.
\newblock {\em J. Nonlinear Sci.}, 31(5):1--56, (2021).

\bibitem[Mar63]{Marquardt:Levenberg-Marquardt}
D.~W. Marquardt.
\newblock An algorithm for least-squares estimation of nonlinear parameters.
\newblock {\em J. Soc. Indust. Appl. Math.}, 11:431--441, (1963).

\bibitem[MB02]{MaB}
A.~J. Majda and A.~L. Bertozzi.
\newblock {\em Vorticity and incompressible flow}, volume~27 of {\em Cambridge
  Texts in Applied Mathematics}.
\newblock Cambridge University Press, Cambridge, 2002.

\bibitem[MTXX]{MTXX}
Q.~Miao, C.~Tan, L.~Xue, and Z.~Xue.
\newblock Local regularity and finite-time singularity for a class of
  generalized {SQG} patches on the half-plane.
\newblock \href{https://arxiv.org/abs/2410.19273}{arXiv:2410.19273}.

\bibitem[Ohk11]{OH1}
K.~Ohkitani.
\newblock Growth rate analysis of scalar gradients in generalized surface
  quasigeostrophic equations of ideal fluids.
\newblock {\em Phys. Rev. E (3)}, 83(3):036317, 8, (2011).

\bibitem[OSW08]{OSW08}
H.~Okamoto, T.~Sakajo, and M.~Wunsch.
\newblock {O}n a generalization of the {C}onstantin--{L}ax--{M}ajda equation.
\newblock {\em Nonlinearity}, 21(10):2447, (2008).

\bibitem[OSW14]{OSW14}
H.~Okamoto, T.~Sakajo, and M.~Wunsch.
\newblock Steady-states and traveling-wave solutions of the generalized
  {C}onstantin-{L}ax-{M}ajda equation.
\newblock {\em Discrete Contin. Dyn. Syst.}, 34(8):3155--3170, (2014).

\bibitem[PC]{miguel}
M.~M.~G. Pascual-Caballo.
\newblock Stationary radial homogeneous solutions for the inviscid {SQG}
  equation.
\newblock \href{https://arxiv.org/abs/2510.03108}{arXiv:2510.03108}.

\bibitem[RS16]{RS16}
L.~Roncal and P.~R. Stinga.
\newblock Fractional laplacian on the torus.
\newblock {\em Communications in Contemporary Mathematics}, 18(03):1550033,
  (2016).

\bibitem[Sco11]{Scott}
R~K. Scott.
\newblock A scenario for finite-time singularity in the quasigeostrophic model.
\newblock {\em Journal of Fluid Mechanics}, 687(11):492--502, 2011.

\bibitem[SD19]{SD19}
R.~K. Scott and D.~G. Dritschel.
\newblock Scale-invariant singularity of the surface quasigeostrophic patch.
\newblock {\em Journal of Fluid Mechanics}, 863:R2, (2019).

\bibitem[Shv18]{Shv}
R.~Shvydkoy.
\newblock Homogeneous solutions to the 3{D} {E}uler system.
\newblock {\em Trans. Amer. Math. Soc.}, 370:2517--2535, (2018).

\bibitem[Sil07]{Sil}
L.~Silvestre.
\newblock Regularity of the obstacle problem for a fractional power of the
  laplace operator.
\newblock {\em Commun. Pure Appl. Math.}, 60:67--112, (2007).

\bibitem[Ste70]{Stein70}
E.~M. Stein.
\newblock {\em Singular integrals and differentiability properties of
  functions}.
\newblock Princeton Mathematical Series, No. 30. Princeton University Press,
  Princeton, N.J., 1970.

\bibitem[SWZ25]{SWZ25}
F.~Shao, D.~Wei, and Z.~Zhang.
\newblock Self-similar algebraic spiral solution of 2-{D} incompressible
  {E}uler equations.
\newblock {\em Ann. PDE}, 11(1):Paper No. 13, 91, (2025).

\bibitem[Tri83]{Triebel1}
H.~Triebel.
\newblock {\em Theory of function spaces}, volume~78 of {\em Monographs in
  Mathematics}.
\newblock Birkh\"auser Verlag, Basel, 1983.

\bibitem[Visa]{Vishik18}
M.~Vishik.
\newblock {I}nstability and non-uniqueness in the {C}auchy problem for the
  {E}uler equations of an ideal incompressible fluid. {P}art {I}.
\newblock \href{https://arxiv.org/abs/1805.09426}{arXiv:1805.09426}.

\bibitem[Visb]{Vishik18b}
M.~Vishik.
\newblock {I}nstability and non-uniqueness in the {C}auchy problem for the
  {E}uler equations of an ideal incompressible fluid. {P}art {II}.
\newblock \href{https://arxiv.org/abs/1805.09440}{arXiv:1805.09440}.

\bibitem[Wil96]{Willem}
M.~Willem.
\newblock {\em Minimax theorems}, volume~24 of {\em Progress in Nonlinear
  Differential Equations and their Applications}.
\newblock Birkh\"{a}user Boston, Inc., Boston, MA, 1996.

\bibitem[WLGSB23]{WLGB}
Y.~Wang, C.-Y. Lai, J.~G\'{o}mez-Serrano, and T.~Buckmaster.
\newblock Asymptotic self-similar blow-up profile for three-dimensional
  axisymmetric {E}uler equations using neural networks.
\newblock {\em Phys. Rev. Lett.}, 130(24):Paper No. 244002, 6, (2023).

\bibitem[Zla25]{Zlatos23}
A.~Zlato\v{s}.
\newblock Local regularity and finite time singularity for the generalized
  {SQG} equation on the half-plane.
\newblock {\em Duke Math. J.}, 174(17):3493 -- 3533, (2025).

\end{thebibliography}

\newcommand{\etalchar}[1]{$^{#1}$}

\end{document}